\documentclass[11pt]{amsart}
\usepackage{lmodern}
\usepackage[T1]{fontenc}
\usepackage{microtype}
\usepackage{amssymb}
\usepackage{amsthm}
\usepackage{amscd}
\usepackage{cite}
\usepackage{refcount}
\usepackage{xpatch}
\usepackage[framemethod=tikz]{mdframed}
\usepackage{xcolor}

\newtheorem{theorem}{Theorem}[section]
\newtheorem{lemma}[theorem]{Lemma}

\newtheorem{cor}[theorem]{Corollary}

\theoremstyle{definition}
\newtheorem{definition}[theorem]{Definition}

\newtheorem{claim}{Claim}

 \newenvironment{claimproof}{\begin{proof}[Proof of claim]}{\end{proof}}

\newcounter{Htheorem}
\makeatletter
\g@addto@macro\lemma{\stepcounter{Htheorem}} 

\makeatother

\makeatletter

\def\dotminussym#1#2{%
  \setbox0=\hbox{$\m@th#1-$}%
  \kern.5\wd0%
  \hbox to 0pt{\hss\hbox{$\m@th#1-$}\hss}%
  \raise.6\ht0\hbox to 0pt{\hss$\m@th#1.$\hss}%
  \kern.5\wd0}

\newcommand{\fn}[1]{\widehat{\mathbf{#1}}}

\mathchardef\mhyphen="2D

\makeatletter
\xpatchcmd{\mdf@preenvsetting}
 {\mdf@envdepth >\tw@}
 {\mdf@envdepth >20}
 {}
 {}
\makeatother

\mdtheorem[
hidealllines=true,
leftline=true,
innertopmargin=0pt,
innerbottommargin=0pt,
linewidth=1pt,
linecolor=black,
innerrightmargin=0pt,
skipabove=0pt,
innertopmargin=-22pt,
]{quotationb}{}

\allowdisplaybreaks[2]

\begin{document}

\title[Towards Effective Theory of Measures]{Towards an Effective Theory of Absolutely Continuous Measures}
\author{Henry Towsner}
\date{\today}


\maketitle

\section{Introduction}

There has been a great deal of work extracting quantitative results from non-constructive theorems in analysis (see \cite{kohlenbach:MR2445721}, and for some recent examples,\cite{MR2943215,MR3210080,MR3175627,MR3278188}), often from fairly new results involving sophisticated techniques.  However even very basic results can turn out to be deeply non-constructive, and a library of quantitative versions of such results is a needed resource for extracting bounds from theorems which depend on them.

In this paper we consider the following innocuous looking theorem:
\begin{theorem}\label{thm:original}
  Let $(f_n)_n$ and $(g_p)_p$ be sequences of $L^1$ functions such that
  \begin{itemize}
  \item the sequences $(f_n)_n$ and $(g_p)_p$ converge weakly,
  \item all the functions $f_ng_p$ are $L^1$,
  \item for each fixed $n$, the sequence $(f_ng_p)_p$ converges weakly, and
  \item for each fixed $p$, the sequence $(f_ng_p)_n$ converges weakly.
  \end{itemize}
Then
\[\lim_n\lim_p\int f_ng_p\, d\mu=\lim_p\lim_n\int f_ng_p\, d\mu.\]
\end{theorem}
Replacing $\mu$ with the measure concentrating on $\sigma$, this immediately implies that for all sets $\sigma$,
\[\lim_n\lim_p\int_\sigma f_ng_p\, d\mu=\lim_p\lim_n\int_\sigma f_ng_p\, d\mu.\]

This is part of (or at least follows from) the standard development of $L^1$ functions, as considered in \cite{MR737004} for instance.  The proof, however, is surprisingly non-trival---a crucial step passes through the Radon-Nikodym derivative.  Our interest in this example is motivated by the fact that this turns out to be the crucial step in a theorem about Banach spaces; the application of the results in this paper to producing a constructive version of that theorem is given in \cite{towsner:banach}.  In this paper, our goal is to begin the project of creating a library of constructive versions of the basic theory of the $L^1$ spaces.

The main technique used to obtain the results in this paper is the functional (or ``Dialectica'') translation \cite{avigad:MR1640329}; in particular the variant known as the monotone functional interpretation \cite{MR1428007}.  We do not describe the process of using the functional interpretation to obtain these results here, but see \cite{towsner_worked} for more about the general method.  Section \ref{sec:regularity} is devoted to an $L^1$, one-dimensional analog of Szemer\'edi's regularity lemma which is particularly likely to be a useful tool in other applications involving $L^1$ spaces.  This regularity lemma is the constructive analog of the statement that an $L^1$ functions can be approximated by its level sets; the appearance of a regularity like statement is a reflection of the general connection between infinitary $\Pi_3$ statements and finitary regularity-like statements \cite{goldbring:_approx_logic_measure,tao08Norm}.

In this case, we are interested in how long it takes for the convergence to occur---that is, how big do $n$ and $p$ have to be for the two sides to be close to each other.  More precisely, since the actual rate of convergence may be both non-computable and non-uniform, we are interested in the \emph{metastable convergence} of these limits.

Metastable convergence was introduced in the context of ergodic theory in \cite{avigad:MR2550151,tao:MR2408398}.  Suppose $(r_n)_n$ is a sequence of real numbers with the property that $\lim_n r_n$ exists (for some fixed $\sigma$); that is, for each $E$, there is an $n$ so that for every $m\geq n$, $|r_n-r_m|<1/E$.  It is well known that the function mapping $E$ to the corresponding bound $n$ may be uncomputable, and (worse for our purposes) may be highly non-uniform.

Metastable convergence is a seemingly weaker property which addresses this: in its simplest form, we say the sequence $(r_n)_n$ is metastably convergent if for each $E$ and each function $\widehat{\mathbf{m}}:\mathbb{N}\rightarrow\mathbb{N}$ there exists an $n$ so that
\[|r_n-r_{\widehat{\mathbf{m}}(n)}|<1/E.\]
(In fact, in this example metastable convergence implies ordinary convergence, but we will not need this fact, and it will not hold for more complicated limits.)  However when a sequence is convergent, we can typically show metaconvergence with $n$ depending computably on the values of $E$ and $\widehat{\mathbf{m}}$.  Further, in a precise formal sense, metastable convergence captures all the computable content of the original result: any computation which could be proven to halt using the original convergence result can also be shown to halt using metastable convergence.  This is because metastable convergence is an instance of the functional interpretation \cite{kohlenbach:MR2445721,avigad:MR1640329}.

Abstract meta-theorems of the sort in \cite{MR2373327,MR2054493,MR1195271,MR1428007} say that, even though the proof of Theorem \ref{thm:original} goes through the highly non-constructive Radon-Nikodym theorem, it should be possible to extract from the proof explicit, computable, bounds on the metastable convergence, uniformly in computable bounds on the premises---that from bounds on the $L^1$ norms of the functions in question and the rates of metastable convergence of the sequences $(f_n)_n, (g_p)_p, (f_ng_p)_n$, and $(f_ng_p)_p$.  Because the resulting argument would be unreasonably complicated, we settle for a slightly weaker result where we make some additional uniformity assumptions.

In this case, because we are dealing with a double limit, the right notion of metastable convergence is more complicated.  Our main result, Theorem \ref{thm:control_interval_E2}, will have the form:
\begin{quote}
Suppose $(f_n)_n$ and $(g_p)_p$ are sequences of $L^1$ functions such that the functions $f_ng_p$ satisfy a convergence condition discussed below.  Then for every $\epsilon>0$, every $p$ and $n$, and all functions $\fn{k}$ and $\fn{r}$, there exist:
\begin{itemize}
\item Values $m\geq n$ and $q\geq p$, and
\item Functions $\fn{l}$ and $\fn{s}$,
\end{itemize}
such that, setting $k=\fn{k}(m,q,\fn{l},\fn{s})$ and $r=\fn{r}(m,q,\fn{l},\fn{s})$, we have $\fn{l}(k,r)\geq k$, $\fn{s}(k,r)\geq r$, and
\[\left|\int f_m g_{\fn{s}(k,r)} d\mu-\int f_{\fn{l}(k,r)}g_q d\mu\right|<\epsilon.\]
\end{quote}

In Section \ref{sec:bounds}, we illustrate the resulting bounds by calculating them explicitly in the simplest interesting case, where $p=n=0$, $\fn{l}(m,q,\fn{l},\fn{s})=q+1$, and $\fn{s}(m,q,\fn{l},\fn{s})=m+1$, which gives the statement:
\begin{quote}
Suppose $(f_n)_n$ and $(g_p)_p$ are sequences of $L^1$ functions satisfying a convergence condition discussed below and that each $f_ng_p$ is an $L^1$ function.  Then for every $\epsilon>0$ there exist $s> m$ and $l>q$ so that
\[\left|\int f_m g_s d\mu-\int f_lg_q d\mu\right|<\epsilon.\]
\end{quote}

\section{Absolutely Continuous Measures}

Rather than work with $L^1$ functions, it turns out to be more natural to work with the corresponding absolutely continuous measures.

\subsection{Measures}

We fix a Boolean algebra $\Sigma$ containing a largest element $\Omega$ and a smallest element $\emptyset$.  Because we are thinking of $\Sigma$ as an algebra of sets, we write $\cup$ and $\cap$ for the the lattice operations on $\Sigma$, and write $\sigma\subseteq\tau$ as an abbreviation for ``$\sigma\cup\tau=\tau$''.

\begin{definition}
  If $\nu:\Sigma\rightarrow\mathbb{R}$, we say $\nu$ is \emph{additive} if $\nu(\emptyset)=0$ and whenever $\sigma,\tau\in\Sigma$, $\nu(\sigma\cup \tau)=\nu(\sigma)+\nu(\tau)-\nu(\sigma\cap \tau)$.

We write $|\nu|$ for the function $|\nu|(\sigma)=|\nu(\sigma)|$.  Note that for a general additive $\nu$, $|\nu|$ need not be additive.

A \emph{partition} in $\Sigma$ is a finite set $\mathcal{A}\subseteq\Sigma$ such that the elements of $\mathcal{A}$ are pairwise disjoint  (We do not assume that $\bigcup\mathcal{A}=\Omega$.)   We define $\nu(\mathcal{A})=\sum_{\sigma\in\mathcal{A}}\nu(\sigma)$.  By abuse of notation we will write $\sigma$ for the partition $\{\sigma\}$.  

We write $\mathcal{A}\preceq\mathcal{B}$ ($\mathcal{B}$ \emph{refines} $\mathcal{A}$) if $\bigcup\mathcal{B}=\bigcup\mathcal{A}$ and for every $\sigma\in\mathcal{B}$ there is a $\sigma_{\mathcal{A}}\in\mathcal{A}$ with $\sigma\subseteq\sigma_{\mathcal{A}}$.  When $\mathcal{A}\preceq\mathcal{B}$ and $\sigma\in\mathcal{A}$, we write $\mathcal{B}_\sigma=\{\sigma'\in\mathcal{B}\mid \sigma'\subseteq\sigma\}$. 
Clearly $\sigma\preceq\mathcal{A}_\sigma$.  For any $\tau\in\mathcal{B}$ we write $\tau_{\mathcal{A}}$ for the unique $\sigma\in\mathcal{A}$ so that $\tau\subseteq\sigma$.  

We write $[\mathcal{A},\mathcal{B}]$ for the set of all partitions $\mathcal{C}$ with $\mathcal{A}\preceq\mathcal{C}\preceq\mathcal{B}$.
\end{definition}
To help keep the notation straight, note that $\sigma_{\mathcal{A}}$ is itself a set---the same type as $\sigma$---namely the element of $\mathcal{A}$ containing $\sigma$, while $\mathcal{B}_\tau$ is a partition---the same type as $\mathcal{B}$---namely a partition refining $\tau$.

Throughout this paper we work with a fixed additive function $\mu:\Sigma\rightarrow[0,1]$ such that $\mu(\Omega)=1$.

\begin{definition}
We write $\delta_\nu(\mathcal{A})$, the \emph{density} of $\nu$ on $\mathcal{A}$, for $\frac{\nu(\mathcal{A})}{\mu(\mathcal{A})}$.

We say $\nu:\Sigma\rightarrow\mathbb{R}$ is \emph{absolutely continuous} if for every $E$ there is a $D$ so that whenever $\mathcal{A}$ is a partition with $\mu(\mathcal{A})<1/D$, $|\nu|(\mathcal{A})<1/E$.  A \emph{modulus of continuity} for $\nu$ is a function $\omega_\nu:\mathbb{N}\rightarrow\mathbb{N}$ such that for every $E$ and every $\mathcal{A}$ with $\mu(\mathcal{A})<1/\omega_\nu(E)$, $|\nu|(\mathcal{A})<1/E$.
\end{definition}

Here, and throughout the paper, we will prefer to work with bounds given by natural numbers.  Thus, we write $1/E$ in place of $\epsilon$ and $1/D$ in place of $\delta$. 

In general, if $\nu$ is absolutely continuous, we write $\omega_\nu$ for some canonical modulus of continuity (if there is one).

We will use the letters $\rho$, $\lambda$, $\nu$, and $\mu$ exclusively to refer to additive functions.  

\begin{lemma}
  If $\mathcal{A}\preceq\mathcal{B}$ then $\delta_{|\nu|}(\mathcal{A})\leq\delta_{|\nu|}(\mathcal{B})$.
\end{lemma}
\begin{proof}
  Since
\[\delta_{|\nu|}(\mathcal{A})=\frac{1}{\mu(\mathcal{A})}\sum_{\sigma\in\mathcal{A}}\mu(\sigma)\delta_{|\nu|}(\sigma),\]
and
\[\delta_{|\nu|}(\mathcal{B})=\frac{1}{\mu(\mathcal{A})}\sum_{\sigma\in\mathcal{A}}\mu(\sigma)\delta_{|\nu|}(\mathcal{B}_\sigma),\]
it suffices to show that $\delta_{|\nu|}(\sigma)\leq\delta_{|\nu|}(\mathcal{B}_\sigma)$.
\[
  \delta_{|\nu|}(\sigma)
=\frac{|\nu(\sigma)|}{\mu(\sigma)}\\
=\frac{1}{\mu(\sigma)}|\sum_{\sigma'\in\mathcal{B}_\sigma}\nu(\sigma')|\\
\leq\frac{1}{\mu(\sigma)}\sum_{\sigma'\in\mathcal{B}_\sigma}|\nu(\sigma')|\\
=\delta_{|\nu|}(\mathcal{B}_\sigma).
\]
\end{proof}

\begin{definition}
  The $L^1$ norm, $||\nu||_{L^1}$, is $\sup_{\mathcal{A}}|\nu|(\mathcal{A})$.
\end{definition}

\begin{lemma}
If $\nu$ is absolutely continuous, $||\nu||_{L^1}$ is finite.
\end{lemma}
\begin{proof}
Apply absolute continuity with $E=1$.  Then there is a $D$ so that whenever $\mu(\mathcal{A})<1/D$, $|\nu|(\mathcal{A})<1$.  We claim that for any $\mathcal{B}$, $|\nu|(\mathcal{B})< 2D$.  Take any $\mathcal{B}$ and choose $\mathcal{B}_0\subseteq\mathcal{B}$ so that $\mu(\mathcal{B}_0)<1/D$ and $\mu(\mathcal{B}_0)$ is maximal among subsets of $\mathcal{B}$ with measure $<1/D$.  (Such a $\mathcal{B}_0$ exists because there are only finitely many subsets of $\mathcal{B}$.)  Choose $\mathcal{B}_1\subseteq\mathcal{B}\setminus\mathcal{B}_0$ similarly, and repeat until we have $\mathcal{B}_0,\ldots,\mathcal{B}_k$.  For $i<k$, we must have $1/2D\leq\mu(\mathcal{B}_i)<1/D$.  In particular, $k\leq 2D$.  Since $\mu(\mathcal{B}_i)<1/D$ for all $i$, $|\nu|(\mathcal{B}_i)<1$ for all $i$.  Since $|\nu|(\mathcal{B})=\sum_i|\nu|(\mathcal{B}_i)$, $|\nu|(\mathcal{B})< 2D$.

This holds for any $\mathcal{B}$, so $||\nu||_{L^1}< 2D$.
\end{proof}

This gives us an easily expressed bound on densities of large partitions:
\begin{lemma}\label{thm:L1_density_bound}
  If $\mu(\mathcal{A})\geq 1/D$ then $\delta_{|\nu|}(\mathcal{A})\leq D||\nu||_{L^1}$.
\end{lemma}
\begin{proof}
  For any $\mathcal{A}$ we have $|\nu|(\mathcal{A})\leq||\nu||_{L^1}$, and therefore $\delta_{|\nu|}(\mathcal{A})\leq D||\nu||_{L^1}$.
\end{proof}



\subsection{Products}

When $\rho$ and $\lambda$ are induced by integrals---that is, $\rho(\sigma)=\int_\sigma f\,d\mu$ and $\lambda(\sigma)=\int_\sigma g\,d\mu$---we can consider a product $(\rho\lambda)(\sigma)=\int_\sigma fg\, d\mu$.  Of course, since $f$ and $g$ need only be $L^1$ functions, the product may be infinite on some sets.  As a result, the relationship between the separate measures $\rho$ and $\lambda$ and the product $\rho\lambda$ is not trivial to compute.

We can define a local version of the product:
\begin{definition}
  If $\rho,\lambda$ are functions from $\Sigma$ to $\mathbb{R}$, we define $\rho\ast\lambda$ to be the function
\[(\rho\ast\lambda)(\sigma)=\frac{\rho(\sigma)\lambda(\sigma)}{\mu(\sigma)}.\]
\end{definition}
Note that $\rho\ast\lambda$ need not be additive or absolutely continuous.

Then
\[(\rho\lambda)(\sigma)=\lim_{\mathcal{A}\succeq\sigma}(\rho\ast\lambda)(\mathcal{A}).\]
Much of the complexity of the proof will come from our need to approximate $\rho\lambda$ using $\rho\ast\lambda$.

\section{Notation}

We will ultimately need a series of techical computational lemmas, which will involve a large number of interrelated numeric bounds.  In order to keep the values somewhat organized, we adopt the following notation.  Most of our theorems and definitions will have the general form 
\begin{quote}
For all data $E$, $n$, etc., there exist values $D$, $m$, etc., such that something happens.
\end{quote}
We adopt the convention that the given data in a statement will always use use subscript $\flat$, while the values shown to exist will always have subscript $\sharp$.  Thus the statement above would be written:
\begin{quote}
  For all data $E_\flat$, $n_\flat$, etc., there exist values $D_\sharp$, $m_\sharp$, etc., such that something happens.
\end{quote}
We also need to avoid notation conflicts when applying theorems.  We adopt the rule that all the data corresponding to a single application of a theorem or definition will share a subscript, which will take the place of the $\flat$ or $\sharp$ which was used in the original statement.  Thus, if some later theorem makes use of the statement above, it would say:
\begin{quote}
  We apply the statement to the case $E_0=\cdots$ and $n_0=\cdots$, and the statement guarantees the existence of values $D_0$ and $m_0$ such that...
\end{quote}

We also adopt the rule that functions are always written in bold with a hat, so a function whose output is $m_\flat$ would be written $\fn{m}_\flat$.  Functions whose output is itself a function have the same name with a capital letter, so $\fn{M}_\flat(\cdots)=\fn{m}_\flat$ and $\fn{m}_\flat(\cdots)=m_\flat$.

Because most of our lemmas involve a sequence of numeric values, we use the letters $n,m,k,l$ for the indices of such a sequence, with the convention that typically $n\leq m\leq k\leq l$ (these letters will typically have subscripts as well).  When we have two distinct sets of indices, we use $p\leq q\leq r\leq s$ for the other indices.  When a theorem is stated involving the values $n,m,k,l$, we will sometimes apply to values of the form $p,q,r,s$; when we do so, we will be consistent---$m$ in the original theorem will correspond to $q$ in the application, and so on.

We assume throughout that all functions are monotone \cite{MR1428007}---that is, if $n\leq m$ then $\fn{m}(n)\leq\fn{m}(m)$---and that $\fn{m}(m)\geq m$.  This assumption is harmless, since we could always specify our theorems to replace $\fn{m}$ with $\fn{m}'(m)=\max_{n\leq m}\fn{m}(n)$.

\section{Sequences}

\subsection{Convergence}

The metastable analog of weak convergence is:
\begin{definition}
We say $(\nu_n)_n$ is \emph{metastably weakly convergent} if for every $E_\flat,\widehat{\mathbf{m}}_\flat,n_\flat$, there is an $M_\sharp\geq n_\flat$ so that for every $\sigma$, there is an $m_\sharp\leq M_\sharp$ such that whenever $m,m'\in[m_\sharp,\widehat{\mathbf{m}}_\flat(m_\sharp)]$,  $|\nu_{m}(\sigma)-\nu_{m'}(\sigma)|<1/E_\flat$.
\end{definition}
This is slightly more complicated than the notion for sequences of real numbers because of the uniformity.  (We are also following our general notation for the complicated functions produced by the functional interpretation, which creates an excessive number of subscripts on a simple statement like this.)  Note that the precise amount of uniformity is important; if we replaced $\sigma$ in the definition with an arbitrary partition $\mathcal{A}$ we would actually have the appropriate analog of $L^1$-convergence instead.

If we want to consider partitions, we have the following statement, which is \emph{not} uniform in the size of the partition:
 \begin{lemma}\label{thm:quant_convergence_partition_weak}
   If $(\nu_n)$ is metastably weakly convergent then for every $E_\flat$, $\mathcal{B}_\flat$, $\widehat{\mathbf{m}}_\flat$, $n_\flat$ there is an $m_\sharp\geq n_\flat$ such that whenever $m,m'\in[m_\sharp,\widehat{\mathbf{m}}_\flat(m_\sharp)]$, for each $\sigma_\flat\in\mathcal{B}_\flat$, $|\nu_{m}-\nu_{m'}|(\sigma_\flat)<1/E_\flat$.
 \end{lemma}
 \begin{proof}
   By induction on $|\mathcal{B}_\flat|$.  When $|\mathcal{B}_\flat|=1$, this follows immediately from metastable weak convergence applied to $E_\flat, \widehat{\mathbf{m}}_\flat, n_\flat$.

   Suppose the claim holds for $\mathcal{B}_\flat$ and we have some $\sigma_0\not\in\mathcal{B}_\flat$.  Given any $m_0$, by metastable weak convergence applied to $E_\flat, \fn{m}_\flat, n_\flat$, there is some $m_{m_0}\geq m_0$ so that for all $m,m'\in[m_{m_0},\widehat{\mathbf{m}}_\flat(m_{m_0})]$, $|\nu_{m}-\nu_{m'}|(\sigma_0)<1/E_\flat$.  Define $\widehat{\mathbf{m}}_0(m_0)=\widehat{\mathbf{m}}_\flat(m_{m_0})$ and apply the inductive hypothesis to $E_\flat,\mathcal{B}_\flat,\widehat{\mathbf{m}}_0,n_\flat$.  We obtain $m_0\geq n_\flat$ so that for all $m,m'\in[m_0,\widehat{\mathbf{m}}_0(m_0)]$ and all $\sigma_\flat\in\mathcal{B}_\flat$, $|\nu_{m}-\nu_{m'}|(\sigma_\flat)<1/E$.

We set $m_\sharp=m_{m_0}\geq m_0$.  Then $[m_{m_0},\widehat{\mathbf{m}}_\flat(m_{m_0})]\subseteq[m_0,\widehat{\mathbf{m}}_0(m_0)]$, so $m_\sharp$ satisfies the claim.
 \end{proof}

There is a natural strengthening of metastable weak convergence:
\begin{definition}
  $(\nu_n)$ has \emph{bounded fluctuations} if for every $E_\flat$ there is a $V_\sharp$ so that for every $\fn{m}_\flat,n_\flat,\sigma$ there is an $m_\sharp\in[n_\flat, \fn{m}^{V_\sharp}_\flat(n_\flat)]$ such that whenever $m,m'\in[m_\sharp,\fn{m}_\flat(m_\sharp)]$, $|\nu_{m}(\sigma)-\nu_{m'}(\sigma)|<1/E_\flat$.
\end{definition}
Metastable weak convergence corresponds to the statement that a certain tree is well-founded (see \cite{gaspar}); having bounded fluctuations implies that the height of this tree is bounded by $\omega$.  

It will be convenient to be able to assume that $m_\sharp=\fn{m}^v_\flat(n_\flat)$ exactly for some $v$:
\begin{lemma}
  Suppose $(\nu_n)$ has bounded fluctuations.  Then for every $E_\flat$ there is a $V_\sharp$ so that for every $\fn{m}_\flat,n_\flat,\sigma$ there is a $v_\sharp\leq V_\sharp$ such that whenever $m,m'\in[\fn{m}_\flat^{v_\sharp}(n_\flat),\fn{m}_\flat^{v_\sharp+1}(n_\flat)]$, $|\nu_{m}(\sigma)-\nu_{m'}(\sigma)|<1/E_\flat$.
\end{lemma}
\begin{proof}
  Let $V_0$ be the bound for the bounded fluctuation of $(\nu_n)$; applying this to the function $\fn{m}_\flat^2$, for any $\sigma, n_\flat$ there is an $m_\sharp\in[n_\flat,\fn{m}^{2V_0}(n_\flat)]$ such that whenever $m,m'\in[m_\sharp,\fn{m}^2_\flat(m_\sharp)]$, $|\nu_{m}(\sigma)-\nu_{m'}(\sigma)|<1/E_\flat$.  Let $v_0< 2V_0$ be greatest such that $\fn{m}_\flat^{v_0}(n_\flat)< m_\sharp$.  Then $\fn{m}_\flat^{v_0+1}(n_\flat)\geq m_\sharp$, so $m_\sharp\leq\fn{m}_\flat^{v_0+1}(n_\flat)\leq\fn{m}_\flat(m_\sharp)$ and $\fn{m}_\flat^{v_0+2}(n_\flat)\leq\fn{m}^2_\flat(m_\sharp)$, so for any $m,m'\in[\fn{m}_\flat^{v_0+1}(n_\flat),\fn{m}_\flat^{v_0+2}(n_\flat)]\subseteq[m_\sharp,\fn{m}_\flat^2(m_\sharp)]$ we have $|\nu_{m_\sharp}(\sigma)-\nu_m(\sigma)|<1/E_\flat$ as desired.
\end{proof}

In this case we can get also get some uniform bounds on partitions if we are willing to accept a set of defective $\sigma$ of small measure:
\begin{lemma}\label{thm:quant_convergence_partition}
  If $(\nu_n)$ has bounded fluctuations then for every $E_\flat,D_\flat,\widehat{\mathbf{m}}_\flat,n_\flat,\mathcal{B}_\flat$ there is an $m_\sharp\geq n_\flat$ such that, taking 
\[\mathcal{B}=\{\sigma\in\mathcal{B}_\flat\mid \text{for every }m,m'\in[m_\sharp,\widehat{\mathbf{m}}_\flat(m_\sharp)]\text{, }|\nu_m-\nu_{m'}|(\sigma)< 1/E_\flat\},\]
we have $\mu(\mathcal{B})\geq (1-/D_\flat) \mu(\mathcal{B}_\flat) $.
\end{lemma}
\begin{proof}
Let $V_\sharp$ be the bound on the number of fluctuations when $2E_\flat$.  Given $n_\flat,\fn{m}_\flat,v$, let
\[\mathcal{E}(v,n_\flat,\fn{m}_\flat)=\{\sigma\in\mathcal{B}_\flat\mid \text{for some }m,m'\in[\fn{m}_\flat^v(n_\flat),\fn{m}_\flat^{v+1}(n_\flat)]\text{, }|\nu_m-\nu_{m'}|(\sigma)\geq 1/E_\flat\},\]
the ``exceptional'' $\sigma$.  We will show that $\mu(\mathcal{E}(v,n_\flat,\fn{m}_\flat))<1-\mu(\mathcal{B}_\flat)/D_\flat$ for some $v$.

By induction on $k$ we will show that, for any $\fn{m}_\flat$, there is a $v\leq V_\sharp^k$ so that $\mathcal{E}(v,n_\flat,\fn{m}_\flat)<(1-1/V_\sharp)^k\mu(\mathcal{B}_\flat)$.

When $k=1$, since $(\nu_n)$ has bounded fluctuations, for each $\sigma\in\mathcal{B}_\flat$ there is a $v_\sigma\leq V_\sharp$ so that for each $m,m'\in[\fn{m}_\flat^{v_\sigma}(n_\flat),\fn{m}_\flat^{v_\sigma+1}(n_\flat)]$, $|\nu_m-\nu_{m'}|(\sigma)< 1/E_\flat$---that is, $\sigma\not\in\mathcal{E}(v_\sigma,n_\flat,\fn{m}_\flat)$.  In particular, there must be some $v\leq V_\sharp$ such that the set of $\sigma$ with $v_\sigma=v$ has measure $\geq  \mu(\mathcal{B}_\flat)/V_\sharp$, so $\mathcal{E}(v,n_\flat,\fn{m}_\flat)<(1-1/V_\sharp)\mu(\mathcal{B}_\flat)$.

Suppose the claim holds for $k$.  We apply the inductive hypothesis to the function $\fn{m}_\flat^{V_\sharp}$, so there is some $v\leq V_\sharp^k$ so that $\mu(\mathcal{E}(v,n_\flat,\fn{m}_\flat^{V_\sharp}))<(1-1/V_\sharp)^k\mu(\mathcal{B}_\flat)$.  Then applying the $k=1$ case to $\fn{m}_\flat^{v\cdot V_\sharp}(n_\flat),\fn{m}_\flat$, $\mathcal{E}(v,n_\flat,\fn{m}_\flat^{V_\sharp})$, we obtain a $v'$ so that
\[\mu(\mathcal{E}(v',\fn{m}_\flat^{v\cdot V_\sharp}(n_\flat),\fn{m}_\flat))<(1-1/V_\sharp) (1-1/V_\sharp)^k\mu(\mathcal{B}_\flat).\]
Therefore
\[\mu(\mathcal{E}(v\cdot V_\sharp+v',n_\flat,\fn{m}_\flat))<(1-1/V_\sharp)^{k+1}\mu(\mathcal{B}_\flat).\]
Therefore $v\cdot V_\sharp+v'$ $v\cdot V_\sharp+v'$ satisfies the claim.

The lemma follows by taking $k=\lceil\frac{\ln (1/D_\flat)}{\ln (1-1/V_\sharp)}\rceil$.
\end{proof}

\subsection{Uniform Continuity}

The Vitali-Hahn-Saks Theorem says roughly that a weakly convergent sequence of additive functions $\nu_m$ is actually uniformly continuous---that is, for each $\epsilon>0$ there is a $\delta>0$ so that when $\mu(\sigma)<\delta$, $|\nu_m(\sigma)|<\epsilon$ for all $m$ simultaneously.  The metastable analog of uniform continuity is:
\begin{definition}
  We say a sequence of functions $(\nu_n)_n$ is \emph{metastably uniformly continuous} if for every $E_\flat,\widehat{\mathbf{m}}_\flat,n_\flat$ there exist $m_\sharp\geq n_\flat$ and $D_\sharp$ such that whenever $\mu(\sigma)<1/D_\sharp$ and $m\in[m_\sharp,\widehat{\mathbf{m}}_\flat(D_\sharp,m_\sharp)]$, $|\nu_m(\sigma)|<1/E_\flat$.
\end{definition}
Note that this immediately implies the same statement with uniformity over partitions:
\begin{lemma}
Let $(\nu_n)_n$ be metastably uniformly continuous.  Then for any $E_\flat, \widehat{\mathbf{m}}_\flat,n_\flat$ there are $m_\sharp\geq n_\flat$ and $D_\sharp$ such that whenever $\mu(\mathcal{A})<1/D_\sharp$ and $m\in[m_\sharp,\widehat{\mathbf{m}}_\flat(D_\sharp,m_\sharp)]$, $|\nu_m|(\mathcal{A})<1/E_\flat.$
\end{lemma}
\begin{proof}
  Given $E_\flat, \widehat{\mathbf{m}}_\flat,n_\flat$, apply metastable uniform continuity to $2E_\flat, \widehat{\mathbf{m}}_\flat, n_\flat$ to obtain $m_\sharp\geq n_\flat$ and $D_\sharp$.  Then for any $m\in[m_0,\widehat{\mathbf{m}}(D,m_0)]$ and any $\mathcal{A}$, we may decompose $\mathcal{A}=\mathcal{A}_+\cup\mathcal{A}_-$ where
\[\mathcal{A}_+=\{\sigma\in\mathcal{A}\mid\nu_m(\sigma)\geq 0\},\ \mathcal{A}_-=\{\sigma\in\mathcal{A}\mid\nu_m(\sigma)<0\}.\]
Then
\[|\nu_m|(\mathcal{A})=\nu_m(\mathcal{A}_+)-\nu_m(\mathcal{A}_-)=\nu_m(\bigcup\mathcal{A}_+)-\nu_m(\bigcup\mathcal{A}_-)<2/2E_\flat.\]
\end{proof}

We now give a quantitative version of Vitali-Hahn-Saks.

\begin{theorem}\label{thm:q_vhs}
Let $E_\flat, \widehat{\mathbf{m}}_\flat,n_\flat$ be given and let $(\nu_n)_n$ be a metastably weakly convergent sequence of additive functions with moduli of absolute continuity $\omega_{\nu_n}$.  Then there are $m_\sharp\geq n_\flat$ and $D_\sharp$ so that, for each $m\in[m_\sharp,\widehat{\mathbf{m}}_\flat(D_\sharp,m_\sharp)]$, whenever $\mu(\sigma)<1/D_\sharp$, $|\nu_m(\sigma)|<1/E_\flat$.
\end{theorem}
\begin{proof}
We assume that the moduli of absolute continuity are rapidly growing, specifically that $\omega_{\nu_{m+1}}(E)\geq 2\omega_{\nu_m}(E)$.  (This is without loss of generality, since we can always replace $\omega_{\nu}$ with a larger function.)

We define a function
\[\fn{m}(m_0)=\fn{m}_\flat(2\omega_{\nu_{m_0}}(16E_\flat),m_0).\]

We now define a sequence of values $m_i, \sigma_i$.  We always have $D_i=2\omega_{\nu_{m_i}}(16E_\flat)$.   We will always have $m_i<m_{i+1}$, and therefore for any $j<i$ we have $m_i\geq m_{j}+(j-i)$, and so $D_i\geq 2^{i-j}D_j$.

We set $m_0=n_\flat$ and $\sigma_0=\emptyset$.  Suppose $m_i, \sigma_i$ are given.  We suppose that there is some $m\in[m_i,\fn{m}(m_i)]$ and a $\sigma$ with $\mu(\sigma_i\bigtriangleup\sigma)<1/D_i$ so that $|\nu_{m_i}(\sigma)-\nu_m(\sigma)|\geq 1/4E_\flat$.  (If not, the process stops and we will be able to prove the theorem as described below.)  We define $m_{i+1}$ to be this value of $m$ and $\sigma_{i+1}=\sigma$.

Note that for any $j<i$,
\[\mu(\sigma_j\bigtriangleup\sigma_i)\leq\sum_{j'\in[j,i)}\mu(\sigma_{j'}\bigtriangleup\sigma_{j'+1})\leq\sum_{j'\in[j,i)}1/D_{j'}\leq\sum_{j'\in[j,i)}2^{j-j'}/D_j<2/D_j.\]

Suppose we construct $m_i$ for all $i$.  Now define a function $\fn{m}'(m')$ to be $m_{i+1}$ where $i$ is least so $m_i\geq m'$.  Let $M'$ be given by metastable weak convergence applied to $8E_\flat, \fn{m}', n_\flat$ and let $m'\leq M'$ be such that whenever $k,k'\in[m',\fn{m}'(m')]$, $|\nu_{k}(\sigma_{M'})-\nu_{k'}(\sigma_{M'})|<1/8E_\flat$.  In particular, since $m_i,m_{i+1}\in[m',\fn{m}'(m')]$, $|\nu_{m_i}(\sigma_{M'})-\nu_{m_{i+1}}(\sigma_{M'})|<1/8E_\flat$.

As noted above, we have $\mu(\sigma^{i+1}\bigtriangleup\sigma^{M'})<2/D_{i+1}\leq 2/D_i$.  This means
\begin{align*}
|\nu_{m_i}(\sigma_{i+1})-\nu_{m_{i+1}}(\sigma_{i+1})|
&\leq|\nu_{m_i}(\sigma_{M'})-\nu_{m_{i+1}}(\sigma_{M'})|\\
&\ \ \ \ +
|\nu_{m_i}(\sigma_{i+1}\bigtriangleup\sigma_{M'})|+
|\nu_{m_{i+1}}(\sigma_{i+1}\bigtriangleup\sigma_{M'})|\\
&<1/8E_\flat+1/16E_\flat+1/16E_\flat\\
&=1/4E_\flat.
\end{align*}
But this contradicts the choice of $\sigma_{i+1}$.

So the process must eventually stop, and we find some $m_i,\sigma_i$ so that for every $m\in[m_i,\fn{m}(m_i)]$ and $\sigma$ with $\mu(\sigma_i\bigtriangleup\sigma)<1/D_i$ we have $|\nu_{m_i}(\sigma)-\nu_m(\sigma)|<1/4E_\flat$.  We take $D_\sharp=D_i\geq D_0$ and $m_\sharp=m_i$.  Then for any $m\in[m_\sharp,\fn{m}_\flat(D_\sharp,m_\sharp)]=[m_i,\fn{m}(m_i)]$ and any $\sigma$ with $\mu(\sigma)<1/D_\sharp$, 
  \begin{align*}
    |\nu_m(\sigma)|
&= |\nu_m(\sigma_i\cup\sigma)-\nu_m(\sigma_i\setminus\sigma)|\\
&<|\nu_{m_\sharp}(\sigma_i\cup\sigma)-\nu_{m_\sharp}(\sigma_i\setminus\sigma)|+1/2E_\flat\\
&\leq|\nu_{m_\sharp}(\sigma_i)-\nu_{m_\sharp}(\sigma_i)|+1/2E_\flat\\
&\ \ \ \ +|\nu_{m_\sharp}(\sigma_i)-\nu_{m_\sharp}(\sigma_i\cup\sigma)|\\
&\ \ \ \ +|\nu_{m_\sharp}(\sigma_i)-\nu_{m_\sharp}(\sigma_i\setminus\sigma)|\\
&< 1/E_\flat.
  \end{align*}
\end{proof}

\subsection{Double Sequences}

We need a similar notion for doubly indexed sequences---that is, given a collection of measures $(\rho_n\lambda_p)_{n,p}$, we need to be able to express uniform continuity.

\begin{definition}
  We say $(\rho_n\lambda_p)_{n,p}$ is $n_{/p}$-metastably uniformly continuous (``$n$ over $p$ metastably uniformly continuous'') if for every $E_\flat, \widehat{\mathbf{m}}_\flat$, and $\widehat{\mathbf{q}}_\flat$ there are $D_\sharp, m_\sharp, p_\sharp, \widehat{\mathbf{r}}_\sharp$ so that for if $\widehat{\mathbf{m}}_\flat(D_\sharp, m_\sharp, p_\sharp, \widehat{\mathbf{r}}_\sharp)\geq m_\sharp$ and $\widehat{\mathbf{q}}_\flat(D_\sharp, m_\sharp, p_\sharp, \widehat{\mathbf{r}}_\sharp)\geq p_\sharp$ then $\widehat{\mathbf{r}}_\sharp(\fn{m}_\flat(D_\sharp,m_\sharp,p_\sharp,\fn{r}_\sharp),\widehat{\mathbf{q}}_\flat(D_\sharp, m_\sharp, p_\sharp, \widehat{\mathbf{r}}_\sharp))\geq\widehat{\mathbf{q}}_\flat(D_\sharp, m_\sharp, p_\sharp, \widehat{\mathbf{r}}_\sharp)$ and for any $\sigma$ with $\mu(\sigma)<1/D_\sharp$,
\[|(\rho_{\fn{m}_\flat(D_\sharp,m_\sharp,p_\sharp,\fn{r}_\sharp),}\lambda_{\widehat{\mathbf{r}}_\sharp(\fn{m}_\flat(D_\sharp,m_\sharp,p_\sharp,\fn{r}_\sharp),\widehat{\mathbf{q}}_\flat(D_\sharp, m_\sharp, p_\sharp, \widehat{\mathbf{r}}_\sharp))})(\sigma)|<1/E_\flat.\]
\end{definition}
Of course there is also a dual version, $p_{/n}$-metastable uniform continuity, with the indices flipped.

Note that this is the metastable statement corresponding to the double limit $\lim_n\lim_p (\rho_n\lambda_p)(\sigma)$; the additional complexity is due to the higher quantifier complexity of a double limit.

In general we could prove that that ``$n_{/p}$-metastable weak convergence'' (which could be defined analogously) implies $n_{/p}$-metastable uniform continuity.  For our purpose we only need a special case which lets us avoid this notion.  The following lemma is the main step, which includes a stronger inductive hypothesis we need to complete the proof.

\begin{lemma}\label{thm:meta_bnd_1}
Suppose that
\begin{itemize}
\item $(\rho_1\lambda_p)_p$ has bounded fluctuations, and
\item for each $m$, $(\rho_m\lambda_r)_r$ is metastably uniformly continuous.
\end{itemize}

Then for any $E_\flat$, $\widehat{\mathbf{m}}_\flat$, $\widehat{\mathbf{q}}_\flat$, $n_\flat$, $p_\flat$, there are $D_\sharp, m_\sharp\geq n_\flat, q_\sharp\geq p_\flat, \widehat{\mathbf{r}}_\sharp$ so that setting
\begin{itemize}
\item $m_\flat=\fn{m}_\flat(D_\sharp,m_\sharp,q_\sharp,\widehat{\mathbf{r}}_\sharp)$,
\item $q_\flat=\fn{q}_\flat(D_\sharp,m_\sharp,q_\sharp,\widehat{\mathbf{r}}_\sharp)$, and
\item $r_\sharp=\fn{r}_\sharp(m_\flat,q_\flat)$,
\end{itemize}
if $m_\flat\geq m_\sharp$ and $q_\flat\geq q_\sharp$ then
\begin{itemize}
\item $r_\sharp\geq q_\flat$,
\item there is a $\sigma_0$ such that whenever $\mu(\sigma_0\bigtriangleup\sigma)<1/D_\sharp$,
\[|(\rho_{m_\sharp}\lambda_{r_\sharp})(\sigma)-(\rho_{m_\flat}\lambda_{r_\sharp})(\sigma)|<1/E_\flat,\text{ and }\]
\item whenever $\mu(\sigma)<2/D_\sharp$, $|(\rho_{m_\sharp}\lambda_{r_\sharp})(\sigma)|<1/4E_\flat$.
\end{itemize}
\end{lemma}
\begin{proof}
  We define functions $\widehat{\mathbf{r}}_{i,D,n,p}$ so that for any $m, q$ we have $\widehat{\mathbf{r}}_{i,D,n,p}(m,q)\geq\max\{p,q\}$, and for any $\sigma_0$ one of the following holds:
  \begin{itemize}
  \item There exist $D_\sharp$, $m_\sharp$, $q_\sharp$, $\widehat{\mathbf{r}}_\sharp$ satisfying the lemma,
  \item There is a $\sigma$ with $\mu(\sigma)<2/D$ such that $|\rho_n\lambda_{\fn{r}_{i,D,n,p}(m,q)}(\sigma)|\geq 1/4E_\flat$,
  \item Whenever $\mu(\sigma_0\bigtriangleup\sigma)<1/D$,
\[|(\rho_n\lambda_{\fn{r}_{i,D,n,p}(m,q)})(\sigma)-(\rho_{m}\lambda_{\fn{r}_{i,D,n,p}(m,q)})(\sigma)|<1/E_\flat,\]
  \item There is a sequence $n=k_0<\cdots<k_i$ and a $\sigma$ with $\mu(\sigma_0\bigtriangleup\sigma)<2/D$ such that for each $j<i$, 
\[|(\rho_{k_j}\lambda_{\fn{r}_{i,D,n,p}(m,q)})(\sigma)-(\rho_{k_{j+1}}\lambda_{\fn{r}_{i,D,n,p}(m,q)})(\sigma)|\geq 1/2E_\flat.\]
  \end{itemize}
For $i=0$ we take $\widehat{\mathbf{r}}_{0,D,n,p}(m,q)=\max\{p,q\}$ since the final clause is satisfied trivially.

Suppose we have defined $\widehat{\mathbf{r}}_{i,D,n,p}(m,q)$ for all $D,n,p,m,q,\sigma_0$.  We now define $\fn{r}_{i+1,D,n,p}(m,q)$ for some fixed values $D,n,p,m,q$.  We assume $m\geq n$ and $q\geq p$; if not, we replace $m$ with $n$ or $q$ with $p$ as necessary.  We define $\fn{r}^*(D^*,q^*)$ by
\[\fn{r}^*_0(D^*,q^*)=\fn{r}_{i,D^*,m,q^*}(\fn{m}_\flat(D^*,m,q^*,\fn{r}_{i,D^*,m,q^*}),\fn{q}_\flat(D^*,m,q^*,\fn{r}_{i,D^*,m,q^*}))\]
and $\fn{r}^*(D^*,q^*)=\fn{r}^*_0(\max\{D^*,2D\},q^*)$.

By the metastable uniform continuity of $(\rho_m\lambda_r)_r$, we obtain $D^*,q^*$ such that whenever $\mu(\sigma)<2/D^*$ and $q\in[q^*,\fn{r}^*(D^*,q^*)]$, $|\rho_m\lambda_q(\sigma)|<1/4E_\flat$.  Without loss of generality we may assume $D^*\geq 2D$.  Let $m'=\fn{m}_\flat(D^*,m,q^*,\fn{r}_{i,D^*,m,q^*})$, $q'=\fn{q}_\flat(D^*,m,q^*,\fn{r}_{i,D^*,m,q^*})$, and $r=\fn{r}^*(D^*,q^*)=\fn{r}_{i,D^*,m,q^*}(m',p')$.  We define $\fn{r}_{i+1,D,n,p}(m,q)=r$.

We now check that for every $\sigma_0$, one of the four properties holds.  If there is any $\sigma$ with $\mu(\sigma)<2/D$ and $|(\rho_n\lambda_r)(\sigma)|\geq 1/4E_\flat$ then the second case holds, so assume not.  Similarly, if for every $\sigma$ with $\mu(\sigma_0\bigtriangleup\sigma)<1/D$ we have $|(\rho_n\lambda_r)(\sigma)-(\rho_m\lambda_r)(\sigma)|<1/E_\flat$ then the third case holds, so assume not, and fix a counterexample $\sigma$.

We apply the inductive hypothesis to $\fn{r}_{i,D^*,m,q^*}(m',q')$ and $\sigma$, so one of the four cases above must hold.  If the first case holds, we are done, since it resolves the first case for $\fn{r}_{i+1,D,n,q}$ as well.  We have chosen $D^*,q^*$ to rule out the second case.  If the third case holds then $D^*,m,q^*,\fn{r}_{i,D^*,m,q^*}$ satisfies the lemma.

The remaining possibility is the fourth case: there is a sequence $m=k_1<\cdots<k_{i+1}$ and a $\sigma'$ with $\mu(\sigma\bigtriangleup\sigma')<1/D^*$ such that for each $0<j<i+1$, $|(\rho_{k_j}\lambda_r)(\sigma)-(\rho_{k_{j+1}}\lambda_r)(\sigma)|\geq 1/2E_\flat$.  We take $n=k_0$.  Since $\mu(\sigma\bigtriangleup\sigma')<1/D^*<2/D^*\leq 1/D$, $|(\rho_m\lambda_r)(\sigma\bigtriangleup\sigma')|<1/4E_\flat$ and $|(\rho_n\lambda_r)(\sigma\bigtriangleup\sigma')|<1/4E_\flat$, so
\begin{align*}
  1/E_\flat
&\leq|(\rho_n\lambda_r)(\sigma)-(\rho_m\lambda_r)(\sigma)|\\
&\leq|(\rho_n\lambda_r)(\sigma')-(\rho_m\lambda_r)(\sigma')|
+|(\rho_n\lambda_r)(\sigma\bigtriangleup\sigma')|
+|(\rho_m\lambda_r)(\sigma\bigtriangleup\sigma')|\\
&\leq|(\rho_n\lambda_r)(\sigma')-(\rho_m\lambda_r)(\sigma')|+1/2E_\flat,
\end{align*}
and so $|(\rho_n\lambda_r)(\sigma')-(\rho_m\lambda_r)(\sigma')|\geq 1/2E_\flat$.  Since $\mu(\sigma_0\bigtriangleup\sigma')\leq\mu(\sigma_0\bigtriangleup\sigma)+\mu(\sigma\bigtriangleup\sigma')\leq 1/D+1/D^*\leq 2/D$ as needed, we satisfy the fourth case.

This completes the construction of the functions $\fn{r}_{i,D,n,p}$ and shows they have the desired properties.

Now fix $B$ large enough by the bounded fluctuations of $(\rho_1\lambda_p)_p$ and consider the function
\[\fn{r}^*(D^*,q^*)=\fn{r}_{B,D^*,1,q^*}(\fn{m}_\flat(D^*,1,q^*,\fn{r}_{B,D^*,1,q^*}),\fn{q}_\flat(D^*,1,q^*,\fn{r}_{B,D^*,1,q^*})).\]
By the metastable uniform continuity of $(\rho_1\lambda_p)_p$ we obtain $D^*,q^*$ such that whenever $\mu(\sigma)<2/D^*$, $|\rho_1\lambda_{\fn{r}^*(D^*,q^*)}|(\sigma)<1/4E_\flat$.  Let $m'=\fn{m}_\flat(D^*,1,q^*,\fn{r}_{B,D^*,1,q^*})$ and $q'=\fn{q}_\flat(D^*,1,q^*,\fn{r}_{B,D^*,1,q^*})$ and consider $r=\fn{r}^*_{B,D^*,1,q^*}(m',q')$ with $\emptyset$.  One of the four cases must hold; if the first holds, we are done.  We have ruled out the second by choice of $D^*, q^*$.  If the third holds then $D^*, 1, q^*, \fn{r}_{B,D^*,1,q^*}$ satisfies the claim.  If the fourth holds then we have a sequence $k_0<\cdots<k_B$ and a $\sigma$ so that for each $j<B$, $|(\rho_{k_j}\lambda_r)(\sigma)-(\rho_{k_{j+1}}\lambda_r)(\sigma)|\geq 1/2E_\flat$.  But this violates the choice of $B$.
\end{proof}

\begin{lemma}\label{thm:meta_bnd_3}
Suppose that
\begin{itemize}
\item $(\rho_1\lambda_p)_p$ has bounded fluctuations, and
\item for each $m$, $(\rho_m\lambda_r)_r$ is metastably uniformly continuous.
\end{itemize}

Then $(\rho_n\lambda_p)_{n,p}$ is $n_{/p}$-metastably uniformly continuous.
\end{lemma}
\begin{proof}
Apply the previous lemma to $4E_\flat,\widehat{\mathbf{m}}_\flat,\widehat{\mathbf{p}}_\flat,0,0$ to obtain $D_\sharp,m_\sharp,q_\sharp,\widehat{\mathbf{r}}_\sharp$.  Choose $\sigma_0$ given by the second clause, let $m_\flat=\fn{m}_\flat(D_\sharp,m_\sharp,q_\sharp,\fn{r}_\sharp)$, $q_\flat=\fn{q}_\flat(D_\sharp,m_\sharp,q_\sharp,\fn{r}_\sharp)$, and $r_\sharp=\fn{r}_\sharp(m_\flat,q_\flat)$.  Then if $\mu(\sigma)<1/D_\sharp$,
\begin{align*}
  |(\rho_{m_\flat}\lambda_{r_\sharp})(\sigma)|
&=|(\rho_{m_\flat}\lambda_{r_\sharp})(\sigma_0\cup\sigma)-(\rho_{m_\flat}\lambda_{r_\sharp})(\sigma_0\setminus\sigma)|\\
&<|(\rho_{m_\sharp}\lambda_{r_\sharp})(\sigma_0\cup\sigma)-(\rho_{m_\sharp}\lambda_{r_\sharp})(\sigma_0\setminus\sigma)|+1/2E_\flat\\
&=|(\rho_{m_\sharp}\lambda_{r_\sharp})(\sigma)|+1/2E_\flat\\
&<1/E_\flat.
\end{align*}
\end{proof}

\section{Regularity Lemma}\label{sec:regularity}

The usual proof of our main theorem, involving actual $L^1$ functions, would use level sets.  In order to obtain an analog for absolutely continuous measures, we need approximate level sets.  These are given by a ``one-dimensional'' $L^1$ analog of the Szemer\'edi regularity lemma.  (One dimensional regularity lemmas show up in some expositions \cite{croot_Sz_notes} of the usual regularity lemma.)  Roughly, this will say that we can find pairs of partitions $\mathcal{B}\succeq\mathcal{A}$ such that for most $\sigma\in\mathcal{A}$ and most $\sigma'\in\mathcal{B}_\sigma$, $\delta_\nu(\sigma)$ is close to $\delta_\nu(\sigma')$, even though $\mathcal{B}$ is ``much finer'' than $\mathcal{A}$.  To make this precise we will need a number of definitions.

If we were working with $L^2$ bounded functions, the argument would be much simpler.  In order to deal with $L^1$ functions---equivalently, the absolutely continuous measures were are considering---we need to be able to ``cut-off'' sets of sufficiently high density.

\begin{definition}
  Given a partition $\mathcal{B}$, we define $\mathcal{B}_{\nu>K}=\{\sigma\in\mathcal{B}\mid|\delta_\nu(\sigma)|>K\}$ and $\mathcal{B}_{\nu\leq K}=\{\sigma\in\mathcal{B}\mid|\delta_\nu(\sigma)|\leq K\}$.
\end{definition}

Then $\mathcal{B}=\mathcal{B}_{\nu>K}\cup\mathcal{B}_{\nu\leq K}$, and when $K$ is large relative to $||\nu||_{L^1}$, we can be sure that $\mu(\mathcal{B}_{\nu>K})$ is small.

\begin{definition}
  By a \emph{function on partitions} we mean a function $\widehat{\mathbf{B}}$ such that for any $\mathcal{A}$, $\mathcal{A}\preceq\widehat{\mathbf{B}}(\mathcal{A})$.  

\end{definition}

\begin{definition}
  Let $\mathcal{B}_0\preceq\mathcal{B}$ be given.  We define
\[\mathfrak{D}_{E,\mathcal{B}_0,\nu}(\mathcal{B})=\{\sigma\in\mathcal{B}\mid \left|\delta_\nu(\sigma)-\delta_\nu(\sigma_{\mathcal{B}_0})\right|\geq 1/E.\}\]

\end{definition}
$\mathfrak{D}$ stands for ``difference'', since it is those elements of $\mathcal{B}$ on which the density $\delta_\nu$ has changed significantly.

Our goal is to prove the following theorem:
\begin{theorem}[One-dimensional $L^1$ Regularity]\label{thm:1D}
Let $\nu$, $\mathcal{A}_\flat$, $E_\flat$, $D_\flat$, and a function on partitions $\fn{B}_\flat$ be given.  Then there exists a $\mathcal{B}_\sharp\succeq\mathcal{A}_\flat$ such that for every $\mathcal{B}\in[\mathcal{B}_\sharp,\fn{B}_\flat(\mathcal{B}_\sharp)]$,
\[\mu(\mathfrak{D}_{E_\flat,\mathcal{B}_\sharp,\nu}(\mathcal{B}))<1/D_\flat.\]

\end{theorem}

By analogy with Szemer\'edi regularity, we expect the proof to proceed as follows: we define a notion of density $\theta(\mathcal{C})$ such that:
\begin{itemize}
\item For all partitions $\mathcal{C}$, $\theta(C)$ is non-negative and bounded by some fixed value $C$,
\item If $\mathcal{C}\preceq\mathcal{D}$ then $\theta(\mathcal{C})\leq\theta(\mathcal{D})$,
\item If $\mathcal{C}$ is not the desired $\mathcal{B}_\sharp$ then there exists a $\mathcal{C}'\succeq\mathcal{C}$ such that $\theta(\mathcal{C}')\geq\theta(\mathcal{C})+c$ where $c$ is a fixed constant.
\end{itemize}
Then failure to witness the theorem means we can increment $\theta$, and so within roughly $1/c$ steps we must find the desired witness.  (This method is known as the \emph{density} or \emph{energy increment method}, and is characteristic of finitary analogs of the proofs of $\Pi_3$ statements.)

If $\nu$ has bounded $L^2$ norm, the choice of density notion is standard:
\[\theta_{L^2}(\mathcal{C})=\sum_{\sigma\in\mathcal{C}}\mu(\sigma)\delta^2_{\nu}(\sigma).\]
It is easy to see that $\theta_{L^2}$ is bounded by the square of the $L^2$ norm of $\nu$.

However since $\nu$ need not have bounded $L^2$ norm, we have to ``cut-off'' this norm, making it linear when $\delta_{|\nu|}$ gets large enough.  By choosing the cut-off large enough, we can ensure that the portion where the cut-off occurs has small measure---say, measure at most $1/2D_\flat$---and is therefore negligible.  We choose
\[\theta_{L^1}(\mathcal{C})=\sum_{\sigma\in\mathcal{C}_{\nu\leq K}}\mu(\sigma)\delta^2_{\nu}(\sigma)+2K\sum_{\sigma\in\mathcal{C}_{\nu>K}}\mu(\sigma)\left|\delta_{\nu} (\sigma)\right|\]
where $K=\max\{2D_\flat||\nu||_{L^1},1\}$.

Unfortunately, we have now violated monotonicity under a minor but unavoidable circumstance: if $\mathcal{C}\preceq\mathcal{D}$ and $\sigma\in\mathcal{C}_{\nu>K}$, it could nonetheless be that $\mathcal{D}_\sigma\not\subseteq\mathcal{D}_{\nu>K}$.  The second, linear term in $\theta_{L^1}$ has some (necessary) leeway built into it---we multiply by $2K$, not just $K$---and that interferes with monotonicity.

We solve this by weakening the monotonicity requirement to only consider pairs $\mathcal{C}\preceq\mathcal{D}$ where $\mathcal{C}_{\nu>K}\subseteq\mathcal{D}$.  Given a $\mathcal{D}\succeq\mathcal{C}$ violating this condition, we can modify $\mathcal{D}$ on a set of small measure to satisfy this condition.

Given a function on partitions $\fn{B}$, we can think of $\fn{B}(\mathcal{A})$ as specifying, for each $\sigma\in\mathcal{A}$, a partition of $\sigma$, $\sigma=\bigcup\{\sigma\in\fn{B}(\mathcal{A})\mid \sigma_{\mathcal{A}}=\sigma\}$.  We modify $\fn{B}$ so that we only apply $\fn{B}$ to elements of $\mathcal{A}_{\nu\leq K}$.

\begin{definition}
Let $\fn{B}$ be a function on partitions.  We define $\fn{B}^K$ to be the function on partitions given by
\[\fn{B}^K(\mathcal{A})=\mathcal{A}_{\nu>K}\cup\{\sigma\in\fn{B}(\mathcal{A})\mid \sigma_{\mathcal{A}}\in\mathcal{A}_{\nu\leq K}\}.\]
\end{definition}
It is convenient that for any $\mathcal{A}$, $|\fn{B}^K(\mathcal{A})|\leq|\fn{B}(\mathcal{A})|$.

We now prove some basic properties about $\theta_{L^1}$.

\begin{lemma}
  For any $\mathcal{C}$, $0\leq \theta_{L^1}(\mathcal{C})\leq K^2+2K||\nu||_{L^1}$
\end{lemma}
\begin{proof}
  The lower bound is obvious.  For the upper bound,
  \begin{align*}
    \theta_{L^1}(\mathcal{C})
&=\sum_{\sigma\in\mathcal{C}_{\nu\leq K}}\mu(\sigma)\delta_{\nu}^2(\sigma)+2K\sum_{\sigma\in\mathcal{C}_{\nu>K}}\mu(\sigma)\left|\delta_{\nu} (\sigma)\right|\\
&\leq\sum_{\sigma\in\mathcal{C}_{\nu\leq K}}\mu(\sigma)K^2+2K\sum_{\sigma\in\mathcal{C}}|\nu|(\sigma)\\
&\leq K^2+2K||\nu||_{L^1}.
  \end{align*}
\end{proof}

\begin{lemma}
  If $\mathcal{C}\preceq\mathcal{D}$ and $\mathcal{C}_{\nu>K}\subseteq\mathcal{D}$ then $\theta_{L^1}(\mathcal{C})\leq\theta_{L^1}(\mathcal{D})$.
\end{lemma}
\begin{proof}
It suffices to show that for each $\sigma\in\mathcal{C}_{\nu\leq K}$, $\theta_{L^1}(\sigma)\leq\theta_{L^1}(\mathcal{D}_\sigma)$.  Let us write $\mathcal{D}_\leq$ for $(\mathcal{D}_\sigma)_{\nu\leq K}$ and $\mathcal{D}_>$ for $\mathcal{D}_\sigma\setminus\mathcal{D}_\leq$.  Then
  \begin{align*}
\theta_{L^1}(\sigma)
&=    \mu(\sigma)\delta_{\nu}^2(\sigma)\\
&\leq\mu(\sigma)\left(\frac{\sum_{\tau\in\mathcal{D}_\sigma}\mu(\tau)\delta_{\nu}(\tau)}{\mu(\sigma)}\right)^2\\
&=\frac{\left(\sum_{\tau\in\mathcal{D}_\sigma}\mu(\tau)\delta_{\nu}(\tau)\right)^2}{\mu(\sigma)}\\
&=\frac{\left(\sum_{\tau\in\mathcal{D}_{\leq}}\mu(\tau)\delta_{\nu}(\tau)\right)^2}{\mu(\sigma)}\\
&\ \ \ \ +\frac{\left(\sum_{\tau\in\mathcal{D}_{>}}\mu(\tau)\delta_{\nu}(\tau)\right) \left(2\sum_{\tau\in\mathcal{D}_{\leq}}\mu(\tau)\delta_{\nu}(\tau)+\sum_{\tau '\in\mathcal{D}_{>}}\mu(\tau ')\delta_{\nu}(\tau ')\right)}{\mu(\sigma)}\\
&= \frac{\left(\sum_{\tau\in\mathcal{D}_{\leq}}\sqrt{\mu(\tau)}\left[\sqrt{\mu(\tau)}\delta_{\nu}(\tau)\right]\right)^2}{\mu(\sigma)}\\
&\ \ \ \ +\left(\sum_{\tau\in\mathcal{D}_{>}}\mu(\tau)\delta_{\nu}(\tau)\right) \frac{\sum_{\tau\in\mathcal{D}_{\leq}}\mu(\tau)\delta_{\nu}(\tau)}{\mu(\sigma)}\\
&\ \ \ \ +\left(\sum_{\tau\in\mathcal{D}_{>}}\mu(\tau)\delta_{\nu}(\tau)\right) \frac{\sum_{\tau\in\mathcal{D}_{\leq}}\mu(\tau)\delta_{\nu}(\tau)+\sum_{\tau '\in\mathcal{D}_{>}}\mu(\tau ')\delta_{\nu}(\tau ')}{\mu(\sigma)}\\
&\leq  \frac{\sum_{\tau\in\mathcal{D}_{\leq}}\mu(\tau)}{\mu(\sigma)}\sum_{\tau\in\mathcal{D}_{\leq}}\mu(\tau)\delta_{\nu}^2(\tau)\\
&\ \ \ \ +\left(\sum_{\tau\in\mathcal{D}_{>}}\mu(\tau)\delta_{\nu}(\tau)\right)\frac{ K\sum_{\tau\in\mathcal{D}_{\leq}}\mu(\tau)}{\mu(\sigma)}\\
&\ \ \ \ +\left(\sum_{\tau\in\mathcal{D}_{>}}\mu(\tau)\delta_{\nu}(\tau)\right)\delta_{\nu}(\sigma)\\
&\leq \sum_{\tau\in\mathcal{D}_{\leq}}\mu(\tau)\delta_{\nu}^2(\tau)+2K\sum_{\tau\in\mathcal{D}_{>}}\mu(\tau)\left|\delta_{\nu}(\tau)\right|\\
&=\theta_{L^1}(\mathcal{D}_\sigma).
  \end{align*}
\end{proof}

\begin{lemma}
  Suppose $\left|\delta_{\nu}(\sigma)\right|\leq K$ and $\sigma\preceq\mathcal{D}$.  Let $\mathcal{D}^*=\mathfrak{D}_{E_\flat,\sigma,\nu}(\mathcal{D})\cap\mathcal{D}_{\nu\leq K}$.  Then
\[\theta_{L^1}(\mathcal{D})\geq\theta_{L^1}(\sigma)+\mu(\mathcal{D}^*)/E_\flat^2.\]
\end{lemma}
\begin{proof}
For notational simplicity, we consider the case where $\mathcal{D}\setminus\mathcal{D}^*$ is a singleton (possibly a singleton of measure $0$); let us write $\zeta$ for this element.  (The general case follows from combining the two cases below.)

For each $\tau\in\mathcal{D}$, write $\gamma_\tau=\delta_{\nu}(\tau)-\delta_{\nu}(\sigma)$.   Then since
 \[\mu(\sigma)\delta_{\nu}(\sigma)=\sum_{\tau\in\mathcal{D}}\mu(\tau)\delta_{\nu}(\tau),\]
 we have
 \[\sum_\tau\mu(\tau)\gamma_\tau=\sum_\tau\mu(\tau) \delta_{\nu}(\tau)-\mu(\tau)\delta_{\nu}(\sigma)=0.\]

First, suppose $\left|\delta_{\nu}(\zeta)\right|\leq K$, so $\mathcal{D}_{\nu>K}=\emptyset$.  
\begin{align*}
  \theta_{L^1}(\mathcal{D})
&=\sum_{\tau\in\mathcal{D}}\mu(\tau)\delta^2_{\nu}(\tau)\\
&=\sum_{\tau\in\mathcal{D}}\mu(\tau)(\delta_{\nu}(\sigma)+\gamma_\tau)^2\\
&=\sum_{\tau\in\mathcal{D}}\mu(\tau)\delta^2_{\nu}(\sigma)+2\gamma_\tau\mu(\tau)\delta_{\nu}(\sigma)+\mu(\tau)\gamma_\tau^2\\
&=\theta_{L^1}(\sigma)+\sum_{\tau\in\mathcal{D}}\mu(\tau)\gamma_\tau^2\\
&\geq\theta_{L^1}(\sigma)+\mu(\mathcal{D}^*)/E^2_\flat.
\end{align*}

On the other hand, suppose $\left|\delta_{\nu}(\zeta)\right|>K$.  Since $\left|\delta_{\nu}(\tau)\right|<K$ for $\tau\in\mathcal{D}$, $\delta_{\nu}(\zeta)$ has the same sign as $\delta_\nu(\sigma)$, so $\gamma_\zeta$ also has the same sign.  In particular, $\left|\delta_\nu(\sigma)+\gamma_\zeta\right|=\left|\delta_\nu(\sigma)|+|\gamma_\zeta\right|$.  Then we have
\begin{align*}
  \theta_{L^1}(\mathcal{D})
&=\sum_{\tau\in\mathcal{D}}\mu(\tau)\delta^2_{\nu}(\tau)+2K\mu(\zeta)\left|\delta_{\nu}(\zeta)\right|\\
&=\sum_{\tau\in\mathcal{D}^*}\mu(\tau)(\delta_{\nu}(\sigma)+\gamma_\tau)^2+2K\mu(\zeta)\left|\delta_{\nu}(\sigma)+\gamma_\zeta\right|\\
&\geq\sum_{\tau\in\mathcal{D}^*}(\mu(\tau)\delta^2_{\nu}(\sigma)+2\mu(\tau) \delta_{\nu}(\sigma)\gamma_\tau+\mu(\tau)\gamma_\tau^2)+2K\mu(\zeta)\left|\delta_{\nu}(\sigma)\right|+2K \mu(\zeta)\left|\gamma_\zeta\right|\\
&\geq\sum_{\tau\in\mathcal{D}^*}(\mu(\tau)\delta^2_{\nu}(\sigma)+\mu(\tau)\gamma_\tau^2)+2K\mu(\zeta)\left|\delta_{\nu}(\sigma)\right|+2\mu(\zeta)\left(K \left|\gamma_\zeta\right|-\gamma_\zeta\delta_{\nu}(\sigma)\right)\\
&\geq\sum_{\tau\in\mathcal{D}^*}(\mu(\tau)\delta^2_{\nu}(\sigma)+\mu(\tau)\gamma_\tau^2)+\mu(\zeta)\delta^2_{\nu}(\sigma)\\
&\geq\theta_{L^1}(\sigma)+\mu(\mathcal{D}^*)/E^2_\flat.
\end{align*}
\end{proof}

\begin{cor}
  If $\mathcal{C}\preceq\mathcal{D}$ and $\mathcal{C}_{\nu>K}\subseteq\mathcal{D}$ then
\[\theta_{L^1}(\mathcal{D})\geq\theta_{L^1}(\mathcal{C})+\frac{1}{E_\flat^2}\mu(\mathcal{D}_{\nu\leq K}\cap\mathfrak{D}_{E_\flat,\mathcal{C},\nu}(\mathcal{D})).\]
\end{cor}

\begin{cor}\label{thm:increment}
  If $\mathcal{C}\preceq\mathcal{D}$, $\mathcal{C}_{\nu>K}\subseteq\mathcal{D}$ and $\mu(\mathfrak{D}_{E_\flat,\mathcal{C},\nu}(\mathcal{D}))\geq 1/D_\flat$ then $\theta_{L^1}(\mathcal{D})\geq\theta_{L^1}(\mathcal{C})+\frac{1}{2D_\flat E_\flat^2}$.
\end{cor}
\begin{proof}
  Follows from the previous corollary using the fact that, by Lemma \ref{thm:L1_density_bound}, $\mu(\mathcal{D}_{\nu>K})<1/2D_\flat$.
\end{proof}

We can now prove the regularity lemma:
\begin{proof}[Proof of Theorem \ref{thm:1D}]
We assume $K=2D_\flat||\nu||_{L^1}$.  (In the case where $2D_\flat||\nu||_{L^1}<1$, we obtain slightly different bounds, but the argument is unchanged.)

  Let $\mathcal{A}_0=\mathcal{A}_\flat$.  Given $\mathcal{A}_i$, if there is any $\mathcal{B}\in[\mathcal{A}_i,\widehat{\mathbf{B}}^K_\flat(\mathcal{A}_i)]$ such that $\mu(\mathfrak{D}_{E_\flat,\mathcal{A}_i,\nu}(\mathcal{B}))\geq 1/2D_\flat$, take $\mathcal{A}_{i+1}$ to be such a $\mathcal{B}$.

By Corollary \ref{thm:increment}, $\theta_{L^1}(\mathcal{A}_{i+1})\geq\theta_{L^1}(\mathcal{A}_i)+\frac{1}{4D_\flat E_\flat^2}$.  Since $\theta_{L^1}(\mathcal{A}_i)\leq K^2+2K||\nu||_{L^1}=4D_\flat^2||\nu||_{L^1}^2+4D_\flat||\nu||_{L^1}^2$, there must be some $i\leq 16D_\flat^3 E_\flat^2||\nu||^2_{L^1}+16D_\flat^2 E_\flat^2||\nu||_{L^1}^2$ so that for every $\mathcal{B}\in[\mathcal{A}_i,\widehat{\mathbf{B}}^K_\flat(\mathcal{A}_i)]$, $\mu(\mathfrak{D}_{E_\flat,\mathcal{A}_i,\nu}(\mathcal{B}))< 1/2D_\flat$.

Suppose $\mathcal{B}\in[\mathcal{A}_i,\fn{B}_\flat(\mathcal{A}_i)]$, and let $\mathcal{B}'=(\mathcal{A}_i)_{\nu>K}\cup\{\sigma\in\mathcal{B}\mid \sigma_{\mathcal{A}_i}\in(\mathcal{A}_i)_{\nu\leq K}\}$.  Then $\mathcal{B}'\preceq\fn{B}_\flat^K(\mathcal{A}_i)$, so $\mu(\mathfrak{D}_{E_\flat,\mathcal{A}_i,\nu}(\mathcal{B}'))<1/2D_\flat$.  If $\sigma\in\mathfrak{D}_{E_\flat,\mathcal{A}_i,\nu}(\mathcal{B})$ then either $\sigma\in\mathcal{B}'$, and therefore $\sigma\in \mathfrak{D}_{E_\flat,\mathcal{A}_i,\nu}(\mathcal{B}')$, or $\sigma_{\mathcal{A}}\in(\mathcal{A}_i)_{\nu>K}$.  Therefore 
\[\mu(\mathfrak{D}_{E_\flat,\mathcal{A}_i,\nu}(\mathcal{B}))\leq \mu(\mathfrak{D}_{E_\flat,\mathcal{A}_i,\nu}(\mathcal{B}'))+\mu(\mathcal{A}_{\nu>K})<1/2D_\flat+1/2D_\flat.\]
\end{proof}

We need to strengthen this theorem to sequences of functions.  Since we're no longer able to fix $\fn{B}^K$ in advance (we don't know what $\nu$ to use), we need a modification.


\begin{theorem}\label{thm:1D_sequential}
  Let $(\nu_n)_n$ be a sequence with $||\nu_n||_{L^1}\leq B$ for all $n$.  Let $\mathcal{A}_\flat,E_\flat, D_\flat,\fn{m}_\flat$, and $\fn{B}_\flat$ be given.  Then there exists a $\mathcal{B}_\sharp\succeq\mathcal{A}_\flat$, an $n_\sharp$, and a $\fn{k}_\sharp$ so that whenever $\mathcal{B}_\sharp\preceq\mathcal{B}\preceq\mathcal{B}'\preceq\widehat{\mathbf{B}}_\flat(n_\sharp,\fn{k}_\sharp,\mathcal{B}_\sharp)$, setting
\[m_\flat=\fn{m}_\flat(n_\sharp,\fn{k}_\sharp,\mathcal{B}_\sharp,\mathcal{B},\mathcal{B}')\]
and $k_\sharp=\fn{k}_\sharp(m_\flat,\mathcal{B}')$, if $m_\flat\geq n_\sharp$ then we have $k_\sharp\geq m_\flat$ and
 $\mu(\mathfrak{D}_{E_\flat,\mathcal{B},k_\sharp}(\mathcal{B}'))<1/D_\flat$.
\end{theorem}
That $k_\sharp$ is independent of $\mathcal{B}$ is an incidental simplification because of the actual calculations involved.
\begin{proof}
As above, we set $K=4D_\flat B$ and define
\[\theta_k(\mathcal{C})=\sum_{\sigma\in\mathcal{C}_{\nu_k\leq K}}\mu(\sigma)\delta^2_{\nu_k}(\sigma)+2K\sum_{\sigma\in\mathcal{C}_{\nu_k>K}}\mu(\sigma)\left|\delta_{\nu_k} (\sigma)\right|.\]

The main step is the following claim:
\begin{claim}
  Let $m_0$, $\mathcal{A}_0\preceq\ldots\preceq\mathcal{A}_d$ and $i$ be given.  Then either:
  \begin{itemize}
  \item There are $n_\sharp$, $\widehat{\mathbf{k}}_\sharp$, and $\mathcal{B}_\sharp$ satisfying the theorem or
  \item There is a sequence of extensions $\mathcal{A}_d\preceq\mathcal{A}_{d+1}\preceq\cdots\preceq\mathcal{A}_{d+i}$ and an $m\geq m_0$ so that 
for all $j<i$ we have $\theta_m(\mathcal{A}_{d+j+1})\geq\theta_m(\mathcal{A}_{d+j})+1/2^5D_\flat E^2_\flat$.
  \end{itemize}
\end{claim}
\begin{claimproof}
  By induction on $i$.  When $i=0$, this is trivial.

  Suppose the claim holds for $i$; we show it for $i+1$.  Suppose we are given $\mathcal{A}_0,\ldots,\mathcal{A}_d$ and $m_0$.  Consider the function $\widehat{\mathbf{k}}_i(m,\mathcal{B}')$ given by applying the inductive hypothesis to $m$ and $\mathcal{A}_0,\ldots,\mathcal{A}_d,\mathcal{B}'$.  If $m_0$, $\mathcal{A}_d$, $\widehat{\mathbf{k}}_i$ satisfy the theorem (or we ever end up in the first case), we are done, so suppose not.  Then there are $\mathcal{B},\mathcal{B}'\in[\mathcal{A}_d,\widehat{\mathbf{B}}_\flat(m_0,\widehat{\mathbf{k}}_i,\mathcal{A}_d)]$ so that if $m=\fn{m}_\flat(m_0,\fn{k}_i,\mathcal{A}_d,\mathcal{B},\mathcal{B}')\geq m_0$ and $k=\fn{k}_i(m,\mathcal{B}')$, we have $\mu(\mathfrak{D}_{E_\flat,\mathcal{B},k}(\mathcal{B}'))\geq 1/D_\flat$.  This means that either $\mu(\mathfrak{D}_{2E_\flat,\mathcal{A}_d,k}(\mathcal{B}))\geq 1/2D_\flat$ or $\mu(\mathfrak{D}_{2E_\flat,\mathcal{A}_d,k}(\mathcal{B}'))\geq 1/2D_\flat$.

Let $\mathcal{B}_\star=(\mathcal{A}_d)_{\nu_k>K}\cup\{\sigma\in\mathcal{B}\mid\left|\delta_{\nu_k}(\sigma_{\mathcal{A}_d})\right|\leq K\}$ and $\mathcal{B}'_\star=(\mathcal{B}_\star)_{\nu_k>K}\cup\{\sigma\in\mathcal{B}'\mid\left|\delta_{\nu_k}(\sigma_{\mathcal{B}_\star})\right|\leq K\}$.  Then either $\mu(\mathfrak{D}_{2E_\flat,\mathcal{A}_d,k}(\mathcal{B}))\geq 1/4D_\flat$ or $\mu(\mathfrak{D}_{2E_\flat,\mathcal{A}_d,k}(\mathcal{B}'))\geq 1/4D_\flat$, so $\theta_m(\mathcal{B}'_\star)\geq\theta_m(\mathcal{A}_d)+1/2^5D_\flat E_\flat^2$.  We may take $\mathcal{A}_{d+1}$ to be $\mathcal{B}'_\star$.

By the definition of $\widehat{\mathbf{k}}_i$, either we obtain $\widehat{\mathbf{k}}_\sharp$ and $\mathcal{B}_\sharp$ satisfying the theorem, in which case we are done, or we find an extension $\mathcal{A}_{d+1}\preceq\mathcal{A}_{(d+1)+1}\preceq\cdots\preceq\mathcal{A}_{(d+1)+i}$ so that
\[\theta_k(\mathcal{A}_{(d+1)+j+1})\geq\theta_k(\mathcal{A}_{(d+1)+j})+1/2^5D_\flat E_\flat^2\]
for all $j<i$.  So $\mathcal{A}_{d+1}\preceq\cdots\preceq\mathcal{A}_{d+i+1}$ is the desired extension.
\end{claimproof}

Since $\theta_m(\mathcal{B})\leq K^2+2KB=2^4D_\flat^2 B^2+2^3D_\flat B^2$, applying the claim with $i=2^{10}D_\flat^3 E^2_\flat B^2+2^9D_\flat^2 E^2_\flat B^2$ and $m_0=0$ means the second case is impossible, so we must be in the first case, satisfying the theorem.
\end{proof}

Before going on, we need the following observation:
\begin{lemma}
  Suppose $|\nu-\nu'|(\mathcal{B})<1/DE$.  Then
\[\mu(\{\sigma\in\mathcal{B}\mid |\delta_\nu(\sigma)-\delta_{\nu'}(\sigma)|\geq 1/E\})<1/D.\]
\end{lemma}
\begin{proof}
  Suppose not, so setting $\mathcal{B}^*=\{\sigma\in\mathcal{B}\mid |\delta_\nu(\sigma)-\delta_{\nu'}(\sigma)|\geq 1/E\}$.  Then
\[|\nu-\nu'|(\mathcal{B})
=\sum_{\sigma\in\mathcal{B}}|\delta_\nu(\sigma)-\delta_{\nu'}(\sigma)|\mu(\sigma)
\geq\sum_{\sigma\in\mathcal{B}^*}|\delta_\nu(\sigma)-\delta_{\nu'}(\sigma)|\mu(\sigma)
\geq 1/DE.\]
\end{proof}

We will need a stronger form of Theorem \ref{thm:1D_sequential} in which we achieve regularity, not for a single $\nu_n$, but for all $\nu_n$ with $n$ in an interval.

\subsection{An Aside about Notation}
The following lemma is the first instance of a pattern we will need many times: we wish to apply several lemmas in a nested fashion, and to do this, we need to define a nested series of functions satisfying the premises of those theorems.  We will distinguish variables belonging to a given application of a theorem by subscripts.  The outermost theorem will be subscripted with $0$ (to indicate no dependencies).  For instance, suppose we have two theorems: Theorem A says 
\begin{quote}
  For all $\fn{m}_\flat$ there are $n_\sharp$ and $p_\sharp$ such that $\fn{m}_\flat(n_\sharp,p_\sharp)$ has a convenient property.
\end{quote}
and Theorem B says 
\begin{quote}
  For all $\fn{m}_\flat$ there is an $n_\sharp$ such that $\fn{m}_\flat(n_\sharp)$ has a desirable property.
\end{quote}

We wish to prove a theorem in which the function $\fn{m}_\flat$ to which we apply Theorem A is defined using Theorem B.  We would write this as follows:

\begin{proof}[Proof of a hypothetical theorem]
We define a function $\fn{m}_0$ so that we can apply Theorem A to it.
\begin{quotationb*}
  Suppose $n_0$ and $p_0$ are given.  We write $\dagger$ to abbreviate $n_0,p_0$.  We now define a function $\fn{m}_\dagger$ so that we can apply Theorem B to it.

  \begin{quotationb*}
    Suppose $n_\dagger$ is given.  Define $\fn{m}_\dagger(n_\dagger)=f(n_\dagger,n_0,p_0)$.
  \end{quotationb*}

  By Theorem B applied to $\fn{m}_\dagger$, we obtain a value $n_\dagger$ so that $\fn{m}_\dagger(n_\dagger)$ has a desirable property.  We now define $\fn{m}_0(n_0,p_0)=g(n_\dagger)$.
\end{quotationb*}
By Theorem A applied to $\fn{m}_0$, we obtain $n_0,p_0$ so that $\fn{m}_0(n_0,p_0)$ has a convenient property.  This allows us to complete the proof.
\end{proof}

In particular, note that the subscripts distinguish the variables relevant to Theorem A from those relevant to Theorem B, and indicate the dependencies ($n_\dagger$ depends on $n_0,p_0$, for instance), and the use of the bars on the left of the text to indicate the scope of the variables.

\subsection{The Strong Form of Regularity}
\begin{theorem}\label{thm:1D_sequential_interval}
     Let $(\nu_n)_n$ be a metastably weakly convergent sequence of functions with $||\nu_n||_{L^1}\leq B$ for all $n$.  Let $\mathcal{A}_\flat,E_\flat,D_\flat,\widehat{\mathbf{B}}_\flat,\fn{m}_\flat,\widehat{\mathbf{L}}_\flat$ be given.  Then there exists a $\mathcal{B}_\sharp\succeq\mathcal{A}_\flat$, an $n_\sharp$, and an $\widehat{\mathbf{k}}_\sharp$ so that, setting $m_\flat=\fn{m}_\flat(n_\sharp,\fn{k}_\sharp,\mathcal{B}_\sharp)$ and $\widehat{\mathbf{l}}_\flat=\widehat{\mathbf{L}}_\flat(n_\sharp,\widehat{\mathbf{k}}_\sharp,\mathcal{B}_\sharp)$, if $m_\flat\geq n_\sharp$ then whenever $\mathcal{B}_\sharp\preceq\mathcal{B}\preceq\mathcal{B}'\preceq\widehat{\mathbf{B}}_\flat(n_\sharp,\fn{k}_\sharp,\mathcal{B}_\sharp)$, we may set
\[k_\sharp=\fn{k}_\sharp(m_\flat,\widehat{\mathbf{l}}_\flat,\mathcal{B},\mathcal{B}')\]
and then $k_\sharp\geq m_\flat$ and for every $l\in[k_\sharp,\widehat{\mathbf{l}}_\flat(k_\sharp)]$, $\mu(\mathfrak{D}_{E_\flat,\mathcal{B},l}(\mathcal{B}'))<1/D_\flat$.
 \end{theorem}
 \begin{proof}
In order to apply Theorem \ref{thm:1D_sequential}, we prepare to define functions $\fn{m}_0(n_0,\fn{k}_0,\mathcal{B}_0,\mathcal{B},\mathcal{B}')$ and $\fn{B}_0(n_0,\fn{k}_0,\mathcal{B}_0)$.

\begin{quotationb*}
  Suppose $n_0$, $\fn{k}_0$ and $\mathcal{B}_0$ are given.  We abbreviate $n_0, \fn{k}_0,\mathcal{B}_0$ by $\dagger$.

In order to apply $\fn{L}_\flat$, we need to define a function $\fn{k}_{\dagger}(m,\fn{l}_\flat,\mathcal{B},\mathcal{B}')$.

\begin{quotationb*}
  On input $m,\fn{l}_\flat,\mathcal{B}\preceq\mathcal{B}'$ we may apply the metastable weak convergence of $(\nu_n)_n$ twice to find an $m_{\dagger,m,\fn{l}_\flat,\mathcal{B},\mathcal{B}'}\geq m$ so that for each
\[k,k'\in[m_{\dagger,m,\fn{l}_\flat,\mathcal{B},\mathcal{B}'},\fn{l}_\flat(\fn{k}_0(m_{\dagger,n,\fn{l}_\flat,\mathcal{B},\mathcal{B}'},\mathcal{B}'))],\]
the set of $\sigma\in\mathcal{B}$ such that $|\nu_{k}-\nu_{k'}|(\sigma)\geq 1/3 E_\flat$ has measure $<1/3D_\flat$ and the set of $\sigma\in\mathcal{B}'$ such that $|\nu_k-\nu_{k'}|(\sigma)\geq 1/3 E_\flat$ has measure $<1/3D_\flat$.

We define
\begin{itemize}
\item $\fn{k}_{\dagger}(m,\fn{l}_\flat,\mathcal{B},\mathcal{B}')=\fn{k}_0(m_{\dagger,m,\fn{l}_\flat,\mathcal{B},\mathcal{B}'},\mathcal{B}')$.
\end{itemize}
\end{quotationb*}

We now define
\begin{itemize}
\item $\fn{m}_0(n_0,\fn{k}_0,\mathcal{B}_0,\mathcal{B},\mathcal{B}')=m_{\dagger,\fn{m}_\flat(n_0,\fn{k}_\dagger,\mathcal{B}_0),\fn{L}_\flat(n_0,\fn{k}_\dagger,\mathcal{B}_0),\mathcal{B},\mathcal{B}'}$,
\item $\fn{B}_0(n_0,\fn{k}_0,\mathcal{B}_0)=\fn{B}_\flat(n_0,\fn{m}_{\dagger},\mathcal{B}_0)$.
\end{itemize}
\end{quotationb*}

By Lemma \ref{thm:1D_sequential} applied to $\mathcal{A}_\flat,3E_\flat,3D_\flat,\fn{m}_0,\fn{B}_0$, we obtain $n_0$, $\fn{k}_0$ and $\mathcal{B}_0$.  We set $n_\sharp=n_0$, $\fn{k}_\sharp=\fn{k}_\dagger$ and $\mathcal{B}_\sharp=\mathcal{B}_0$.  Let $m_\flat=\fn{m}_\flat(n_\sharp,\fn{k}_\sharp,\mathcal{B}_\sharp)$ and $\fn{l}_\flat=\fn{L}_\flat(n_\sharp,\fn{k}_\sharp,\mathcal{B}_\sharp)$ and consider some $\mathcal{B}\preceq\mathcal{B}'\in[\mathcal{B}_\sharp,\fn{B}_\flat(n_\sharp,\fn{k}_\sharp,\mathcal{B}_\sharp)]=[\mathcal{B}_0,\fn{B}_0(n_0,\fn{k}_0,\mathcal{B}_0)]$.

Set $k_\sharp=\fn{k}_\sharp(m_\flat,\fn{l}_\flat,\mathcal{B},\mathcal{B}')$.  By choice of $\fn{k}_0$ and $\mathcal{B}_0$, $k_\sharp=\fn{k}_0(m_{\dagger,m_\flat,\fn{l}_\flat,\mathcal{B},\mathcal{B}'},\mathcal{B}')\geq m_{\dagger,m_\flat,\fn{l}_\flat,\mathcal{B},\mathcal{B}'}\geq m_\flat$ and $\mu(\mathfrak{D}_{3E_\flat,\mathcal{B},k_\sharp}(\mathcal{B}'))<1/3D_\flat$.

Consider any $k\in[k_\sharp,\fn{l}_\flat(k_\sharp)]\subseteq [m_{\dagger,m_\flat,\fn{l}_\flat,\mathcal{B},\mathcal{B}'},\fn{l}_\flat(\fn{k}_0(m_{\dagger,m_\flat,\fn{l}_\flat,\mathcal{B},\mathcal{B}'},\mathcal{B}'))]$.  Suppose $\sigma\not\in\mathfrak{D}_{3E_\flat,\mathcal{B},k_\sharp}(\mathcal{B}')$.  Then
\[|\nu_k(\sigma)-\nu_k(\sigma_{\mathcal{B}})|\leq|\nu_{k_\sharp}-\nu_k|(\sigma)+|\nu_{k_\sharp}-\nu_k|(\sigma_{\mathcal{B}})+1/3E_\flat<1/E_\flat.\]
Except for a set of $\sigma$ of measure less than $2/3D_\flat$, the first two values are each bounded by $1/3E_\flat$, so except on a set of measure less than $1/D_\flat$, $|\nu_k(\sigma)-\nu_k(\sigma_{\mathcal{B}})|<1/E_\flat$.  That is, $\mu(\mathfrak{D}_{E_\flat,\mathcal{B},k}(\mathcal{B}'))<1/D_\flat$.
\end{proof}

\section{Exchanges}

\subsection{$E$-Constant Partitions}

When $\rho(A)=\int_A f\,d\mu$, a natural and useful operation is to decompose $\Omega$ into approximate level sets---to fix $E$ and define $A_c=\{\omega\mid -1/2E\leq f(\omega)-c<1/2E\}$.  We could pick a collection of values $c$ so that these sets are pairwise disjoint.  This is an infinite partition, but if we know $||f||_{L^1}\leq B$ then we could choose a large number of values $c$ so that $\int_{\Omega\setminus\bigcup_{c\in S}A_c}|f|d\mu<1/E$ by taking $S=\{-K,-K+1/E,-K+2/E,\ldots,-K+2KE/E$ for big enough $K$.

In our setting, the analog of a partition into sets $A_c$ is the notion of an $E$-constant set:
\begin{definition}
  We say $\rho$ is \emph{$E$-constant} on $\mathcal{B}$ if for every $\sigma\in\mathcal{B}$ and every $\lambda$, $|(\rho\lambda)(\sigma)-(\rho\ast\lambda)(\sigma)|<|\lambda(\sigma)|/E$.
\end{definition}

The useful situation is to have a partition $\mathcal{B}'$ and a $\mathcal{B}^-\subseteq \mathcal{B}'$ so that some $\rho$ is $E$-constant on $\mathcal{B}'\setminus\mathcal{B}^-$---we may think of $\mathcal{B}'$ as having the form $\{A_c\mid c\in S\}$ for some large finite $S$, and $\mathcal{B}^-$ as being some further partition of $\{\omega\mid |f(\omega)|>K\}$.

\begin{lemma}\label{thm:refinement}
  Suppose $\rho$ is $E$-constant on $\mathcal{B}'\setminus\mathcal{B}^-$ and that $||\rho||_{L^1},||\lambda||_{L^1}\leq B$.  Then for any $\mathcal{B}\preceq\mathcal{B}'$ and any $C$,
\[\sum_{\sigma\in\mathcal{B}_{\lambda\leq C}}\left|(\rho\lambda)(\sigma)-(\rho\ast\lambda)(\sigma)\right|<\frac{2B}{E}+C|\rho|(\mathfrak{D}_{E,\mathcal{B},\lambda}(\mathcal{B}')\cup\mathcal{B}^-)+|\rho\lambda|(\mathfrak{D}_{E,\mathcal{B},\lambda}(\mathcal{B}')\cup\mathcal{B}^-).\]
\end{lemma}
\begin{proof}
  Since $\rho$ is $E$-constant on $\mathcal{B}'\setminus\mathcal{B}^-$, for each $\tau\in\mathcal{B}'\setminus\mathcal{B}^-$ we have
\[|(\rho\lambda)(\tau)-(\rho\ast\lambda)(\tau)|<\frac{1}{E}|\lambda(\tau)|=\frac{1}{E}|\lambda|(\tau).\]
Write $\mathcal{B}^+=\mathcal{B}'\setminus(\mathcal{B}^-\cup\mathfrak{D}_{E,\mathcal{B},\lambda}(\mathcal{B}'))$.  For any $\sigma\in\mathcal{B}$, let $\upsilon_\sigma=\sigma\setminus\bigcup\mathcal{B}^+$.  (We may think of $\upsilon_\sigma$ as the ``bad'' subset of $\sigma$.)  Note that if $\tau\in\mathcal{B}^+$ then $\tau\not\in\mathfrak{D}_{E,\mathcal{B},\lambda}(\mathcal{B}')$, so $|\delta_\lambda(\tau)-\delta_\lambda(\sigma)|<1/E$.  

Then for any $\sigma\in\mathcal{B}_{\leq C}$ we have
\begin{align*}
  (\rho\lambda)(\sigma)
&=\sum_{\tau\in\mathcal{B}'_\sigma}(\rho\lambda)(\tau)\\
&=(\rho\lambda)(\upsilon_\sigma)+\sum_{\tau\in\mathcal{B}_\sigma^+}(\rho\lambda)(\tau)\\
&=(\rho\lambda)(\upsilon_\sigma)+\sum_{\tau\in\mathcal{B}_\sigma^+}(\rho\ast\lambda)(\tau) +\gamma_\tau&|\gamma_\tau|<\frac{1}{E}|\lambda|(\tau)\\
&=(\rho\lambda)(\upsilon_\sigma)+\sum_{\tau\in\mathcal{B}_\sigma^+}\rho(\tau)\delta_\lambda(\tau)+\gamma_\tau &|\gamma_\tau|<\frac{1}{E}|\lambda|(\tau)\\
&=(\rho\lambda)(\upsilon_\sigma)+\sum_{\tau\in\mathcal{B}_\sigma^+}\rho(\tau)\delta_\lambda(\sigma)+\gamma'_\tau&|\gamma'_\tau|<\frac{1}{E}|\rho|(\tau)+\frac{1}{E}|\lambda|(\tau)\\
&=(\rho\lambda)(\upsilon_\sigma)+\rho(\sigma\setminus\upsilon_\sigma)\delta_\lambda(\sigma)+\gamma'_\sigma&|\gamma'_\sigma|<\frac{1}{E}\sum_{\tau\in\mathcal{B}'_\sigma}(|\rho|(\tau)+|\lambda|(\tau))\\
&=(\rho\lambda)(\upsilon_\sigma)+\rho(\sigma)\delta_\lambda(\sigma)+\gamma''_\sigma&|\gamma''_\sigma|<C|\rho|(\upsilon_\sigma)+\frac{1}{E}\sum_{\tau\in\mathcal{B}'_\sigma}(|\rho|(\tau)+|\lambda|(\tau))\\
&=(\rho\lambda)(\upsilon_\sigma)+(\rho\ast\lambda)(\sigma)+\gamma''_\sigma&|\gamma''_\sigma|<C|\rho|(\upsilon_\sigma)+\frac{1}{E}\sum_{\tau\in\mathcal{B}'_\sigma}(|\rho|(\tau)+|\lambda|(\tau)).\\
\end{align*}
Summing over all $\sigma\in\mathcal{B}_{\lambda\leq C}$, we obtain
\[\sum_{\sigma\in\mathcal{B}_{\lambda\leq C}}\left|(\rho\lambda)(\sigma)-(\rho\ast\lambda)(\sigma)\right|<\frac{2B}{E}+C|\rho|(\mathfrak{D}_{E,\mathcal{B},\lambda}(\mathcal{B}')\cup\mathcal{B}^-)+|\rho\lambda|(\mathfrak{D}_{E,\mathcal{B},\lambda}(\mathcal{B}')\cup\mathcal{B}^-).\]
\end{proof}

\subsection{An Exchange of Limits}

We now come to a series of lemma constituting the main part of our argument.  We make the following assumptions:
\begin{itemize}
\item[$(\ast)_1$] Each $(\lambda_p)_p$ has bounded fluctuations with bound independent of $p$,
\item[$(\ast)_2$] $(\rho_n)_n$ has bounded fluctuations with bound independent of $n$,
\item[$(\ast)_3$] There is a fixed bound $B$ such that for each $n$, $||\rho_n||_{L^1}\leq B$ and for each $p$, $||\lambda_p||_{L^1}\leq B$,
\item[$(\ast)_4$] For any $E, D$, $\rho_n$ (resp. $\lambda_p$) and $\mathcal{B}$, there is a $\mathcal{B}'\succeq\mathcal{B}$ and a $\mathcal{B}''\subseteq\mathcal{B}'$ so that $\rho_n$ (resp. $\lambda_p$) is $E$-constant on $\mathcal{B}'\setminus\mathcal{B}''$ and $\mu(\mathcal{B}'')<1/D$.
\end{itemize}
This last assumption is of course true if $\rho_n$ is given by an $L^1$ function---intersect the elements of $\mathcal{B}$ with the level sets of the function.  We could drop this assumption, replacing it by uses of the regularity lemma above, at the cost of further complicating the proof.

We refer to these assumptions collectively as $(\ast)$.

For technical reasons, we need a variant of Lemma \ref{thm:1D_sequential_interval} which is essentially the result of combining two applications of it:
\begin{lemma}\label{thm:1D_sequential_interval_double}
Suppose $(*)$ holds.  Let $E_\flat, D_\flat, \widehat{\mathbf{B}}_\flat^0, \widehat{\mathbf{B}}_\flat^1, \fn{m}_\flat,\fn{L}_\flat, \fn{q}_\flat, \fn{S}_\flat$ be given.  There exists a $\mathcal{B}_\sharp\succeq\{\Omega\}$, $n_\sharp$ and $p_\sharp$, a $\fn{k}_\sharp$, and an $\fn{r}_\sharp$ so that, setting
\begin{itemize}
\item $\mathcal{B}_\flat^0=\fn{B}_\flat^0(n_\sharp,p_\sharp,\fn{k}_\sharp,\fn{r}_\sharp,\mathcal{B}_\sharp)$,
\item $\mathcal{B}_\flat^1=\fn{B}_\flat^1(n_\sharp,p_\sharp,\fn{k}_\sharp,\fn{r}_\sharp,\mathcal{B}_\sharp)$,
\item $m_\flat=\fn{m}_\flat(n_\sharp,p_\sharp,\fn{k}_\sharp,\fn{r}_\sharp,\mathcal{B}_\sharp)$,
\item $q_\flat=\fn{q}_\flat(n_\sharp,p_\sharp,\fn{k}_\sharp,\fn{r}_\sharp,\mathcal{B}_\sharp)$,
\item $\fn{l}_\flat=\fn{L}_\flat(n_\sharp,p_\sharp,\fn{k}_\sharp,\fn{r}_\sharp,\mathcal{B}_\sharp)$,
\item $\fn{s}_\flat=\fn{S}_\flat(n_\sharp,p_\sharp,\fn{k}_\sharp,\fn{r}_\sharp,\mathcal{B}_\sharp)$,
\item $k_\sharp=\fn{k}_\sharp(m_\flat,q_\flat,\fn{l}_\flat,\fn{s}_\flat,\mathcal{B}_\flat^0,\mathcal{B}_\flat^1)$,
\item $r_\sharp=\fn{r}_\sharp(m_\flat,q_\flat,\fn{l}_\flat,\fn{s}_\flat,\mathcal{B}_\flat^0,\mathcal{B}_\flat^1)$,
\end{itemize}
if $m_\flat\geq n_\sharp$ and $q_\flat\geq p_\sharp$, we have $k_\sharp\geq m_\flat$, $r_\sharp\geq q_\flat$, for every $l\in[k_\sharp,\fn{l}_\flat(k_\sharp,r_\sharp)]$ $\mu(\mathfrak{D}_{E_\flat,\mathcal{B}_\sharp,l}(\mathcal{B}_\flat^0))<1/D_\flat$ and for every $s\in[r_\sharp,\fn{s}_\flat(k_\sharp,r_\sharp)]$ $\mu(\mathfrak{D}_{E_\flat,\mathcal{B}_\sharp,s}(\mathcal{B}_\flat^1))<1/D_\flat$.
\end{lemma}
\begin{proof}
  We prepare for the first application of Lemma \ref{thm:1D_sequential_interval}.

  \begin{quotationb*}
    Let $\mathcal{B}_0, n_0, \fn{k}_0$ be given.   Write $\dagger=\mathcal{B}_0,n_0,\fn{k}_0$.  We now prepare for a second application of Lemma \ref{thm:1D_sequential_interval}.

    \begin{quotationb*}
      Let $\mathcal{B}_\dagger, p_\dagger, \fn{r}_\dagger$ be given.  Write $\star$ for $\dagger,\mathcal{B}_\dagger,p_\dagger,\fn{r}_\dagger$.

We first define some helper functions:
      \begin{itemize}
      \item $\fn{l}_{\fn{l},r}(k)=\fn{l}(k,r)$,
      \item $\fn{k}_{\star,r}(m,q,\fn{l},\fn{s},\mathcal{B}^0,\mathcal{B}^1) =\fn{k}_0(m,\fn{l}_{\fn{l},r},\mathcal{B}_\dagger,\mathcal{B}^0)$,
      \item $\fn{s}_{\star,\mathcal{B}^0,\mathcal{B}^1,m,q,\fn{l},\fn{s}}(r)=\fn{s}(\fn{k}_{\star,r}(m,q,\fn{l},\fn{s},\mathcal{B}^0,\mathcal{B}^1),r)$,
      \item $\fn{r}_\star(m,q,\fn{l},\fn{s},\mathcal{B}^0,\mathcal{B}^1)=\fn{r}_\dagger(q,\fn{s}_{\star,\mathcal{B}^0,\mathcal{B}^1,m,q,\fn{l},\fn{s}},\mathcal{B}_\dagger,\mathcal{B}^1)$,
      \item $\fn{k}_\star(m,q,\fn{l},\fn{s},\mathcal{B}^0,\mathcal{B}^1)=\fn{k}_{\star,\fn{r}_\star(m,q,\fn{l},\fn{s},\mathcal{B}^0,\mathcal{B}^1)}(m,q,\fn{l},\fn{s},\mathcal{B}^0,\mathcal{B}^1)$,
      \item $\fn{l}_\star=\fn{L}_\flat(n_0,p_\dagger,\fn{k}_\star,\fn{r}_\star,\mathcal{B}_\dagger)$,
      \item $\fn{s}_\star=\fn{S}_\flat(n_0,p_\dagger,\fn{k}_\star,\fn{r}_\star,\mathcal{B}_\dagger)$.
      \item $m_\star=\fn{m}_\flat(n_0,p_\dagger,\fn{k}_\star,\fn{r}_\star,\mathcal{B}_\dagger)$.
      \item $q_\star=\fn{q}_\flat(n_0,p_\dagger,\fn{k}_\star,\fn{r}_\star,\mathcal{B}_\dagger)$.
      \end{itemize}

      We can now define the functions needed for an application of Lemma \ref{thm:1D_sequential_interval}.
      \begin{itemize}
      \item $\fn{B}_\dagger(p_\dagger,\fn{r}_\dagger,\mathcal{B}_\dagger)=\fn{B}_\flat^1(n_0,p_\dagger,\fn{k}_\star,\fn{r}_\star,\mathcal{B}_\dagger)$,
      \item $\fn{q}_\dagger(p_\dagger,\fn{r}_\dagger, \mathcal{B}_\dagger)=\fn{q}_\flat(n_0,p_\dagger,\fn{k}_\star,\fn{r}_\star,\mathcal{B}_\dagger)$,
      \item $\fn{S}_\dagger(p_\dagger,\fn{q}_\dagger,\mathcal{B}_\dagger)=\fn{s}_{\star,\mathcal{B}^0_\star,\mathcal{B}^1_\star,m_\star,q_\star,\fn{l}_\star,\fn{s}_\star}$.
      \end{itemize}
    \end{quotationb*}
    By Lemma \ref{thm:1D_sequential_interval} applied to $\mathcal{B}_0, E_\flat, D_\flat, \fn{B}_\dagger$, $\fn{q}_\dagger$, $\fn{S}_\dagger$ we find $\mathcal{B}_\dagger, p_\dagger, \fn{r}_\dagger$ such that, setting:
    \begin{itemize}
    \item $q_\dagger=\fn{q}_\dagger(p_\dagger,\fn{r}_\dagger, \mathcal{B}_\dagger)$, and
    \item $\fn{s}_\dagger=\fn{S}_\dagger(p_\dagger, \fn{r}_\dagger,\mathcal{B}_\dagger)$,
    \end{itemize}
if $q_\dagger\geq p_\dagger$ then whenever $\mathcal{B}_\dagger\preceq\mathcal{B}\preceq\mathcal{B}'\preceq\fn{B}_\dagger(p_\dagger, \fn{r}_\dagger,\mathcal{B}_\dagger)$, we may set $r_\dagger=\fn{r}_\dagger(q_\dagger,\fn{s}_\dagger,\mathcal{B},\mathcal{B}')$ and have $r_\dagger\geq q_\dagger$ and for every $s\in[r_\dagger,\fn{s}_\dagger(r_\dagger)]$, we have $\mu(\mathfrak{D}_{E_\flat,\mathcal{B},s}(\mathcal{B}'))<1/D_\flat$.

Note that we have now defined values $\mathcal{B}_\dagger, p_\dagger,q_\dagger,\fn{r}_\dagger,\fn{s}_\dagger$, all depending on $\dagger$---that is, as functions of $n_0,\fn{k}_0,\mathcal{B}_0$---as well as functions $\fn{r}_\star,\fn{k}_\star$, $\fn{l}_\star,\fn{s}_\star$ which can be derived from these by the definitions above.

    We now set:
    \begin{itemize}
    \item $\fn{B}_0(n_0,\fn{k}_0,\mathcal{B}_0)=\fn{B}_\flat^0(n_0,p_\dagger,\fn{k}_\star,\fn{r}_\star,\mathcal{B}_\dagger)$,
    \item $\fn{m}_0(n_0,\fn{k}_0,\mathcal{B}_0)=\fn{m}_\flat(n_0,p_\dagger,\fn{k}_\star,\fn{r}_\star,\mathcal{B}_\dagger)$,
    \item $\fn{L}_0(n_0,\fn{k}_0,\mathcal{B}_0)=\fn{l}_{\fn{L}_\flat(n_0,p_\dagger,\fn{k}_\star,\fn{r}_\star,\mathcal{B}_\dagger),\fn{r}_\dagger(q_\dagger,\fn{s}_\dagger,\mathcal{B}_\dagger,\mathcal{B}^1_\star)}$.
    \end{itemize}
  \end{quotationb*}
By Lemma \ref{thm:1D_sequential_interval} applied to $\{\Omega\}, E_\flat, D_\flat, \fn{B}_0, \fn{m}_0, \fn{L}_0$, we find $\mathcal{B}_0, n_0, \fn{k}_0$ such that, setting
\begin{itemize}
\item $m_0=\fn{m}_0(n_0,\fn{k}_0,\mathcal{B}_0)$, and
\item $\fn{l}_0=\fn{L}_0(n_0,\fn{k}_0,\mathcal{B}_0)$,
\end{itemize}
if $m_0\geq n_0$ then whenever $\mathcal{B}_0\preceq\mathcal{B}\preceq\fn{B}_0(n_0,\fn{k}_0,\mathcal{B}_0)$ we may set $k_0=\fn{k}_0(m_0,\fn{l}_0,\mathcal{B},\mathcal{B}')$ and then $k_0\geq m_0$ and for every $l\in[k_0,\fn{l}_0(k_0)]$, $\mu(\mathfrak{D}_{E_\flat,\mathcal{B},l}(\mathcal{B}'))<1/D_\flat$.

We may now set:
\begin{itemize}
\item $\mathcal{B}_\sharp=\mathcal{B}_\dagger$,
\item $n_\sharp=n_0$,
\item $p_\sharp=p_\dagger$,
\item $\fn{k}_\sharp=\fn{k}_\star$, and
\item $\fn{r}_\sharp=\fn{r}_\star$.
\end{itemize}
We must check that these satisfy the claim.  We define the following values as specified in the statement of this lemma:
\begin{itemize}
\item $\mathcal{B}_\flat^0=\fn{B}_\flat^0(n_\sharp,p_\sharp,\fn{k}_\sharp,\fn{r}_\sharp,\mathcal{B}_\sharp)=\fn{B}_0(n_0,\fn{k}_0,\mathcal{B}_0)$,
\item $\mathcal{B}_\flat^1=\fn{B}_\flat^1(n_\sharp,p_\sharp,\fn{k}_\sharp,\fn{r}_\sharp,\mathcal{B}_\sharp)=\fn{B}_\dagger(p_\dagger,\fn{r}_\dagger,\mathcal{B}_\dagger)$,
\item $m_\flat=\fn{m}_\flat(n_\sharp,p_\sharp,\fn{k}_\sharp,\fn{r}_\sharp,\mathcal{B}_\sharp)=\fn{m}_0(m_0,\fn{k}_0,\mathcal{B}_0)=m_0$,
\item $q_\flat=\fn{q}_\flat(n_\sharp,p_\sharp,\fn{k}_\sharp,\fn{r}_\sharp,\mathcal{B}_\sharp)=\fn{q}_\dagger(p_\dagger,\fn{r}_\dagger,\mathcal{B}_\dagger)=q_\dagger$,
\item $\fn{l}_\flat=\fn{L}_\flat(n_\sharp,p_\sharp,\fn{k}_\sharp,\fn{r}_\sharp,\mathcal{B}_\sharp)=\fn{l}_\star$,
\item $\fn{s}_\flat=\fn{S}_\flat(n_\sharp,p_\sharp,\fn{k}_\sharp,\fn{r}_\sharp,\mathcal{B}_\sharp)=\fn{s}_\star$,
\item $k_\sharp=\fn{k}_\sharp(m_\flat,q_\flat,\fn{l}_\flat,\fn{s}_\flat,\mathcal{B}_\flat^0,\mathcal{B}_\flat^1)$,
\item $r_\sharp=\fn{r}_\sharp(m_\flat,q_\flat,\fn{l}_\flat,\fn{s}_\flat,\mathcal{B}_\flat^0,\mathcal{B}_\flat^1)=\fn{r}_\dagger(q_\flat,\fn{s}_{\star,\mathcal{B}^0_\flat,\mathcal{B}^1_\flat,m_\flat,q_\flat,\fn{l}_\flat,\fn{s}_\flat},\mathcal{B}_\sharp,\mathcal{B}_\flat^1)$.
\end{itemize}

Many of the other quantities we defined above are equal to these values:
\begin{itemize}
\item $\fn{s}_{\star,\mathcal{B}^0_\flat,\mathcal{B}^1_\flat,m_\flat,q_\flat,\fn{l}_\flat,\fn{s}_\flat}=\fn{S}_\dagger(p_\dagger,\fn{q}_\dagger,\mathcal{B}_\dagger)=\fn{s}_\dagger$,
\item $r_\dagger=\fn{r}_\dagger(q_\dagger,\fn{s}_\dagger,\mathcal{B}_\dagger,\mathcal{B}^1_\star)=r_\sharp$,
\item $\fn{l}_0=\fn{L}_0(n_0,\fn{k}_0,\mathcal{B}_0)=\fn{l}_{\fn{l}_\flat,r_\dagger}$,
\item $k_\sharp=\fn{k}_{\star,r_\sharp}(m_\flat,q_\flat,\fn{l}_\flat,\fn{s}_\flat,\mathcal{B}^0_\flat,\mathcal{B}^1_\flat)=\fn{k}_0(m_\flat, \fn{l}_{\fn{l}_\flat,r_\sharp},\mathcal{B}_\sharp,\mathcal{B}^1_\flat)=k_0$,
\item $\fn{l}_0(k_\sharp)=\fn{l}_\flat(k_\sharp,r_\sharp)$,
\item $\fn{s}_\dagger(r_\sharp)=\fn{s}_\flat(k_\sharp,r_\sharp)$.
\end{itemize}

Suppose $m_\flat\geq n_\sharp$ and $q_\flat\geq p_\sharp$.  Then we have $\mathcal{B}_0\preceq\mathcal{B}_\sharp\preceq\mathcal{B}^0_\flat=\fn{B}_0(n_0,\fn{k}_0,\mathcal{B}_0)$, and so $k_\sharp=k_0\geq m_0=m_\flat$ and for every $l\in[k_\sharp,\fn{l}_0(k_\sharp)]=[k_\sharp,\fn{l}_\flat(k_\sharp,r_\sharp)]$, $\mu(\mathfrak{D}_{E_\flat,\mathcal{B}_\sharp,l}(\mathcal{B}^1_\flat))<1/D_\flat$.

We also have $\mathcal{B}_\dagger=\mathcal{B}_\sharp\preceq\mathcal{B}^1_\flat=\fn{B}_\dagger(p_\dagger,\fn{r}_\dagger,\mathcal{B}_\dagger)$, so $r_\sharp=r_\dagger\geq q_\dagger=q_\flat$ and for every $s\in[r_\sharp,\fn{s}_\dagger(r_\sharp)]=[r_\sharp,\fn{s}_\flat(k_\sharp,r_\sharp)]$ we have $\mu(\mathfrak{D}_{E_\flat,\mathcal{B}_\sharp,s}(\mathcal{B}^1_\flat))<1/D_\flat$.
\end{proof}

The following lemma is our first approximation to the final result; it shows that we can attain some sort of bound on
\[|(\rho_m\lambda_s)(\Omega)-(\rho_l\lambda_q)(\Omega)|\]
when $s,l$ are suitably chosen and much larger than $m,q$.  The remainder of the argument will amount to refining the right side of the inequality to depend only on $E_\flat$.

\begin{lemma}\label{thm:control_interval}
Suppose $(*)$ holds.  Then for every $E_\flat, D^0_\flat\leq D^1_\flat, p_\flat, n_\flat, \fn{L}_\flat, \fn{S}_\flat$ there are $m_\sharp\geq n_\flat$, $q_\sharp\geq p_\flat$, $\fn{k}_\sharp$, $\fn{r}_\sharp$ so that, setting
\begin{itemize}
\item $\fn{l}_\flat=\fn{L}_\flat(m_\sharp,q_\sharp,\fn{k}_\sharp,\fn{r}_\sharp)$,
\item $\fn{s}_\flat=\fn{S}_\flat(m_\sharp,q_\sharp,\fn{k}_\sharp,\fn{r}_\sharp)$,
\item $k_\sharp=\fn{k}_\sharp(\fn{l}_\flat,\fn{s}_\flat)$,
\item $r_\sharp=\fn{r}_\sharp(\fn{l}_\flat,\fn{s}_\flat)$,
\item $l_\flat=\fn{l}_\flat(k_\sharp,r_\sharp)$,
\item $s_\flat=\fn{s}_\flat(k_\sharp,r_\sharp)$,
\end{itemize}
we have $k_\sharp\geq m_\sharp$, $r_\sharp\geq q_\sharp$, and if $l_\flat\geq k_\sharp$ and $s_\flat\geq r_\sharp$ then for any $s\in[r_\sharp,s_\flat]$ and $l\in[k_\sharp,l_\flat]$ there are sets $\mathcal{B}^-,\mathcal{B}^{0,-}$, and $\mathcal{B}^{1,-}$ with $\mu(\mathcal{B}^-)<4/D^0_\flat$, $\mu(\mathcal{B}^{0,-})<2/D_\flat^1$, and $\mu(\mathcal{B}^{1,-})<2/D_\flat^1$ so that
\begin{align*}
\left|(\rho_{m_\sharp}\lambda_s)(\Omega)-(\rho_l\lambda_{q_\sharp})(\Omega)\right|
&\leq|\rho_{m_\sharp}\lambda_s|(\mathcal{B}^-)+|\rho_l\lambda_{q_\sharp}|(\mathcal{B}^-)\\
&\ \ \ \ \ \ \ \ +BD_\flat^0|\rho_{m_\sharp}|(\mathcal{B}^{0,-})+|\rho_{m_\sharp}\lambda_s|(\mathcal{B}^{0,-})\\
&\ \ \ \ \ \ \ \ +BD_\flat^0|\lambda_{q_\sharp}|(\mathcal{B}^{1,-})+|\rho_l\lambda_{q_\sharp}|(\mathcal{B}^{1,-})+\frac{6}{E_\flat}.\\
\end{align*}
\end{lemma}
\begin{proof}

By $(\ast)_4$, for any $m,\mathcal{B}$ there are $\mathcal{B}'\succeq\mathcal{B}$ and a $\mathcal{B}''\subseteq\mathcal{B}'$ such that $\rho_m$ is $BE_\flat$-constant on $\mathcal{B}'\setminus\mathcal{B}''$ and $\mu(\mathcal{B}'')<1/D^1_\flat$.  Write $\fn{B}^0_\star(\mathcal{B},m)=\mathcal{B}'$ and $\fn{B}^{0,-}_\star(\mathcal{B},m)=\mathcal{B}''$.  Similarly, for any $q,\mathcal{B}$ there are $\mathcal{B}'\succeq\mathcal{B}$ and a $\mathcal{B}''\subseteq\mathcal{B}'$ such that $\lambda_r$ is $BE_\flat$-constant on $\mathcal{B}'\setminus\mathcal{B}''$ and $\mu(\mathcal{B}'')<1/D^1_\flat$.  Write $\fn{B}^1_\star(\mathcal{B},q)=\mathcal{B}'$ and $\fn{B}^{1,-}_\star(\mathcal{B},q)=\mathcal{B}''$.

We plan to use Lemma \ref{thm:1D_sequential_interval_double}.  We need to define $\fn{B}^0_0(n_0,p_0,\fn{k}_0,\fn{r}_0,\mathcal{B}_0)$, $\fn{B}^1_0(n_0,p_0,\fn{k}_0,\fn{r}_0,\mathcal{B}_0)$, $\fn{m}_0(n_0,p_0,\fn{k}_0,\fn{r}_0,\mathcal{B}_0)$, $\fn{q}_0(n_0,p_0,\fn{k}_0,\fn{r}_0,\mathcal{B}_0)$, $\fn{L}_0(n_0,p_0,\fn{k}_0,\fn{r}_0,\mathcal{B}_0)$, and $\fn{S}_0(n_0,p_0,\fn{k}_0,\fn{r}_0,\mathcal{B}_0)$.

\begin{quotationb*}
  On input $n_0,p_0,\fn{k}_0,\fn{r}_0,\mathcal{B}_0$, which we abbreviate $\dagger$, we proceed as follows.  We plan to apply Lemma \ref{thm:quant_convergence_partition}, so we define $\fn{m}_\dagger(m_\dagger)$.

  \begin{quotationb*}
    Let input $m_\dagger$ be given; we abbreviate $\dagger,m_\dagger$ by $\ddagger$.  We plan to apply Lemma \ref{thm:quant_convergence_partition} again, so we define $\fn{q}_\ddagger(q_\ddagger)$.

    \begin{quotationb*}
      Let $q_\ddagger$ be given.  We define
      \begin{itemize}
      \item $\fn{k}_{\ddagger,q_\ddagger}(\fn{l},\fn{s})=\fn{k}_0(m_\dagger,q_\ddagger,\fn{l},\fn{s},\fn{B}^0_\star(\mathcal{B}_0,m_\dagger),\fn{B}^1_\star(\mathcal{B}_0,q_\ddagger))$,
      \item $\fn{r}_{\ddagger,q_\ddagger}(\fn{l},\fn{s})=\fn{r}_0(m_\dagger,q_\ddagger,\fn{l},\fn{s},\fn{B}^0_\star(\mathcal{B}_0,m_\dagger),\fn{B}^1_\star(\mathcal{B}_0,q_\ddagger))$,
      \item $\fn{l}_{\ddagger,q_\ddagger}=\fn{L}_\flat(m_\dagger,q_\ddagger,\fn{k}_{\ddagger,q_\ddagger},\fn{r}_{\ddagger,q_\ddagger})$,
      \item $\fn{s}_{\ddagger,q_\ddagger}=\fn{S}_\flat(m_\dagger,q_\ddagger,\fn{k}_{\ddagger,q_\ddagger},\fn{r}_{\ddagger,q_\ddagger})$,
      \item $\fn{q}_\ddagger(q_\ddagger)=\fn{s}_{\ddagger,q_\ddagger}(\fn{k}_{\ddagger,q_\ddagger}(\fn{l}_{\ddagger,q_\ddagger},\fn{s}_{\ddagger,q_\ddagger}),\fn{r}_{\ddagger,q_\ddagger}(\fn{l}_{\ddagger,q_\ddagger},\fn{s}_{\ddagger,q_\ddagger}))$.
      \end{itemize}
    \end{quotationb*}

By Lemma \ref{thm:quant_convergence_partition} applied to $|\mathcal{B}_0|BD_\flat^0 E_\flat$, $3D_\flat^0$, $\fn{q}_\ddagger$, $\max\{p_\flat,p_0\}$, we obtain $q_\ddagger\geq \max\{p_\flat,p_0\}$ and a set $\mathcal{B}^{1,=}\subseteq\mathcal{B}_0$ so that $\mu(\mathcal{B}^{1,=})<1/3D_\flat^0$ and for each $q,q'\in[q_\ddagger,\fn{q}_\ddagger(q_\ddagger)]$ and each $\sigma\in\mathcal{B}_0\setminus\mathcal{B}^{1,=}$, $|\lambda_q(\sigma)-\lambda_{q'}(\sigma)|<1/|\mathcal{B}_0|BD_\flat^0E_\flat$.

  We define
\begin{itemize}
\item $\fn{m}_\dagger(m_\dagger)=\fn{l}_{\ddagger,q_\ddagger}(\fn{k}_{\ddagger,q_\ddagger}(\fn{l}_{\ddagger,q_\ddagger},\fn{s}_{\ddagger,q_\ddagger}),\fn{r}_{\ddagger,q_\ddagger}(\fn{l}_{\ddagger,q_\ddagger},\fn{s}_{\ddagger,q_\ddagger}))$.
\end{itemize}
  \end{quotationb*}
By Lemma \ref{thm:quant_convergence_partition} applied to $|\mathcal{B}_0|BD_\flat^0E_\flat$, $3D_\flat^0$, $\fn{m}_\dagger$, $\max\{n_\flat,n_0\}$, we obtain $m_\dagger\geq\max\{n_\flat,n_0\}$ and a set $\mathcal{B}^{0,=}\subseteq\mathcal{B}_0$ so that $\mu(\mathcal{B}^{0,=})<1/3D_\flat^0$ and for each $m,m'\in[m_\dagger,\fn{m}_\dagger(m_\dagger)]$ and each $\sigma\in\mathcal{B}_0\setminus\mathcal{B}^{0,=}$, $|\rho_m(\sigma)-\rho_{m'}(\sigma)|<1/|\mathcal{B}_0|BD_\flat^0 E_\flat$.

We can now set
\begin{itemize}
\item $\fn{B}^0_0(n_0,p_0,\fn{k}_0,\fn{r}_0,\mathcal{B}_0)=\fn{B}^0_\star(\mathcal{B}_0,m_\dagger)$,
\item $\fn{B}^1_0(n_0,p_0,\fn{k}_0,\fn{r}_0,\mathcal{B}_0)=\fn{B}^1_\star(\mathcal{B}_0,q_\ddagger)$,
\item $\fn{q}_0(n_0,p_0,\fn{k}_0,\fn{r}_0,\mathcal{B}_0)=q_\ddagger$,
\item $\fn{m}_0(n_0,p_0,\fn{k}_0,\fn{r}_0,\mathcal{B}_0)=m_\dagger$,
\item $\fn{L}_0(n_0,p_0,\fn{k}_0,\fn{r}_0,\mathcal{B}_0)=\fn{l}_{\ddagger,q_\ddagger}$,
\item $\fn{S}_0(n_0,p_0,\fn{k}_0,\fn{r}_0,\mathcal{B}_0)=\fn{s}_{\ddagger,q_\ddagger}$.
\end{itemize}
\end{quotationb*}
We apply Lemma \ref{thm:1D_sequential_interval_double} to $BE_\flat$, $D_\flat^1$, $\fn{B}^0_0, \fn{B}^1_0, \fn{m}_0,\fn{L}_0,\fn{q}_0,\fn{S}_0$.  We obtain $n_0,p_0,\mathcal{B}_0,\fn{k}_0,\fn{r}_0$.  We set $m_\sharp=m_\dagger$, $q_\sharp=q_\ddagger$, $\fn{k}_\sharp=\fn{k}_{\ddagger,q_\ddagger}$, and $\fn{r}_\sharp=\fn{r}_{\ddagger,q_\ddagger}$.

Set
\begin{itemize}
\item $\fn{l}_\flat=\fn{L}_\flat(m_\sharp,q_\sharp,\fn{k}_\sharp,\fn{r}_\sharp)$,
\item $\fn{s}_\flat=\fn{S}_\flat(m_\sharp,q_\sharp,\fn{k}_\sharp,\fn{r}_\sharp)$,
\item $k_\sharp=\fn{k}_\sharp(\fn{l}_\flat,\fn{s}_\flat)$, and
\item $r_\sharp=\fn{r}_\sharp(\fn{l}_\flat,\fn{s}_\flat)$.
\end{itemize}
Note that $k_\sharp\geq m_\dagger=m_\sharp$ and $r_\sharp\geq q_\dagger=q_\sharp$.  Suppose $\fn{l}_\flat(k_\sharp,r_\sharp)\geq k_\sharp$ and $\fn{s}_\flat(k_\sharp,r_\sharp)\geq r_\sharp$ and let $s\in[r_\sharp,\fn{s}_\flat(k_\sharp,r_\sharp)]$ and $l\in[k_\sharp,\fn{l}_\flat(k_\sharp,r_\sharp)]$ be given.

Observe that $\fn{l}_\flat=\fn{L}_\flat(m_\dagger,q_\ddagger,\fn{k}_{\ddagger,q_\ddagger},\fn{r}_{\ddagger,q_\ddagger})=\fn{l}_{\ddagger,q_\ddagger}$ and similarly, $\fn{s}_\flat=\fn{s}_{\ddagger,q_\ddagger}$.  Therefore
$r_\sharp=\fn{r}_{\ddagger,q_\ddagger}(\fn{l}_\flat,\fn{s}_\flat)=\fn{r}_0(m_\sharp,q_\sharp,\fn{l}_{\ddagger,q_\ddagger},\fn{s}_{\ddagger,q_\ddagger},\fn{B}^0_\star(\mathcal{B}_0,m_\dagger),\fn{B}^1_\star(\mathcal{B}_0,q_\ddagger))$ and similarly, for $k_\sharp$.  Therefore, by our application of Lemma \ref{thm:1D_sequential_interval_double}, since $m_\sharp\geq n_0$ and $q_\sharp\geq p_0$, we have $k_\sharp\geq m_\sharp$, $r_\sharp\geq q_\sharp$, for every $l\in[k_\sharp,\fn{l}_\flat(k_\sharp,r_\sharp)]$,  $\mu(\mathfrak{D}_{BE_\flat,\mathcal{B}_0,l}(\fn{B}^0_\star(\mathcal{B}_0,m_\dagger)))<1/D_\flat^1$, and for every $s\in[r_\sharp, \fn{s}_\flat(k_\sharp,r_\sharp)]$, $\mu(\mathfrak{D}_{B E_\flat,\mathcal{B}_0,s}(\fn{B}^1_\star(\mathcal{B}_0,q_\ddagger)))<1/D_\flat^1$.

Fix some $l\in [k_\sharp,\fn{l}_\flat(k_\sharp,r_\sharp)]$, $s\in[r_\sharp, \fn{s}_\flat(k_\sharp,r_\sharp)]$.  Set $\mathcal{B}^0_\flat=\fn{B}^0_\star(\mathcal{B}_0,m_\dagger)$, $\mathcal{B}^1_\flat=\fn{B}^1_\star(\mathcal{B}_0,q_\ddagger)$, $\mathcal{B}^{0,-}=\mathfrak{D}_{BE_\flat,\mathcal{B}_0,l}(\mathcal{B}^0_\flat)\cup\fn{B}^{0,-}_\star(\mathcal{B}_0,m_\dagger)$, $\mathcal{B}^{1,-}=\mathfrak{D}_{B E_\flat,\mathcal{B}_0,s}(\mathcal{B}^1_\flat)\cup\fn{B}^{1,-}_\star(\mathcal{B}_0,q_\ddagger)$.  Therefore $\mu(\mathcal{B}^{0,-}),\mu(\mathcal{B}^{1,-})<2/D_\flat^1$.

By Lemma \ref{thm:refinement}, 
\[\sum_{\sigma\in(\mathcal{B}_0)_{\lambda_s\leq BD_\flat^0}}|(\rho_{m_\sharp}\lambda_s)(\sigma)-(\rho_{m_\sharp}\ast\lambda_s)(\sigma)|<\frac{2}{E_\flat}+BD_\flat^0|\rho_{m_\sharp}|(\mathcal{B}^{0,-})+|\rho_{m_\sharp}\lambda_s|(\mathcal{B}^{0,-})\]
and
\[\sum_{\sigma\in (\mathcal{B}_0)_{\rho_l\leq BD_\flat^0}}|(\rho_l\lambda_{q_\sharp})(\sigma)-(\rho_l\ast\lambda_{q_\sharp})(\sigma)|<\frac{2}{E_\flat}+BD_\flat^0|\lambda_{q_\sharp}|(\mathcal{B}^{1,-})+|\rho_l\lambda_{q_\sharp}|(\mathcal{B}^{1,-}).\]

If $\sigma\in(\mathcal{B}_0)_{\lambda_s\leq BD_\flat^0,\rho_l\leq BD_\flat^0}\setminus(\mathcal{B}^{0,-}\cup\mathcal{B}^{1,-})$ then we have
\begin{align*}
|(\rho_{m_\sharp}\ast\lambda_s)(\sigma)-(\rho_l\ast\lambda_{q_\sharp})(\sigma)|
&\leq|(\rho_{m_\sharp}\ast\lambda_s)(\sigma)-(\rho_l\ast\lambda_s)(\sigma)|+|(\rho_l\ast\lambda_s)(\sigma)-(\rho_l\ast\lambda_{q_\sharp})(\sigma)|\\
&\leq |\delta_{\lambda_s}(\sigma)|\cdot|\rho_{m_\sharp}(\sigma)-\rho_l(\sigma)|+|\delta_{\rho_l}(\sigma)|\cdot|\lambda_s(\sigma)-\lambda_{q_\sharp}(\sigma)|\\
&\leq BD_\flat^0\cdot\frac{1}{|\mathcal{B}_0|BD_\flat^0 E_\flat}+BD_\flat^0\cdot\frac{1}{|\mathcal{B}_0|BD_\flat^0 E_\flat}\\
&=\frac{2}{|\mathcal{B}_0|E_\flat}.
\end{align*}

Let $\mathcal{B}^-=\mathcal{B}^{0,-}\cup\mathcal{B}^{1,-}\cup(\mathcal{B}_0)_{\rho_l>BD_\flat^0} \cup(\mathcal{B}_0)_{\lambda_s>BD_\flat^0}$ and $\mathcal{B}^+=\mathcal{B}_0\setminus\mathcal{B}^-$.  We have
\begin{align*}
  |(\rho_{m_\sharp}\lambda_s)(\Omega)-(\rho_l\lambda_{q_\sharp})(\Omega)|
&=|(\rho_{m_\sharp}\lambda_s)(\mathcal{B}_0)-(\rho_l\lambda_{q_\sharp})(\mathcal{B}_0)|\\
&\leq|(\rho_{m_\sharp}\lambda_s)(\mathcal{B}^+)-(\rho_l\lambda_{q_\sharp})(\mathcal{B}^+)|+|\rho_{m_\sharp}\lambda_s|(\mathcal{B}^-)+|\rho_l\lambda_{q_\sharp}|(\mathcal{B}^-)\\
&<|(\rho_{m_\sharp}\ast\lambda_s)(\mathcal{B}^+)-(\rho_l\ast\lambda_{q_\sharp})(\mathcal{B}^+)|+|\rho_{m_\sharp}\lambda_s|(\mathcal{B}^-)+|\rho_l\lambda_{q_\sharp}|(\mathcal{B}^-)\\
&\ \ \ \ \ \ \ \ +BD_\flat^0|\rho_{m_\sharp}|(\mathcal{B}^{0,-})+|\rho_{m_\sharp}\lambda_s|(\mathcal{B}^{0,-})\\
&\ \ \ \ \ \ \ \ +BD_\flat^0|\lambda_{q_\sharp}|(\mathcal{B}^{1,-})+|\rho_l\lambda_{q_\sharp}|(\mathcal{B}^{1,-})+\frac{4}{E_\flat}\\
&\leq|\rho_{m_\sharp}\lambda_s|(\mathcal{B}^-)+|\rho_l\lambda_{q_\sharp}|(\mathcal{B}^-)\\
&\ \ \ \ \ \ \ \ +BD_\flat^0|\rho_{m_\sharp}|(\mathcal{B}^{0,-})+|\rho_{m_\sharp}\lambda_s|(\mathcal{B}^{0,-})\\
&\ \ \ \ \ \ \ \ +BD_\flat^0|\lambda_{q_\sharp}|(\mathcal{B}^{1,-})+|\rho_l\lambda_{q_\sharp}|(\mathcal{B}^{1,-})+\frac{6}{E_\flat}\\
\end{align*}
\end{proof}

What remains is a series of lemma in which we choose $D_\flat^0,D_\flat^1$ large enough to bound the various terms.

\begin{lemma}\label{thm:control_interval_D1}
Suppose $(*)$ holds.  Then for every $E_\flat, D^0_\flat\leq D^1_\flat, p_\flat, n_\flat, \fn{L}_\flat, \fn{S}_\flat$ there are $m_\sharp\geq n_\flat$, $q_\sharp\geq p_\flat$, $\fn{k}_\sharp$, $\fn{r}_\sharp$ so that, setting
\begin{itemize}
\item $\fn{l}_\flat=\fn{L}_\flat(m_\sharp,q_\sharp,\fn{k}_\sharp,\fn{r}_\sharp)$,
\item $\fn{s}_\flat=\fn{S}_\flat(m_\sharp,q_\sharp,\fn{k}_\sharp,\fn{r}_\sharp)$,
\item $k_\sharp=\fn{k}_\sharp(\fn{l}_\flat,\fn{s}_\flat)$,
\item $r_\sharp=\fn{r}_\sharp(\fn{l}_\flat,\fn{s}_\flat)$,
\item $l_\flat=\fn{l}_\flat(k_\sharp,r_\sharp)$,
\item $s_\flat=\fn{s}_\flat(k_\sharp,r_\sharp)$,
\end{itemize}
we have $k_\sharp\geq m_\sharp$, $r_\sharp\geq q_\sharp$, and if $l_\flat\geq k_\sharp$ and $s_\flat\geq r_\sharp$ then for any $s\in[r_\sharp,s_\flat]$ and $l\in[k_\sharp,l_\flat]$  there are sets $\mathcal{B}^-,\mathcal{B}^{0,-}$, and $\mathcal{B}^{1,-}$ with $\mu(\mathcal{B}^-)<4/D_\flat^0$, $\mu(\mathcal{B}^{0,-})<2/D_\flat^0$, and $\mu(\mathcal{B}^{1,-})<2/D_\flat^1$ so that
\begin{align*}
\left|(\rho_{m_\sharp}\lambda_s)(\Omega)-(\rho_l\lambda_{q_\sharp})(\Omega)\right|
&\leq|\rho_{m_\sharp}\lambda_s|(\mathcal{B}^-)+|\rho_l\lambda_{q_\sharp}|(\mathcal{B}^-)+|\rho_{m_\sharp}\lambda_s|(\mathcal{B}^{0,-})\\
&\ \ \ \ \ \ \ \ +BD_\flat^0|\lambda_{q_\sharp}|(\mathcal{B}^{1,-})+|\rho_l\lambda_{q_\sharp}|(\mathcal{B}^{1,-})+\frac{7}{E_\flat}.\\
\end{align*}
\end{lemma}
\begin{proof}
  Towards the use of uniform continuity, we define $\fn{m}_0(D_0,m_0)$.
  \begin{quotationb*}
    Let $D_0, m_0$, which we abbreviate $\dagger$, be given.

By Lemma \ref{thm:control_interval} applied to $E_\flat, D^0_\flat, \max\{D_0,D^1_\flat\}, p_\flat, m_0, \fn{L}_\flat,\fn{S}_\flat$ we obtain $m_\dagger, q_\dagger, \fn{k}_\dagger, \fn{r}_\dagger$. 
We define
\[\fn{m}_0(D_0,m_0)=m_\dagger.\]
  \end{quotationb*}
By Theorem \ref{thm:q_vhs} applied to $4BD_\flat^0E_\flat, \fn{m}_0, n_\flat$ we obtain $D_0, m_0\geq n_\flat$ so that whenever $m\in[m_0,\fn{m}_0(D_0,m_0)]$ and $\mu(\sigma)<1/D_0$, $|\rho_m(\sigma)|<1/4BD_\flat^0E_\flat$.

We set $m_\sharp=m_\dagger$, $q_\sharp=q_\dagger$, $\fn{k}_\sharp=\fn{k}_\dagger$, and $\fn{r}_\sharp=\fn{r}_\dagger$.  Let $\fn{l}_\flat,\fn{s}_\flat,k_\sharp,r_\sharp$ be as in the statement.  Let $s\in[r_\sharp,\fn{s}_\flat(k_\sharp,r_\sharp)]$ and $l\in[k_\sharp,\fn{l}_\flat(k_\sharp,r_\sharp)]$ be given, so for appropriate $\mathcal{B}^-,\mathcal{B}^{0,-},\mathcal{B}^{1,-}$, 
\begin{align*}
\left|(\rho_{m_\sharp}\lambda_s)(\Omega)-(\rho_l\lambda_{q_\sharp})(\Omega)\right|
&\leq|\rho_{m_\sharp}\lambda_s|(\mathcal{B}^-)+|\rho_l\lambda_{q_\sharp}|(\mathcal{B}^-)\\
&\ \ \ \ \ \ \ \ +BD_\flat^0|\rho_{m_\sharp}|(\mathcal{B}^{0,-})+|\rho_{m_\sharp}\lambda_s|(\mathcal{B}^{0,-})\\
&\ \ \ \ \ \ \ \ +BD_\flat^0|\lambda_{q_\sharp}|(\mathcal{B}^{1,-})+|\rho_l\lambda_{q_\sharp}|(\mathcal{B}^{1,-})+\frac{6}{E_\flat}.\\
\end{align*}

Since $\mu(\mathcal{B}^{0,-})<2/D_0$, $|\rho_{m_\sharp}(\mathcal{B}^{0,-})|<1/BD_\flat^0 E_\flat$, and therefore
\begin{align*}
\left|(\rho_{m_\sharp}\lambda_s)(\Omega)-(\rho_l\lambda_{q_\sharp})(\Omega)\right|
&\leq|\rho_{m_\sharp}\lambda_s|(\mathcal{B}^-)+|\rho_l\lambda_{q_\sharp}|(\mathcal{B}^-)+|\rho_{m_\sharp}\lambda_s|(\mathcal{B}^{0,-})\\
&\ \ \ \ \ \ \ \ +BD_\flat^0|\lambda_{q_\sharp}|(\mathcal{B}^{1,-})+|\rho_l\lambda_{q_\sharp}|(\mathcal{B}^{1,-})+\frac{7}{E_\flat}\\
\end{align*}
as desired.
\end{proof}

\begin{lemma}\label{thm:control_interval_D2}
Suppose $(*)$ holds.  Then for every $E_\flat, D_\flat, p_\flat, n_\flat, \fn{L}_\flat, \fn{S}_\flat$ there are $m_\sharp\geq n_\flat$, $q_\sharp\geq p_\flat$, $\fn{k}_\sharp$, $\fn{r}_\sharp$ so that, setting
\begin{itemize}
\item $\fn{l}_\flat=\fn{L}_\flat(m_\sharp,q_\sharp,\fn{k}_\sharp,\fn{r}_\sharp)$,
\item $\fn{s}_\flat=\fn{S}_\flat(m_\sharp,q_\sharp,\fn{k}_\sharp,\fn{r}_\sharp)$,
\item $k_\sharp=\fn{k}_\sharp(\fn{l}_\flat,\fn{s}_\flat)$,
\item $r_\sharp=\fn{r}_\sharp(\fn{l}_\flat,\fn{s}_\flat)$,
\item $l_\flat=\fn{l}_\flat(k_\sharp,r_\sharp)$,
\item $s_\flat=\fn{s}_\flat(k_\sharp,r_\sharp)$,
\end{itemize}
we have $k_\sharp\geq m_\sharp$, $r_\sharp\geq q_\sharp$, and if $l_\flat\geq k_\sharp$ and $s_\flat\geq r_\sharp$ then for any $s\in[r_\sharp,s_\flat]$ and $l\in[k_\sharp,l_\flat]$ there are sets $\mathcal{B}^-,\mathcal{B}^{0,-}$, and $\mathcal{B}^{1,-}$ with $\mu(\mathcal{B}^-)<4/D_\flat$, $\mu(\mathcal{B}^{0,-})<2/D_\flat$, and $\mu(\mathcal{B}^{1,-})<2/D_\flat$ so that
\begin{align*}
\left|(\rho_{m_\sharp}\lambda_s)(\Omega)-(\rho_l\lambda_{q_\sharp})(\Omega)\right|
&\leq|\rho_{m_\sharp}\lambda_s|(\mathcal{B}^-)+|\rho_l\lambda_{q_\sharp}|(\mathcal{B}^-)+|\rho_{m_\sharp}\lambda_s|(\mathcal{B}^{0,-})+|\rho_l\lambda_{q_\sharp}|(\mathcal{B}^{1,-})+\frac{8}{E_\flat}.\\
\end{align*}
\end{lemma}
\begin{proof}
  Towards the use of uniform continuity, we define $\fn{q}_0(D_0,q_0)$.
  \begin{quotationb*}
    Let $D_0, q_0$, which we abbreviate $\dagger$, be given.

By Lemma \ref{thm:control_interval_D1} applied to $E_\flat, D_\flat,
\max\{D_\flat,D_0\}, n_\flat, q_0, \fn{L}_\flat,\fn{S}_\flat$ we obtain $m_\dagger, q_\dagger, \fn{k}_\dagger, \fn{r}_\dagger$.  We define
\[\fn{q}_0(D_0,q_0)=q_\dagger.\]
  \end{quotationb*}
By Theorem \ref{thm:q_vhs} applied to $4BD_\flat E_\flat, \fn{q}_0, p_\flat$ we obtain $D_0\geq D_\flat, q_0\geq p_\flat$.

We set $m_\sharp=m_\dagger$, $q_\sharp=q_\dagger$, $\fn{k}_\sharp=\fn{k}_\dagger$, and $\fn{r}_\sharp=\fn{r}_\dagger$, and let $\fn{l}_\flat, \fn{s}_\flat, k_\sharp, r_\sharp$ be as in the statement, so also $\fn{l}_\flat=\fn{l}_\dagger$, $\fn{s}_\flat=\fn{s}_\dagger$, $k_\sharp=k_\dagger$, and $r_\sharp=r_\dagger$.  Let $s\in[r_\sharp,\fn{s}_\flat(k_\sharp,r_\sharp)]$ and $l\in[k_\sharp,\fn{l}_\flat(k_\sharp,r_\sharp)]$ be given, so for appropriate $\mathcal{B}^-,\mathcal{B}^{0,-},\mathcal{B}^{1,-}$, 
\begin{align*}
\left|(\rho_{m_\sharp}\lambda_s)(\Omega)-(\rho_l\lambda_{q_\sharp})(\Omega)\right|
&\leq|\rho_{m_\sharp}\lambda_s|(\mathcal{B}^-)+|\rho_l\lambda_{q_\sharp}|(\mathcal{B}^-)+|\rho_{m_\sharp}\lambda_s|(\mathcal{B}^{0,-})\\
&\ \ \ \ \ \ \ \ +BD_\flat|\lambda_{q_\sharp}|(\mathcal{B}^{1,-})+|\rho_l\lambda_{q_\sharp}|(\mathcal{B}^{1,-})+\frac{7}{E_\flat}.\\
\end{align*}

Since $\mu(\mathcal{B}^{1,-})<2/D_\flat$, we have $|\lambda_{q_\sharp}(\mathcal{B}^{1,-})|<1/BD_\flat E_\flat$, and therefore
\begin{align*}
\left|(\rho_{m_\sharp}\lambda_s)(\Omega)-(\rho_l\lambda_{q_\sharp})(\Omega)\right|
&\leq|\rho_{m_\sharp}\lambda_s|(\mathcal{B}^-)+|\rho_l\lambda_{q_\sharp}|(\mathcal{B}^-)+|\rho_{m_\sharp}\lambda_s|(\mathcal{B}^{0,-})+|\rho_l\lambda_{q_\sharp}|(\mathcal{B}^{1,-})+\frac{8}{E_\flat}\\
\end{align*}
as desired.
\end{proof}

\begin{lemma}\label{thm:control_interval_E1}
Suppose $(*)$ holds.  Then for every $E_\flat, D_\flat, p_\flat, n_\flat, \fn{L}_\flat, \fn{r}_\flat$ there are $m_\sharp\geq n_\flat$, $q_\sharp\geq p_\flat$, $\fn{k}_\sharp$, $\fn{s}_\sharp$ so that, setting
\begin{itemize}
\item $\fn{l}_\flat=\fn{L}_\flat(m_\sharp,q_\sharp,\fn{k}_\sharp,\fn{s}_\sharp)$,
\item $r_\flat=\fn{r}_\flat(m_\sharp,q_\sharp,\fn{k}_\sharp,\fn{s}_\sharp)$,
\item $k_\sharp=\fn{k}_\sharp(\fn{l}_\flat,r_\flat)$,
\item $s_\sharp=\fn{s}_\sharp(\fn{l}_\flat,r_\flat)$,
\item $l_\flat=\fn{l}_\flat(k_\sharp,s_\sharp)$,
\end{itemize}
we have $k_\sharp\geq m_\sharp$, if $r_\flat\geq q_\sharp$ then $s_\sharp\geq r_\flat$, and for any $l\in[k_\sharp,l_\flat]$ there are sets $\mathcal{B}^-,\mathcal{B}^{0,-}$, and $\mathcal{B}^{1,-}$ with $\mu(\mathcal{B}^-)<4/D_\flat$, $\mu(\mathcal{B}^{0,-})<2/D_\flat$, and $\mu(\mathcal{B}^{1,-})<2/D_\flat$ so that
\[\left|(\rho_{m_\sharp}\lambda_{s_\sharp})(\Omega)-(\rho_l\lambda_{q_\sharp})(\Omega)\right|\leq|\rho_l\lambda_{q_\sharp}|(\mathcal{B}^{-})+|\rho_l\lambda_{q_\sharp}|(\mathcal{B}^{1,-})+\frac{20}{E_\flat}.\]
\end{lemma}
\begin{proof}
  Towards the use of $n_{/p}$-metastable uniform continuity, we define functions $\fn{m}_0(D_0,m_0,p_0,\fn{r}_0)$ and $\fn{q}_0(D_0,m_0,p_0,\fn{r}_0)$.
  \begin{quotationb*}
    Let $D_0, m_0, p_0, \fn{r}_0$, which we abbreviate $\dagger$, be given.  Without loss of generality, we assume $D_0\geq D_\flat$.  In order to apply Lemma \ref{thm:control_interval_D2} we define
      \begin{itemize}
      \item $\fn{s}_{m_\dagger,r}(k,r')=\fn{r}_0(m_\dagger,\max\{r,r',p_0\})$,
      \item $\fn{l}_{m_\dagger,\fn{l},r}(k,r')=\fn{l}(k,\fn{r}_0(m_\dagger,\max\{r,r',p_0\}))$,
      \item $\fn{k}_{m_\dagger,\fn{k}_\dagger}(\fn{l},r)=\fn{k}_\dagger(\fn{l}_{m_\dagger,\fn{l},r},\fn{s}_{m_\dagger,r})$,
      \item $\fn{s}_{m_\dagger,\fn{r}_\dagger}(\fn{l},r)=\fn{r}_0(m_\dagger,\max\{r,\fn{r}_\dagger(\fn{l}_{m_\dagger,\fn{l},r},\fn{s}_{m_\dagger,r}),p_0\})$,
      \item $\fn{L}_\dagger(m_\dagger,q_\dagger,\fn{k}_\dagger,\fn{r}_\dagger)=\fn{l}_{m_\dagger,\fn{L}_\flat(m_\dagger,q_\dagger,\fn{k}_{m_\dagger,\fn{k}_\dagger},\fn{s}_{m_\dagger,\fn{r}_\dagger}),\fn{r}_\flat(m_\dagger,q_\dagger,\fn{k}_{m_\dagger,\fn{k}_\dagger},\fn{s}_{m_\dagger,\fn{r}_\dagger})}$,
      \item $\fn{S}_\dagger(m_\dagger,q_\dagger,\fn{k}_\dagger,\fn{r}_\dagger)=\fn{s}_{m_\dagger,\fn{r}_\flat(m_\dagger,q_\dagger,\fn{k}_{m_\dagger,\fn{k}_\dagger},\fn{s}_{m_\dagger,\fn{r}_\dagger})}$.
      \end{itemize}

We apply Lemma \ref{thm:control_interval_D2} to $E_\flat, D_0, \max\{p_\flat,p_0\}, n_\flat, \fn{L}_\dagger, \fn{S}_\dagger$ in order to obtain $m_\dagger, q_\dagger, \fn{k}_\dagger, \fn{r}_\dagger$.  Set $\fn{l}_\dagger=\fn{L}_\dagger(m_\dagger,q_\dagger,\fn{k}_\dagger,\fn{r}_\dagger)$, $\fn{s}_\dagger=\fn{S}_\dagger(m_\dagger,q_\dagger,\fn{k}_\dagger,\fn{r}_\dagger)$, $k_\dagger=\fn{k}_\dagger(\fn{l}_\dagger,\fn{s}_\dagger)$, and $r_\dagger=\fn{r}_\dagger(\fn{l}_\dagger,\fn{s}_\dagger)$.

We set
\begin{itemize}
\item $\fn{m}_0(D_0,m_0,p_0,\fn{r}_0)=m_\dagger$, and
\item $\fn{q}_0(D_0,m_0,p_0,\fn{r}_0)=\max\{\fn{r}_\flat(m_\dagger,q_\dagger,\fn{k}_{m_\dagger,\fn{k}_\dagger},\fn{s}_{m_\dagger,\fn{r}_\dagger}),r_\dagger,p_0\}$.
\end{itemize}
  \end{quotationb*}

By $n_{/p}$-metastable uniform continuity we obtain $D_0,m_0,p_0,\fn{r}_0$.  Set $q_0=\fn{q}_0(D_0,m_0,p_0,\fn{r}_0)$ and $r_0=\fn{r}_0(m_\dagger,q_0)$.  Since $q_0\geq p_0$, we have $r_0\geq q_0$ and when $\mu(\mathcal{B})<1/D_0$, $|\rho_{m_\dagger}\lambda_{r_0}|(\mathcal{B})<1/E_\flat$.

We set $m_\sharp=m_\dagger$, $q_\sharp=q_\dagger$, $\fn{k}_\sharp=\fn{k}_{m_\dagger,\fn{k}_\dagger}$, and $\fn{s}_\sharp=\fn{s}_{m_\dagger,\fn{r}_\dagger}$.  Set $r_\flat=\fn{r}_\flat(m_\sharp,q_\sharp,\fn{k}_\sharp,\fn{s}_\sharp)$, $\fn{l}_\flat=\fn{L}_\flat(m_\sharp,q_\sharp,\fn{k}_\sharp,\fn{s}_\sharp)$, $s_\sharp=\fn{s}_\sharp(\fn{l}_\flat,r_\flat)$, $k_\sharp=\fn{k}_\sharp(\fn{l}_\flat,r_\flat)$, and $l_\flat=\fn{l}_\flat(k_\sharp,s_\sharp)$.

Observe that
\begin{itemize}
\item $r_\flat=\fn{r}_\flat(m_\dagger,q_\dagger,\fn{k}_{m_\dagger,\fn{k}_\dagger},\fn{s}_{m_\dagger,\fn{r}_\dagger})$,
\item $\fn{l}_\dagger=\fn{l}_{m_\dagger,\fn{l}_\flat,r_\flat}$,
\item $\fn{s}_\dagger=\fn{s}_{m_\dagger,r_\flat}$,
\item $r_\dagger=\fn{r}_\dagger(\fn{l}_\dagger,\fn{s}_\dagger)=\fn{r}_\dagger(\fn{l}_{m_\dagger,\fn{l}_\flat,r_\flat},\fn{s}_{m_\dagger,r_\flat})$,
\item $q_0=\max\{r_\flat,r_\dagger,p_0\}$,
\item $s_\sharp=\fn{s}_{m_\dagger,\fn{r}_\dagger}(\fn{l}_\flat,r_\flat)=\fn{r}_0(m_\dagger,\max\{r_\flat,r_\dagger,p_0\})=\fn{r}_0(m_\dagger,q_0)=\fn{s}_{m_\dagger,r_\flat}(k_\dagger,r_\dagger)=\fn{s}_\dagger(k_\dagger,r_\dagger)$,
\item $k_\sharp=\fn{k}_{m_\dagger,\fn{k}_\dagger}(\fn{l}_\flat,r_\flat)=\fn{k}_\dagger(\fn{l}_\dagger,\fn{s}_\dagger)=k_\dagger$,
\item $l_\flat=\fn{l}_\flat(k_\sharp,s_\sharp)=\fn{l}_\flat(k_\dagger,\fn{r}_0(m_\dagger,q_0))=\fn{l}_{m_\dagger,\fn{l}_\flat,r_\flat}(k_\dagger,r_\dagger)=\fn{l}_\dagger(k_\dagger,r_\dagger)$.
\end{itemize}

We have $k_\sharp\geq m_\sharp$.  Suppose $r_\flat\geq q_\sharp$; then $s_\sharp=\fn{r}_0(m_\dagger,q_0)\geq q_0\geq r_\flat$.  Let $l\in[k_\sharp,l_\flat]$ be given.  Since $l\in[k_\dagger,\fn{l}_\dagger(k_\dagger,r_\dagger)]$ and $s_\sharp\in[r_\dagger,\fn{s}_\dagger(k_\dagger,r_\dagger)]$, there are sets $\mathcal{B}^-, \mathcal{B}^{0,-}, \mathcal{B}^{1,-}$ with $\mu(\mathcal{B}^-)<4/D_0$, $\mu(\mathcal{B}^{0,-})<2/D_0$, and $\mu(\mathcal{B}^{1,-})<2/D_0$ so that
\[\left|(\rho_{m_\sharp}\lambda_{s_\sharp})(\Omega)-(\rho_l\lambda_{q_\sharp})(\Omega)\right|\leq|\rho_{m_\sharp}\lambda_{s_\sharp}|(\mathcal{B}^-)+|\rho_l\lambda_{q_\sharp}|(\mathcal{B}^{-})+|\rho_{m_\sharp}\lambda_{s_\sharp}|(\mathcal{B}^{0,-})+|\rho_l\lambda_{q_\sharp}|(\mathcal{B}^{1,-})+\frac{8}{E_\flat}.\]

Since $m_\sharp\in[m_0,\fn{m}_0(D_0,m_0,p_0,\fn{r}_0)]$, $s_\sharp=\fn{r}_0(m_\dagger,q_0)$, $\mu(\mathcal{B}^-)<4/D_0$ and $\mu(\mathcal{B}^{0,-})<2/D_0$, we have $|\rho_{m_\sharp}\lambda_{s_\sharp}|(\mathcal{B}^-)+|\rho_{m_\sharp}\lambda_{s_\sharp}|(\mathcal{B}^{0,-})<12/E_\flat$, so 
\[\left|(\rho_{m_\sharp}\lambda_{s_\sharp})(\Omega)-(\rho_l\lambda_{q_\sharp})(\Omega)\right|\leq|\rho_l\lambda_{q_\sharp}|(\mathcal{B}^{-})+|\rho_l\lambda_{q_\sharp}|(\mathcal{B}^{1,-})+\frac{20}{E_\flat}.\]
\end{proof}

One more application of the same technique eliminates the last extraneous term in the bound, giving the desired result.

\begin{theorem}\label{thm:control_interval_E2}
Suppose $(*)$ holds.  Then for every $E_\flat, p_\flat, n_\flat, \fn{k}_\flat, \fn{r}_\flat$ there are $m_\sharp\geq n_\flat$, $q_\sharp\geq p_\flat$, $\fn{l}_\sharp$, $\fn{s}_\sharp$ so that, setting
\begin{itemize}
\item $k_\flat=\fn{k}_\flat(m_\sharp,q_\sharp,\fn{l}_\sharp,\fn{s}_\sharp)$,
\item $r_\flat=\fn{r}_\flat(m_\sharp,q_\sharp,\fn{l}_\sharp,\fn{s}_\sharp)$,
\item $l_\sharp=\fn{l}_\sharp(k_\flat,r_\flat)$,
\item $s_\sharp=\fn{s}_\sharp(k_\flat,r_\flat)$,
\end{itemize}
we have
\[\left|(\rho_{m_\sharp}\lambda_{s_\sharp})(\Omega)-(\rho_{l_\sharp}\lambda_{q_\sharp})(\Omega)\right|\leq\frac{32}{E_\flat}.\]
\end{theorem}
\begin{proof}
    Towards the use of $p_{/n}$-metastable uniform continuity, we define functions $\fn{q}_0(D_0,q_0,n_0,\fn{k}_0)$ and $\fn{m}_0(D_0,q_0,n_0,\fn{k}_0)$.

    \begin{quotationb*}
      Let $D_0, q_0, n_0, \fn{k}_0$, which we abbreviate $\dagger$, be given.  In order to apply Lemma \ref{thm:control_interval_E1} we define
      \begin{itemize}
      \item $\fn{l}_{\star,q,k}(k',r)=\fn{k}_0(q,\max\{k,k'\})$,
      \item $\fn{l}_{\star,q,\fn{k}}(k,r)=\fn{k}_0(q,\max\{k,\fn{k}(r,\fn{l}_{\star,q,k})\})$,
      \item $\fn{s}_{\star,q,\fn{k},\fn{s}}(k,r)=\fn{s}(r,\fn{l}_{\star,q,\fn{k}})$,
      \item $\fn{r}_\dagger(m_\dagger, q_\dagger, \fn{s}_\dagger, \fn{k}_\dagger)=\fn{r}_\flat(m_\dagger,q_\dagger,\fn{s}_{\star,q_\dagger,\fn{k}_\dagger,\fn{s}_\dagger},\fn{l}_{\star,q_\dagger,\fn{k}_\dagger})$,
      \item $\fn{L}_\dagger(m_\dagger, q_\dagger, \fn{s}_\dagger, \fn{k}_\dagger)=\fn{l}_{\star,q_\dagger,\fn{k}_\flat(m_\dagger,q_\dagger,\fn{s}_{\star,q,\fn{k}_\dagger,\fn{s}_\dagger},\fn{l}_{\star,q_\dagger,\fn{k}_\dagger})}$,
      \end{itemize}

By Lemma \ref{thm:control_interval_E1} we obtain $m_\dagger, q_\dagger, \fn{s}_\dagger, \fn{k}_\dagger$.  It is convenient to define
\begin{itemize}
\item $\fn{s}_\star=\fn{s}_{\star,q_\dagger,\fn{k}_\dagger,\fn{s}_\dagger}$,
\item $\fn{l}_\star=\fn{l}_{\star,q_\dagger,\fn{k}_\dagger}$,
\item $r_\dagger=\fn{r}_\flat(m_\dagger,q_\dagger,\fn{s}_\star,\fn{l}_\star)$,
\item $\fn{l}_\dagger=\fn{l}_{\star,q_\dagger,\fn{k}_\flat(m_\dagger,q_\dagger,\fn{s}_\star,\fn{l}_\star)}$,
\item $s_\dagger=\fn{s}_\dagger(r_\dagger,\fn{l}_\dagger)$,
\item $k_\dagger=\fn{k}_\dagger(r_\dagger,\fn{l}_\dagger)$.
\end{itemize}

We may then set
\begin{itemize}
\item $\fn{q}_0(D_0,q_0,n_0,\fn{k}_0)=q_\dagger$, and
\item $\fn{m}_0(D_0,q_0,n_0,\fn{k}_0)=\max\{\fn{k}_\flat(m_\dagger,q_\dagger,\fn{s}_\star,\fn{l}_\star),k_\dagger\}$.
\end{itemize}
    \end{quotationb*}
By metastable uniform continuity we obtain $D_0, q_0, n_0, \fn{k}_0$.  Set $m_0=\fn{m}_0(D_0,q_0,n_0,\fn{k}_0)$ and $k_0=\fn{k}_0(q_\dagger,m_0)$.  Since $m_0\geq n_0$, we have $k_0\geq m_0$ and when $\mu(\mathcal{B}')<1/D_0$, $|\rho_{k_0}\lambda_{q_\dagger}|(\mathcal{B}')<1/E_\flat$.

We set $m_\sharp=m_\dagger, q_\sharp=q_\dagger$, $\fn{l}_\sharp=\fn{l}_\star$, $\fn{s}_\sharp=\fn{s}_\star$.  Let $k_\flat, r_\flat, l_\sharp, s_\sharp$ be as in the statement.  Then since $l_\sharp=\fn{l}_\star(k_\flat, r_\flat)=\fn{k}_0(q_\dagger,\max\{k_\flat,k_\dagger\})=\fn{l}_{\star,k_\flat}(k_\dagger,s_\dagger)=\fn{l}_\dagger(k_\dagger,s_\dagger)$ and $l_\sharp\geq k_\dagger$, we have 
\[\left|(\rho_{m_\sharp}\lambda_{s_\sharp})(\Omega)-(\rho_{l_\sharp}\lambda_{q_\sharp})(\Omega)\right|\leq|\rho_{l_\sharp}\lambda_{q_\sharp}|(\mathcal{B}^-)+|\rho_{l_\sharp}\lambda_{q_\sharp}|(\mathcal{B}^{1,-})+\frac{20}{E_\flat}\]
where $\mu(\mathcal{B}^-)<4/D_\flat$ and $\mu(\mathcal{B}^{1,-})<2/D_\flat$.  Therefore $|\rho_{l_\sharp}\lambda_{q_\sharp}|(\mathcal{B}^-)+|\rho_{l_\sharp}\lambda_{q_\sharp}|(\mathcal{B}^{1,-})<12/E_\flat$, and so
\[\left|(\rho_{m_\sharp}\lambda_{s_\sharp})(\Omega)-(\rho_{l_\sharp}\lambda_{q_\sharp})(\Omega)\right|\leq\frac{32}{E_\flat}\]
\end{proof}

\section{Quantitative Bounds}\label{sec:bounds}

In this section we work out a concrete case of the bounds given by the work above, in essentially the simplest non-trivial case.  This illustrates just how large the bounds above get, and is the result needed in \cite{towsner:banach}.

The simplest meaningful case of \ref{thm:control_interval_E2} is to swap the order of the indices---that is, to show that, for every $\epsilon$, there exist $s>m$ and $l>q$ so that
\[|(\rho_m\lambda_s)(\Omega)-(\rho_l\lambda_q)(\Omega)|<\epsilon.\]
Our goal in this section is to obtain a bound on $l$ and $s$ under some assumptions about the sequences $\rho_n$ and $\lambda_p$.

Throughout this section, we make the following assumptions (we use $\nu_d$ to stand in for either some $\rho_n$ or $\lambda_p$), which are essentially quantitative versions of the assumptions $(\ast)$ above:
\begin{itemize}
\item[$(\ast)^Q_1$] Each sequence $(\rho_n\lambda_p)_p$ (for fixed $n$) and $(\rho_n\lambda_p)_n$ (for fixed $p$) has bounded fluctuations with the uniform bound $8 B^2E^2$,
\item[$(\ast)^Q_2$] For all $\rho_m$, $\omega_{\rho_m}(E)\leq E2^m$ and for all $\lambda_p$, $\omega_{\lambda_p}(E)\leq E2^p$,
\item[$(\ast)^Q_3$] $||\nu_d||_{L^1}\leq B$ for all $\nu_d$,
\item[$(\ast)^Q_4$] For any $E,D$, any $\nu_d$, and any $\mathcal{B}$, there is a $\mathcal{B}'\succeq\mathcal{B}$ and a $\mathcal{B}''\subseteq\mathcal{B}$ such that $\nu_d$ is $E$-constant on $\mathcal{B}'$, $\mu(\mathcal{B}'')<1/D$, and $|\mathcal{B}'|\leq 2BDE|\mathcal{B}|$.
\end{itemize}
The last condition is exactly what we would expect if $\nu_d$ were the Radon-Nikodym derivative of an actual $L^1$ function---we just take $\mathcal{B}''$ to consist of the points where the underlying function has large absolute value, and $\mathcal{B}'$ to consist of approximate level sets.  We refer to these four assumptions collectively as $(\ast)^Q$.

We will say that a function $f(x)$ is:
\begin{itemize}
\item Polynomial if there are $a,b,c$ so that $f(x)\leq ax^b+c$ for all $x$,
\item Exponential if there are $a,b,c,d$ so that $f(x)\leq a^{bx^c+d}$ for all $x$,
\item Double exponential if there are $a,b,c,d$ so that $f(x)\leq a^{a^{bx^c+d}}$ for all $x$.
\end{itemize}
We say a function of multiple inputs, $f(x_1,\ldots,x_n)$ is polynomial (resp. exponential, double exponential) if there is a polynomial (resp. exponential, double exponential) function $f'(x)$ so that for all $x_1,\ldots,x_n$, $f(x_1,\ldots,x_n)\leq f'(\prod_i x_i)$.

We will briefly need to keep track of functions which are polynomial in some inputs and exponential in others; we say $f(x;y)$ is poly-exp if there are $a,b,c,d$ so that for all $x,y$, $f(x;y)\leq (ax)^{by^c+d}$.  We say a function of multiple inputs, $f(x_1,\ldots,x_n;y_1,\ldots,y_m)$ is poly-exp if there is a poly-exp function $f'(x;y)$ so that for all $x_1,\ldots,x_n, y_1,\ldots,y_m$, $f(x_1,\ldots,x_n;y_1,\ldots,y_m)\leq f'(\prod_i x_i;\prod_j y_j)$.

\begin{lemma}[Bounds for Lemma \ref{thm:quant_convergence_partition}]
Suppose $(\ast)^Q$ holds.  Then there is an exponential function $\mathfrak{a}_0(B,D_\flat,E_\flat)$ so that for every $D_\flat,E_\flat, \mathcal{B}_\flat, \fn{m}_\flat, n_\flat$ there is an $m_\sharp\in[n_\flat, \fn{m}_\flat^{\mathfrak{a}_0(B,D_\flat,E_\flat)}(n_\flat)]$ such that, taking
\[\mathcal{B}=\{\sigma\in\mathcal{B}_\flat\mid \text{for every }m,m'\in[m_\sharp,\widehat{\mathbf{m}}_\flat(m_\sharp)]\text{, }|\nu_m-\nu_{m'}|(\sigma)< 1/E_\flat\},\]
we have $\mu(\mathcal{B})\geq (1-/D_\flat) \mu(\mathcal{B}_\flat)$.
\end{lemma}
\begin{proof}
Set $\mathfrak{a}_0(B,D_\flat,E_\flat)=[2^5B^2E_\flat^2]^{\frac{-\ln D_\flat}{\ln(1-\frac{1}{4B^2E_\flat^2})}}$.

  In the proof, $V_\sharp=2^5B^2E^2_\flat$ and 
\[k=\lceil\frac{\ln(1/D_\flat)}{\ln(1-1/4B^2E^2_\flat)}\rceil.\]

One can check that $(2^5B^2E^2_\flat)^{-1/\ln(1-1/4B^2E^2_\flat)}$ grows less quickly than $e^{(2^5B^2E^2_\flat)^2}$, so we can bound $\mathfrak{a}_0(B,D_\flat,E_\flat)$ by $2^{2^{10}B^4E^4_\flat\ln D_\flat}$, which has the specified bounds.
\end{proof}

We will often need the quantity $\mathfrak{a}_0(B,3D_\flat,3E_\flat)+1$, which we abbreviate $\mathfrak{a}(B,D_\flat,E_\flat)$. $\mathfrak{a}$ is also exponential.

\subsection{Bounds on Regularity}

The next few lemmas give bounds on (cases of) the various forms of the the one-dimensional regularity lemma, culminating in bounds on a special case of Lemma \ref{thm:1D_sequential_interval}.  

We first note the function
\begin{itemize}
\item $\mathfrak{b}_0(B,D_\flat,E_\flat)=2^{10}D^3_\flat E^2_\flat B^2+2^9 D^2_\flat E^2_\flat B^2$.
\end{itemize}
which appears in the proof of Lemma \ref{thm:1D_sequential}; this is essentially the bound on the number of iterations needed in that argument.  It turns out we will mostly need
\begin{itemize}
\item $\mathfrak{b}(B,D_\flat,E_\flat)=\mathfrak{b}(B,3D_\flat,3E_\flat)$.
\end{itemize}
Both of these functions are polynomial.

There are two major simplifications in this special case which will make it much easier to find bounds on the various sequential versions.  The first is that we will fix, throughout this subsection, a constant $z_0$ so that all the functions $\fn{B}_\flat$, $\fn{B}_\flat^0$, and $\fn{B}_\flat^1$ will satisfy the bound $|\fn{B}_\flat(\cdots,\mathcal{B})|\leq z_0|\mathcal{B}|$, independently of other parameters.  Along with this, we will assume that we always apply our lemmas with $|\mathcal{A}_\flat|\leq z_0^{(\mathfrak{b}(B,D_\flat,E_\flat)+1)(\mathfrak{b}(B,D_\flat,E_\flat))}$, and that all partitions $\mathcal{B}$ appearing in the proofs in this subsection satisfy $|\mathcal{B}|\leq z_0^{(\mathfrak{b}(B,D_\flat,E_\flat)+2)\mathfrak{b}(B,D_\flat,E_\flat)}$.

The second simplification is that we have a fixed function $\fn{u}:\mathbb{N}\rightarrow\mathbb{N}$, and all the functions we deal with will be bounded by, roughly, iterations of $\fn{u}$ above some base value.  Abstractly, we say:
\begin{itemize}
\item A function $\fn{m}(n)$ is $\fn{u}$-bounded by $c$ above $n_-$ if $\fn{m}(n)\leq\fn{u}^c(\max\{n,n_-\})$,
\item A function $\fn{m}(n,\fn{k})$ is $\fn{u}$-bounded by $f$ above $n_-$ if whenever $\fn{k}$ is $\fn{u}$-bounded by $c$ above $n_-$, $\fn{m}(n,\fn{k})\leq\fn{u}^{f(c)}(\max\{n,n_-\})$.
\end{itemize}
Note that this second clause allows us to iterate if we interpret $c$ as being itself a function and $f$ as a higher order functional.  We try to limit the quantity of special cases we need to deal with, but it will be helpful to generalize the second clause to the case where there are two input functions:
\begin{itemize}
\item A function $\fn{m}(n,\fn{k},\fn{r})$ is $\fn{u}$-bounded by $f$ above $n_-$ if whenever $\fn{k}$ is $\fn{u}$-bounded by $c$ above $n_-$ and $\fn{r}$ is $\fn{u}$-bounded by $d$ above $n_-$, $\fn{m}(n,\fn{k},\fn{r})\leq\fn{u}^{f(c,d)}(\max\{n,n_-\})$.
\end{itemize}

For consistency with the later arguments, we find bounds for Lemma \ref{thm:1D_sequential} with the values $3D_\flat, 3E_\flat$.

\begin{lemma}[Bounds on \ref{thm:1D_sequential}]
Suppose that $(*)^Q$ holds.  There exist poly-exp functions $\mathfrak{c}(a;B,D_\flat,E_\flat)$ and $\mathfrak{d}(a;B,D_\flat,E_\flat)$ so that whenever $z_0, n_-, a, b, \mathcal{A}_\flat, E_\flat, D_\flat, \fn{m}_\flat$, and $\fn{B}_\flat$ are given such that:
\begin{itemize}
\item $|\mathcal{A}_\flat|\leq z_0^{(\mathfrak{b}(B,D_\flat,E_\flat)+1)(\mathfrak{b}(B,D_\flat,E_\flat))}$,
\item For all $\fn{m},\mathcal{B}$, $|\fn{B}_\flat(\fn{m},\mathcal{B})|\leq z_0|\mathcal{B}|$,
\item The function $(n,\fn{k})\mapsto\sup_{\mathcal{B}_0,\mathcal{B}_1,\mathcal{B}_2}\fn{m}_\flat(n,\fn{k},\mathcal{B}_0,\mathcal{B}_1,\mathcal{B}_1)$ is $\fn{u}$-bounded by $\fn{f}(x)=ax+b$ above $n_-$ (where the supremum ranges over partitions refining $\mathcal{A}_\flat$ and satisfying $|\mathcal{B}_i|\leq z_0^{(\mathfrak{b}(B,D_\flat,E_\flat)+2)(\mathfrak{b}(B,D_\flat,E_\flat))}$),
\end{itemize}
there exist $\mathcal{B}_\sharp\succeq\mathcal{A}_\flat$, $n_\sharp$, and $\fn{k}_\sharp$ such that:
\begin{itemize}
\item $|\mathcal{B}_\sharp|\leq z_0^{\mathfrak{b}(B,D_\flat,E_\flat)}|\mathcal{A}_\flat|$, 
\item $n_\sharp\leq \fn{u}^{\mathfrak{d}(a;B,D_\flat,E_\flat)b}(n_-)$,
\item The function $n\mapsto\sup_{\mathcal{B}}\fn{k}_\sharp(n,\mathcal{B})$ is $\fn{u}$-bounded by $\mathfrak{c}(a;B,D_\flat,E_\flat)b$ above $n_\sharp$ (with the supremum ranging over the same partitions as above),
\item Whenever $\mathcal{B}_\sharp\preceq\mathcal{B}\preceq\mathcal{B}'\preceq\fn{B}_\flat(\fn{m}_\sharp,\mathcal{B}_\sharp)$, setting $m_\flat=\fn{m}_\flat(n_\sharp,\fn{k}_\sharp,\mathcal{B}_\sharp,\mathcal{B},\mathcal{B}')$ and $k_\sharp=\fn{k}_\sharp(m_\flat,\mathcal{B}')$, if $m_\flat\geq n_\sharp$ then we have $k_\sharp\geq m_\flat$ and $\mu(\mathfrak{D}_{3E_\flat,\mathcal{B},k_\sharp}(\mathcal{B}'))<1/3D_\flat$.
\end{itemize}
\end{lemma}
\begin{proof}
If we ignore the bounds, this is essentially Lemma \ref{thm:1D_sequential} applied to $\mathcal{A}_\flat, 3E_\flat, 3D_\flat, \fn{m}_\flat,\fn{B}_\flat$.  We obtain bounds by examining the proof of Lemma \ref{thm:1D_sequential}.  

The main step in the proof is the construction of the sequence of functions $\fn{k}_i$, with $\fn{k}_\sharp$ bounded by $\fn{k}_{\mathfrak{b}(B,D_\flat,E_\flat)}$.  First, note that when $\mathcal{A}_0=\mathcal{A}_\flat$, each element in the sequence of partitions constructed in the proof satisfies $\mathcal{A}_{i+1}\preceq\fn{B}_\flat(m_0,\fn{k}_i,\mathcal{A}_i)$, so we have $|\mathcal{A}_i|\leq z_0^i|\mathcal{A}_\flat|$, and so $|\mathcal{B}_\sharp|\leq z_0^{\mathfrak{b}(B,D_\flat,E_\flat)}|\mathcal{A}_\flat|$.

We show inductively that $\fn{k}_i(m,\mathcal{B})$ is $\fn{u}$-bounded by $(a+2)^ib$ above $n_-$.

Clearly $\fn{k}_0(m,\mathcal{B})=m=\fn{u}^0(m)$.  When $m_0\geq n_-$, we have
\begin{align*}
\fn{k}_{i+1}(m_0,\mathcal{B})
&=\fn{k}_i(\fn{m}_\flat(m_0,\fn{k}_i,\mathcal{A}_d,\mathcal{B},\mathcal{B}'))\\
&\leq \fn{u}^{(a+2)^ib}(\fn{u}^{a(a+2)^ib+b}(m_0))\\
&=\fn{u}^{(a+2)^ib+a(a+2)^ib+b}(m_0)\\
&=\fn{u}^{(a+2)^ib(a+1)+b}(m_0)\\
&\leq\fn{u}^{(a+2)^{i+1}b}(m_0).\\
\end{align*}

It suffices, in the last step, to work with $i=\mathfrak{b}(B,D_\flat,E_\flat)$ and $m_0=n_-$, so $\fn{k}_\sharp=\fn{k}_{\mathfrak{b}(D_\flat,E_\flat)}$ is $\fn{u}$-bounded by $(a+2)^{\mathfrak{b}(B,D_\flat,E_\flat)}b$ above $n_-$; we see that this bound is poly-exp of the right form.  For the largest value that might be used for $n_\sharp$, take the values $n_0=n_-$, $n_{i+1}=\fn{m}_\flat(n_i,\fn{k}_{\mathfrak{b}(B,D_\flat,3E_\flat)-i},\mathcal{A}_d,\mathcal{B},\mathcal{B}')$, and observe that $n_\sharp=n_i$ for some $i$.  Since $n_i\leq \fn{u}^{\sum_{j\leq i}a(a+2)^{\mathfrak{b}(B,D_\flat,3E_\flat)-i}b+b}(n_-)$, we see that the bound $\mathfrak{d}(a,b;B,D_\flat,E_\flat)=\sum_{j\leq i}a(a+2)^{\mathfrak{b}(B,D_\flat,3E_\flat)-i}b+b$ also has the promised size.
\end{proof}

\begin{lemma}[Bounds on \ref{thm:1D_sequential_interval}]
Suppose $(*)^Q$ holds.  Then there exists a poly-exp function $\mathfrak{e}(a;B,D_\flat,E_\flat)$ so that whenever $z_0, n_-,\mathcal{A}_\flat, E_\flat, D_\flat, \fn{B}_\flat, \fn{n}_\flat, \fn{L}_\flat$ are given such that:
\begin{itemize}
\item $|\mathcal{A}_\flat|\leq z_0^{(\mathfrak{b}(B,D_\flat,E_\flat)+1)(\mathfrak{b}(B,D_\flat,E_\flat))}$,
\item For all $\fn{m},\mathcal{B}$, $|\fn{B}_\flat(\fn{m},\mathcal{B})|\leq z_0|\mathcal{B}|$,
\item For all $n,\fn{k}$ and all $\mathcal{B}\succeq\mathcal{A}_\flat$ with $|\mathcal{B}|\leq z_0^{\mathfrak{b}(B,D_\flat,E_\flat)}|\mathcal{A}_\flat|$, the function $k\mapsto\fn{L}_\flat(n,\fn{k},\mathcal{B})(k)$ is $\fn{u}$-bounded by $1$ above $\max\{n_-,n,\fn{m}_\flat(n,\fn{k},\mathcal{B})\}$,
\item The function $(n,\fn{k})\mapsto\sup_{\mathcal{B}}\fn{m}_\flat(n,\fn{k},\mathcal{B})$ is $\fn{u}$-bounded by $\fn{f}(f)=af(1)+b$ above $n_-$ (where the supremum is over suitable $\mathcal{B}$ as above),
\end{itemize}
there are $n_\sharp$, $\mathcal{B}_\sharp$ and $\fn{m}_\sharp$ such that:
\begin{itemize}
\item $|\mathcal{B}_\sharp|\leq z_0^{\mathfrak{b}(B,D_\flat,E_\flat)}|\mathcal{A}_\flat|$,
\item $n_\sharp\leq \fn{u}^{\mathfrak{e}(a;B,D_\flat,E_\flat)b}(n_-)$,
\item $\fn{k}_\sharp$ is $\fn{u}$-bounded by $f(x)=\mathfrak{a}(B,D_\flat,E_\flat)x+\mathfrak{e}(a;B,D_\flat,E_\flat)b$ above $n_\sharp$,
\item Setting $m_\flat=\fn{m}_\flat(n_\sharp,\fn{k}_\sharp,\mathcal{B}_\sharp)$ and $\fn{l}_\flat=\fn{L}_\flat(n,\fn{m},\mathcal{B})$, if $m_\flat\geq n_\sharp$ then for any $\mathcal{B}_\sharp\preceq\mathcal{B}\preceq\mathcal{B}'\preceq\fn{B}_\flat(n_\sharp,\fn{k}_\sharp,\mathcal{B}_\sharp)$, setting $k_\sharp=\fn{k}_\sharp(m_\flat,\fn{l}_\flat,\mathcal{B},\mathcal{B}')$, we have $k_\sharp\geq m_\flat$ and for every $l\in[k_\sharp,\fn{l}_\flat(k_\sharp)]$, $\mu(\mathfrak{D}_{E_\flat,\mathcal{B},k}(\mathcal{B}'))<1/D_\flat$.
\end{itemize}
\end{lemma}
\begin{proof}
We examine the proof of Lemma \ref{thm:1D_sequential_interval}.  Suppose we have been given $n_0,\mathcal{B}_0$, and a $\fn{k}_0$ which is $\fn{u}$-bounded by $y$ above $n_0$.  

Suppose we have also fixed $\fn{l}_\flat$ which is $\fn{u}$-bounded by $x$ above $\max\{n_0,n_-\}$ and a $\mathcal{B}'\succeq\mathcal{B}\succeq\mathcal{B}_0$ with $|\mathcal{B}'|\leq z_0^{\mathfrak{b}(B,D_\flat,E_\flat)}|\mathcal{A}_\flat|$.  The function $m\mapsto\fn{l}_\flat(\fn{k}_0(m,\mathcal{B}'))$ is $\fn{u}$-bounded by $y+x$ above $\max\{n_0,n_-\}$.  The function
\[m\mapsto m_{\dagger,m,\fn{l}_\flat,\mathcal{B},\mathcal{B}'}\]
is obtained by applying Lemma \ref{thm:quant_convergence_partition} with $3D_\flat, 3E_\flat$, and is therefore $\fn{u}$-bounded by $\mathfrak{a}_0(B,3D_\flat,3E_\flat)(y+x)$ above $\max\{n_0,n_-\}$.  

The function $(m,\fn{l}_\flat)\mapsto m_{\dagger,m,\fn{l}_\flat,\mathcal{B},\mathcal{B}'}$ is therefore $\fn{u}$-bounded by $f_y(x)=\mathfrak{a}_0(B,3D_\flat,3E_\flat)(y+x)$ above $\max\{n_0,n_-\}$, and so 
\[(m,\fn{l}_\flat)\mapsto\sup_{\mathcal{B},\mathcal{B}'}\fn{k}_\dagger(m,\fn{l}_\flat,\mathcal{B},\mathcal{B}')\]
is bounded by $f'_y(x)=\mathfrak{a}_0(B,3D_\flat,3E_\flat)(y+x)+y+x=\mathfrak{a}(B,D_\flat,E_\flat)(y+x)$.

The function $(n_0,\fn{k}_0)\mapsto\sup_{\mathcal{B}}\fn{m}_\flat(n_0,\fn{k}_\dagger,\mathcal{B}_0)$ is therefore $\fn{u}$-bounded by $af'_y(1)+b=\mathfrak{a}(B,D_\flat,E_\flat)a(y+1)+b$ above $n_-$.  This at last lets us consider the function $\fn{m}_0$: it is $\fn{u}$-bounded by
\[f''(y)=\mathfrak{a}(B,D_\flat,E_\flat)a(y+1)+b+\mathfrak{a}_0(B,3D_\flat,3E_\flat)(y+1)\leq \mathfrak{a}(B,D_\flat,E_\flat)(a+1)(y+1)+b\]
above $n_-$.

We are in the setting of the previous lemma where $(n,\fn{k})\mapsto\sup_{\mathcal{B}_0,\mathcal{B}_1,\mathcal{B}_2}\fn{m}_0(n,\fn{k},\mathcal{B}_0,\mathcal{B}_1,\mathcal{B}_1)$ is $\fn{u}$-bounded by $f''(y)$; rewriting $f''(y)$ as a linear function,
\[f''(y)=\mathfrak{a}(B,D_\flat,E_\flat)(a+1)y+\mathfrak{a}(B,D_\flat,E_\flat)(a+1)+b.\]
 Therefore $\fn{k}_0$ is $\fn{u}$-bounded by
\[\mathfrak{c}(\mathfrak{a}(B,D_\flat,E_\flat)(a+1);B,D_\flat, E_\flat)(\mathfrak{a}(B,D_\flat,E_\flat)(a+1)+b)\]
and
\[n_0\leq \fn{u}^{\mathfrak{d}(\mathfrak{a}(B,D_\flat,E_\flat)(a+1);D_\flat, E_\flat)(\mathfrak{a}(B,D_\flat,E_\flat)(a+1)+b)}(n_-).\]

The function $\fn{k}_\sharp$ is therefore bounded by
\[f'_{\mathfrak{c}(\mathfrak{a}(B,D_\flat,E_\flat)(a+1);D_\flat, E_\flat)(\mathfrak{a}(B,D_\flat,E_\flat)(a+1)+b)}(x),\]
 which has the specified form (in particular, it is linear in $b$, polynomial in $a$ and exponential in $B,D_\flat,E_\flat$).

Similarly $n_\sharp=n_0$ is bounded by $\fn{u}^{\mathfrak{d}(\mathfrak{a}(B,D_\flat,E_\flat)(a+1);D_\flat, E_\flat)(\mathfrak{a}(B,D_\flat,E_\flat)(a+1)+b)}(n_-)$, which also has the specified bounds.
\end{proof}

\subsection{Controlling Intervals}

In this subsection we obtain bounds on a special case of Lemma \ref{thm:control_interval}.  We continue to work with a fixed function $\fn{u}$.

\begin{lemma}[Bounds on \ref{thm:1D_sequential_interval_double}]
Suppose $(*)^Q$ holds.  There are poly-exp functions $\mathfrak{f}(\rho;B,D_\flat,E_\flat)$ and $\mathfrak{g}(\rho;B,D_\flat,E_\flat)$ so that whenever $z, \rho, n_-, D_\flat, E_\flat, \fn{B}^0_\flat, \fn{B}^1_\flat, \fn{m}_\flat, \fn{L}_\flat, \fn{q}_\flat$, and $\fn{S}_\flat$ are given such that:
\begin{itemize}
\item For all $n,p,\mathcal{B},\fn{k},\fn{r}$ and $i\in\{0,1\}$, $|\fn{B}^i_\flat(n,p,\fn{k},\fn{r},\mathcal{B})|\leq z|\mathcal{B}|$,
\item $\fn{m}_\flat(n,p,\fn{k},\fn{r},\mathcal{B})$ is $\fn{u}$-bounded by $\rho( f(1,1)+ g(1,1)+1)$ above $n_-$,
\item $\fn{q}_\flat(n,p,\fn{k},\fn{r},\mathcal{B})$ is $\fn{u}$-bounded by $\rho( f(1,1)+ g(1,1)+1)$ above $n_-$,
\item For any $n,p,\mathcal{B},\fn{k},\fn{r}$, the function $m\mapsto\fn{L}_\flat(n,p,\fn{k},\fn{r},\mathcal{B})(m,q)$ does not depend on $q$ and is $\fn{u}$-bounded by $1$ above $\max\{n_-,n,p, \fn{m}_\flat(n,p,\fn{k},\fn{r},\mathcal{B}),\allowbreak \fn{q}_\flat(n,p,\fn{k},\fn{r},\mathcal{B})+1\}$,
\item For any $n,p,\mathcal{B},\fn{k},\fn{r}$, the function $q\mapsto\fn{S}_\flat(n,p,\fn{k},\fn{r},\mathcal{B})(m,q)$ does not depend on $m$ and is $\fn{u}$-bounded by $1$ above $\max\{n_-,n,p,\fn{m}_\flat(n,p,\fn{k},\fn{r},\mathcal{B})+1,\fn{q}_\flat(n,p,\fn{k},\fn{r},\mathcal{B})\}$,
\end{itemize}
then there exist $\mathcal{B}_\sharp$, $n_\sharp$, $p_\sharp$, $\fn{k}_\sharp$, and $\fn{r}_\sharp$ such that:
\begin{itemize}
\item $|\mathcal{B}_0|\leq z^{\mathfrak{b}(B,D_\flat,E_\flat)(\mathfrak{b}(B,D_\flat,E_\flat)+2)}$,
\item $n_\sharp\leq\fn{u}^{\mathfrak{f}(\rho;B,D_\flat,E_\flat)}(n_-)$,
\item $p_\sharp\leq\fn{u}^{\mathfrak{g}(\rho;B,D_\flat,E_\flat)}(n_-)$,
\item $\fn{k}_\sharp$ is $\fn{u}$-bounded by $\mathfrak{a}(B,D_\flat,E_\flat)x+\mathfrak{f}(\rho;B,D_\flat,E_\flat)$ above $\max\{n_\sharp,p_\sharp\}$,
\item $\fn{r}_\sharp$ is $\fn{u}$-bounded by $\mathfrak{a}(B,D_\flat,E_\flat)y+\mathfrak{g}(\rho;B,D_\flat,E_\flat)$ above $\max\{n_\sharp,p_\sharp\}$,
\end{itemize}
\end{lemma}
\begin{proof}
We need to first analyze the inner application of Lemma \ref{thm:1D_sequential_interval}.  Suppose we have fixed $\mathcal{B}_0, n_0$, and $\fn{k}_0$ so that $|\mathcal{B}_0|\leq z^{\mathfrak{b}(B,D_\flat,E_\flat)(\mathfrak{b}(B,D_\flat,E_\flat)+1)}$ and $\fn{k}_0$ is $\fn{u}$-bounded by $h_1(x)$ above $\max\{n_-,n_0\}$.

Suppose we are given $\mathcal{B}_\dagger, p_\dagger, \fn{r}_\dagger$ with $|\mathcal{B}_\dagger|\leq z^{\mathfrak{b}(B,D_\flat,E_\flat)}|\mathcal{B}_0|$ and $\fn{r}_\dagger$ is $\fn{u}$-bounded by $h_2(y)$ above $\max\{n_-,n_0,p_\dagger\}$. 

We note bounds on the various helper functions under the assumption (as will turn out to be the case) that $\fn{l}(m,q)$ does not depend on $q$ and is $\fn{u}$-bounded by $1$ above $\max\{n_-,n_0,p_\dagger\}$, and that $\fn{s}(m,q)$ does not depend on $m$ and is $\fn{u}$-bounded by $1$ above $\max\{n_-,n_0,p_\dagger\}$.
\begin{itemize}
\item $\fn{l}_{\fn{l},r}$ does not depend on $r$ and is also $\fn{u}$-bounded by $1$ above $\max\{n_-,n_0,p_\dagger\}$,
\item $\fn{k}_{\star,r}$ does not depend on $q$ and is $\fn{u}$-bounded by $h_1(1)$ above $\max\{n_-,n_0,p_\dagger\}$,
\item $\fn{s}_{\star,\mathcal{B}^0,\mathcal{B}^1,m,q,\fn{l},\fn{s}}$ is also $\fn{u}$-bounded by $1$ above $\max\{n_-,n_0,p_\dagger\}$,
\item $\fn{r}_\star$ is $\fn{u}$-bounded by $h_2(y)$ above $\max\{n_-,n_0,p_\dagger\}$,
\item $\fn{k}_\star$ is $\fn{u}$-bounded by $h_1(x)$ above $\max\{n_-,n_0,p_\dagger\}$,
\item $\fn{l}_\star(m,q)$ does not depend on $q$ and is $\fn{u}$-bounded by $1$ above $\max\{n_-,n_0,p_\dagger\}$,
\item $\fn{s}_\star(m,q)$ does not depend on $m$ and is $\fn{u}$-bounded by $1$ above $\max\{n_-,n_0,p_\dagger\}$,
\item $|\mathcal{B}^0_\star|,|\mathcal{B}^1_\star|\leq z|\mathcal{B}|$,
\item $m_\star\leq\fn{u}^{\rho (h_1(1)+h_2(1)+1)}(\max\{n_-,n_0,p_\dagger\})$,
\item $q_\star\leq\fn{u}^{\rho (h_1(1)+h_2(1)+1)}(\max\{n_-,n_0,p_\dagger\})$.
\end{itemize}

We now see that $\fn{q}_\dagger$ is $\fn{u}$-bounded by $\mathfrak{g}_{h_1}(h_2)=\rho (h_1(1)+h_2(1)+1)$ above $\max\{n_-,n_0\}$ and 
\[r\mapsto \fn{S}_\dagger(p_\dagger,\mathcal{B}_\dagger,\fn{q}_\dagger)(r)\]
 is $\fn{u}$-bounded by $1$ above $\max\{n_-,n_0,p_\dagger\}$.  

When we apply the previous lemma to $z, \mathcal{B}_0, E_\flat, D_\flat, \fn{B}_\dagger, \fn{q}_\dagger, \fn{S}_\dagger$, we obtain $\mathcal{B}_\dagger, p_\dagger, \fn{r}_\dagger$ such that:
\begin{itemize}
\item $|\mathcal{B}_\dagger|\leq z^{\mathfrak{b}(B,D_\flat,E_\flat)}|\mathcal{B}_0|$,
\item $p_\dagger\leq \fn{u}^{\mathfrak{e}(\rho;B,D_\flat,E_\flat)\rho (h_1(1)+1)}(n_0)$,
\item $\fn{r}_\dagger$ is $\fn{u}$-bounded by $\mathfrak{a}(B,D_\flat,E_\flat)y+\mathfrak{e}(\rho;B,D_\flat,E_\flat)\rho (h_1(1)+1)$ above $p_\dagger$.
\end{itemize}

We now turn to the outer application of Lemma \ref{thm:1D_sequential_interval}.  $\fn{m}_0$ is $\fn{u}$-bounded by
\begin{align*}
&\rho(h_1(1)+\mathfrak{a}(B,D_\flat,E_\flat)1+\mathfrak{e}(\rho;B,D_\flat,E_\flat)\rho (h_1(1)+1)+1)\\
=&\rho h_1(1)+\rho\mathfrak{a}(B,D_\flat,E_\flat)+\rho\mathfrak{e}(\rho;B,D_\flat,E_\flat)\rho (h_1(1)+1)+\rho\\
=&\rho\mathfrak{a}(B,D_\flat,E_\flat)+\rho^2\mathfrak{e}(\rho;B,D_\flat,E_\flat)+\rho+(\rho +\rho^2\mathfrak{e}(\rho;B,D_\flat,E_\flat))h_1(1).
\end{align*}

For any $n_0, \fn{k}_0$, the function $k\mapsto\fn{L}_0(n_0,\fn{k}_0,\mathcal{B}_0)(k)$ is $\fn{u}$-bounded by $1$ above $\max\{n_-,n_0,p_\dagger\}$.  Therefore by the previous lemma applied with $z_0=z^{\mathfrak{b}(B,D_\flat,E_\flat)+1}$, we obtain $\mathcal{B}_0, n_0, \fn{k}_0$ such that:
\begin{itemize}
\item $n_0\leq\fn{u}^{\mathfrak{e}(\rho\mathfrak{a}(B,D_\flat,E_\flat)+\rho^2\mathfrak{e}(\rho;B,D_\flat,E_\flat)+\rho;B,D_\flat,E_\flat)(\rho +\rho^2\mathfrak{e}(\rho;B,D_\flat,E_\flat))}(n_-)$,
\item $\fn{k}_\sharp$ is $\fn{u}$-bounded by $\mathfrak{a}(B,D_\flat,E_\flat)x+\mathfrak{e}(\rho\mathfrak{a}(B,D_\flat,E_\flat)+\rho^2\mathfrak{e}(\rho;B,D_\flat,E_\flat)+\rho;B,D_\flat,E_\flat)(\rho +\rho^2\mathfrak{e}(\rho;B,D_\flat,E_\flat))$ above $\max\{n_-,n_0\}$.
\end{itemize}

We obtain final bounds with:
\begin{itemize}
\item $|\mathcal{B}_\sharp|\leq z^{\mathfrak{b}(B,D_\flat,E_\flat)(\mathfrak{b}(B,D_\flat,E_\flat)+2)}$,
\item $n_\sharp=n_0\leq \fn{u}^{\mathfrak{f}(\rho;B,D_\flat;E_\flat)}(n_-)$,
\item $p_\sharp=p_\dagger\leq \fn{u}^{\mathfrak{g}(\rho;B,D_\flat;E_\flat)}(n_-)$,
\item $\fn{k}_\sharp$ is $\fn{u}$-bounded by $\mathfrak{a}(B,D_\flat,E_\flat)x+\mathfrak{f}(\rho;D_\flat;E_\flat)$ above $\max\{n_-,n_\sharp,p_\sharp\}$,
\item $\fn{r}_\sharp$ is $\fn{u}$-bounded by $\mathfrak{a}(B,D_\flat,E_\flat)y+\mathfrak{g}(\rho;D_\flat;E_\flat)$ above $\max\{n_-,n_\sharp,p_\sharp\}$.
\end{itemize}

\end{proof}

\begin{lemma}[Bounds on \ref{thm:control_interval}]
Suppose $(*)^Q$ holds.  There is a double exponential function $\mathfrak{i}(B,D^0_\flat,D^1_\flat,E_\flat)$ so that for any $n_\flat,p_\flat, D^0_\flat, D^1_\flat, E_\flat, \fn{u}$ such that:
\begin{itemize}
\item $\fn{u}(m)>m$ for all $m$,
\item For any $m,q,\fn{k},\fn{r}$, the function $k\mapsto\fn{L}_\flat(m,q,\fn{k},\fn{r})(k,r)$ does not depend on $r$ and is $\fn{u}$-bounded by $1$ above $\max\{n_\flat,p_\flat,m,q+1\}$,
\item For any $m,q,\fn{k},\fn{r}$, the function $r\mapsto \fn{S}_\flat(m,q,\fn{k},\fn{r})(k,r)$ does not depend on $k$ and is $\fn{u}$-bounded by $1$ above $\max\{n_\flat,p_\flat,m+1,q\}$,
\end{itemize}
 there are $m_\sharp\geq n_\flat$, $q_\sharp\geq p_\flat$, $\fn{k}_\sharp$, and $\fn{r}_\sharp$ such that:
\begin{itemize}
\item $m_\sharp\leq \fn{u}^{\mathfrak{i}(B,D^0_\flat,D^1_\flat,E_\flat)}(\max\{n_\flat,p_\flat\})$,
\item $q_\sharp\leq \fn{u}^{\mathfrak{i}(B,D^0_\flat,D^1_\flat,E_\flat)}(\max\{n_\flat,p_\flat\})$,
\item $\fn{k}_\sharp$ is $\fn{u}$-bounded by $\mathfrak{a}(B,D^1_\flat,BE_\flat)x+\mathfrak{i}(B,D^0_\flat,D^1_\flat,E_\flat)$ above $\max\{m_\sharp, q_\sharp\}$,
\item $\fn{r}_\sharp$ is $\fn{u}$-bounded by $\mathfrak{a}(B,D^1_\flat,BE_\flat)x+\mathfrak{i}(B,D^0_\flat,D^1_\flat,E_\flat)$ above $\max\{m_\sharp, q_\sharp\}$,
\item Setting $\fn{l}_\flat=\fn{L}_\flat(m_\sharp,q_\sharp,\fn{k}_\sharp,\fn{r}_\sharp)$, $\fn{s}_\flat=\fn{S}_\flat(m_\sharp,q_\sharp,\fn{k}_\sharp,\fn{r}_\sharp)$, $k_\sharp=\fn{k}_\sharp(\fn{l}_\flat,\fn{s}_\flat)$ and $r_\sharp=\fn{r}_\sharp(\fn{l}_\flat,\fn{s}_\flat)$, we have $k_\sharp\geq m_\sharp$, $r_\sharp\geq q_\sharp$, and if $l_\flat\geq k_\sharp$ and $s_\flat\geq r_\sharp$ then for any $s\in[r_\sharp,\fn{s}_\flat(k_\sharp,r_\sharp)]$ and $l\in[k_\sharp,\fn{l}_\flat(k_\sharp,r_\sharp)]$, there are sets $\mathcal{B}^-$, $\mathcal{B}^{0,-}$, and $\mathcal{B}^{1,-}$ with $\mu(\mathcal{B}^-)<4/D^0_\flat$, $\mu(\mathcal{B}^{0,-})<2/D^0_\flat$, and $\mu(\mathcal{B}^{1,-})<2/D^1_\flat$ so that
\begin{align*}
\left|(\rho_{m_\sharp}\lambda_s)(\Omega)-(\rho_l\lambda_{q_\sharp})(\Omega)\right|
&\leq|\rho_{m_\sharp}\lambda_s|(\mathcal{B}^-)+|\rho_l\lambda_{q_\sharp}|(\mathcal{B}^-)\\
&\ \ \ \ \ \ \ \ +BD_\flat^0|\rho_{m_\sharp}|(\mathcal{B}^{0,-})+|\rho_{m_\sharp}\lambda_s|(\mathcal{B}^{0,-})\\
&\ \ \ \ \ \ \ \ +BD_\flat^0|\lambda_{q_\sharp}|(\mathcal{B}^{1,-})+|\rho_l\lambda_{q_\sharp}|(\mathcal{B}^{1,-})+\frac{6}{E_\flat}.\\
\end{align*}
\end{itemize}
\end{lemma}
\begin{proof}
We examine the proof of Lemma \ref{thm:control_interval}.

The functions $\fn{B}^i_\star$ are particularly simple: $|\fn{B}^i_\star(\mathcal{B},\cdot)|\leq 2D^1_\flat B^2E_\flat|\mathcal{B}|$ no matter what the second input is.  This means that when we apply the previous lemma, we will do so with $z=2D^1_\flat B^2E_\flat$, which means that the set $\mathcal{B}_0$ we ultimately consider will have $|\mathcal{B}_0|\leq (2D^1_\flat B^2E_\flat)^{\mathfrak{b}(B,D^1_\flat,BE_\flat)(\mathfrak{b}(B,D^1_\flat,BE_\flat)+2)}$.  We write $\mathfrak{h}_0(B,D^1_\flat, E_\flat)=(2D^1_\flat B^2E_\flat)^{\mathfrak{b}(B,D^1_\flat,BE_\flat)(\mathfrak{b}(B,D^1_\flat,BE_\flat)+2)}$.

Suppose that, as in the proof, we are given $n_0, p_0, \fn{k}_0, \fn{r}_0, \mathcal{B}_0$ with $|\mathcal{B}_0|\leq \mathfrak{h}_0(B,D^1_\flat, E_\flat)$, $\fn{k}_0$ is $\fn{u}$-bounded by $h_1(x,y)$ above $\max\{n_\flat, p_\flat, n_0,p_0\}$, and $\fn{r}_0$ is $\fn{u}$-bounded by $h_2(x,y)$ above $\max\{n_\flat, p_\flat, n_0,p_0\}$.

Suppose we are further given the values $m_\dagger$ and $q_\ddagger$.  Then
\begin{itemize}
\item $\fn{l}_{\ddagger,q_\ddagger}=\fn{L}_\flat(m_\dagger,q_\ddagger,\fn{k}_{\ddagger,q_\ddagger},\fn{r}_{\ddagger,q_\ddagger})$ is $\fn{u}$-bounded by $1$ above $\max\{n_\flat, p_\flat, n_0,p_0,m_\dagger,q_\ddagger+1\}$,
\item $\fn{s}_{\ddagger,q_\ddagger}=\fn{S}_\flat(m_\dagger,q_\ddagger,\fn{k}_{\ddagger,q_\ddagger},\fn{r}_{\ddagger,q_\ddagger})$ is $\fn{u}$-bounded by $1$ above $\max\{n_\flat, p_\flat, n_0,p_0,m_\dagger+1,q_\ddagger\}$,
\item $\fn{k}_{\ddagger,q_{\ddagger}}(\fn{l}_{\ddagger,q_{\ddagger}},\fn{s}_{\ddagger,q_{\ddagger}})=\fn{k}_0(m_\dagger,q_\ddagger,\fn{l}_{\ddagger,q_{\ddagger}},\fn{s}_{\ddagger,q_{\ddagger}},\ldots)$ is bounded by $\fn{u}^{h_1(1,1)}(\max\{n_\flat, p_\flat, n_0,p_0,m_\dagger+1,q_\ddagger+1\})$,
\item $\fn{r}_{\ddagger,q_{\ddagger}}(\fn{l}_{\ddagger,q_{\ddagger}},\fn{s}_{\ddagger,q_{\ddagger}})=\fn{r}_0(m_\dagger,q_{\ddagger},\fn{l}_{\ddagger,q_\ddagger},\fn{s}_{\ddagger,q_\ddagger},\cdots)$ is bounded by $\fn{u}^{h_2(1,1)}(\max\{n_\flat, p_\flat, n_0,p_0,m_\dagger+1,q_\ddagger+1\})$,
\item $\fn{q}_\ddagger$ is $\fn{u}$-bounded by $1+h_2(1,1)$ above $\max\{n_\flat, p_\flat, n_0, p_0, m_\dagger+1\}$.
\end{itemize}

Since the value $q_\ddagger$ is obtained by an application of Lemma \ref{thm:quant_convergence_partition} to $\fn{q}_\ddagger$, it follows that $q_\ddagger$ is bounded by $\fn{u}^{\mathfrak{a}_0(B,3D_\flat^0,|\mathcal{B}_0|BD^0_\flat E_\flat)(1+h_2(1,1))}(\max\{n_\flat, p_\flat, n_0, p_0, m_\dagger+1\})$.  We write $\theta$ for $\mathfrak{a}_0(B,3D_\flat^0,\mathfrak{h}_0(B,D^1_\flat, E_\flat)BD^0_\flat E_\flat)$; note that $\theta$ is exponential in $D_\flat^0$ and double exponential in $B$, $D^1_\flat$ and $E_\flat$, and we have $q_\ddagger\leq\fn{u}^{\theta(1+h_2(1,1))}(\max\{n_\flat, p_\flat, n_0, p_0, m_\dagger+1\})$. 

In particular, the function $m_\dagger\mapsto q_\ddagger$ is $\fn{u}$-bounded by $\theta(1+h_2(1,1))+1$ above $\max\{n_\flat, p_\flat, n_0,p_0\}$.

Given $m_\dagger$, the value $\fn{k}_{\ddagger,q_\ddagger}(\fn{l}_{\ddagger,q_\ddagger},\fn{r}_{\ddagger,q_\ddagger})$ is bounded by $\fn{u}^{h_1(1,1)+\theta(1+h_2(1,1))}(\max\{n_\flat, p_\flat, n_0,p_0,m_\dagger+1\})$, so $\fn{m}_\dagger$ is $\fn{u}$-bounded by $2+h_1(1,1)+\theta(1+h_2(1,1))$ above $\max\{n_\flat, p_\flat, n_0,p_0\}$.  Therefore $m_\dagger$ is bounded by $\fn{u}^{\theta(2+h_1(1,1)+\theta(1+h_2(1,1)))}(\max\{n_\flat, p_\flat, n_0,p_0\})$.

We now prepare to apply the previous lemma:
\begin{itemize}
\item Each $\fn{B}^i_0$ satisfies $|\fn{B}^i_0(n_0,p_0,\fn{k}_0,\fn{r}_0,\mathcal{B}_0)|\leq z|\mathcal{B}_0|$,
\item $\fn{m}_0(n_0,p_0,\fn{k}_0,\fn{r}_0,\mathcal{B}_0)$ is $\fn{u}$-bounded by $\theta(2+h_1(1,1)+\theta(1+h_2(1,1)))$ above $\max\{n_\flat, p_\flat\}$,
\item $\fn{q}_0(n_0,p_0,\fn{k}_0,\fn{r}_0,\mathcal{B}_0)$ is $\fn{u}$-bounded by $\theta(2+h_1(1,1)+\theta(1+h_2(1,1)))+\theta(1+h_2(1,1))$ above $\max\{n_\flat, p_\flat\}$,
\item For any $n_0,p_0,\fn{k}_0,\fn{r}_0,\mathcal{B}_0$, the function $m\mapsto\fn{L}_0(n_0,p_0,\fn{k}_0,\fn{r}_0,\mathcal{B}_0)(m,q)$ does not depend on $q$ and is $\fn{u}$-bounded by $1$ above $\max\{n_\flat, p_\flat, n_0,p_0,m_\dagger,q_\ddagger+1\}$,
\item For any $n_0,p_0,\fn{k}_0,\fn{r}_0,\mathcal{B}_0$, the function $q\mapsto\fn{S}_0(n_0,p_0,\fn{k}_0,\fn{r}_0,\mathcal{B}_0)(m,q)$ does not depend on $m$ and is $\fn{u}$-bounded by $1$ above $\max\{n_\flat, p_\flat,n_0,p_0,m_\dagger+1,q_\ddagger\}$,.
\end{itemize}

This puts us in the setting of the previous lemma with $\rho=\theta^2+2\theta+1$, so we obtain
\begin{itemize}
\item $n_0\leq\fn{u}^{f(\rho;B,D^1_\flat,BE_\flat)}(\max\{n_\flat,p_\flat\})$,
\item $p_0\leq\fn{u}^{g(\rho;B,D^1_\flat,BE_\flat)}(\max\{n_\flat,p_\flat\})$,
\item $\fn{k}_0$ is $\fn{u}$-bounded by $\mathfrak{a}(B,D^1_\flat,BE_\flat)x+\mathfrak{f}(\rho;B,D^1_\flat,BE_\flat)$ above $\max\{n_\flat, p_\flat, n_0,p_0\}$,
\item $\fn{r}_0$ is $\fn{u}$-bounded by $\mathfrak{a}(B,D^1_\flat,BE_\flat)y+\mathfrak{g}(\rho;B,D^1_\flat,BE_\flat)$ above $\max\{n_\flat, p_\flat, n_0,p_0\}$.
\end{itemize}

Therefore, plugging $h_1(x,y)=\mathfrak{a}(B,D^1_\flat,BE_\flat)x+\mathfrak{f}(\rho;B,D^1_\flat,BE_\flat)$ and $h_2(x,y)=\mathfrak{a}(B,D^1_\flat,BE_\flat)y+\mathfrak{g}(\rho;B,D^1_\flat,BE_\flat)$ in to the equations above, we obtain the desired bounds:
\begin{itemize}
\item $m_\sharp\leq \fn{u}^{\mathfrak{i}(B,D^0_\flat,D^1_\flat,E_\flat)}(\max\{n_\flat,p_\flat\})$,
\item $q_\sharp\leq \fn{u}^{\mathfrak{i}(B,D^0_\flat,D^1_\flat,E_\flat)}(\max\{n_\flat,p_\flat\})$,
\item $\fn{k}_\sharp$ is $\fn{u}$-bounded by $\mathfrak{a}(B,D^1_\flat,BE_\flat)x+\mathfrak{i}(B,D^0_\flat,D^1_\flat,E_\flat)$ above $\max\{m_\sharp, q_\sharp\}$,
\item $\fn{r}_\sharp$ is $\fn{u}$-bounded by $\mathfrak{a}(B,D^1_\flat,BE_\flat)x+\mathfrak{i}(B,D^0_\flat,D^1_\flat,E_\flat)$ above $\max\{m_\sharp, q_\sharp\}$.
\end{itemize}

\end{proof}

\subsection{Fast Growing Functions}
At this point our bounds start growing much more rapidly.  Suppose we have fixed a function $\fn{u}$.  We define:
\begin{itemize}
\item $C=2^{29}B^6E_\flat^6$,
\item $\fn{w}_{0,\fn{u},B,E}(m)=\fn{u}(m)$,
\item $\fn{w}_{j+1,\fn{u},B,E}(m)=\fn{w}_{j,\fn{u},E}^{\mathfrak{a}(B, E 2^{m+5},BE)C^2+\mathfrak{i}(B,E2^{m+5},E 2^{m+5},E)C}(m)$.
\end{itemize}

We will ultimately be interested in the case where $\fn{u}=\mathrm{suc}$ where $\mathrm{suc}(m)=m+1$.  Observe that in this case, $\fn{w}_{1,\mathrm{suc},E}$ is triply exponential, $\fn{w}_{2,\mathrm{suc},E}(m)$ is a tower of exponents of size roughly triply exponential in $m$.  To express the bounds more generally, recall the fast-growing hierarchy:
\begin{itemize}
\item $f_0(m)=m+1$,
\item $f_{j+1}(m)=f_j^m(m)$.
\end{itemize}

Then we have
\begin{lemma}
  There is a $c$ so that for all $j,m,E,B$, $\fn{w}_{j,\mathrm{suc},E}(m)\leq f_{2j+1}(m+E+B+c)$.
\end{lemma}
\begin{proof}
Since $\mathfrak{a}(B, E 2^{m+5},BE)C^2+\mathfrak{i}(B,E2^{m+5},E 2^{m+5},E)C$ is triply exponential in $m$ while $f_3$ is tower exponential, there is a $c$ so that for all $m,E,B$, $\mathfrak{a}(B, E 2^{m+5},BE)C^2+\mathfrak{i}(B,E2^{m+5},E 2^{m+5},E)C\leq f_3(m+E+B+c)$.  (Indeed, $c=5$ suffices.)

 When $j=0$ the statement is immediate, since $\fn{w}_{0,\mathrm{suc},E}(m)=m+1\leq f_3(m+c)$.

  Suppose the claim holds for $j$.  Then $\fn{w}_{j+1,\mathrm{suc},E}(m)=\fn{w}_{j,\fn{u},E}^{\mathfrak{a}(B, E 2^{m+5},BE)C^2+\mathfrak{i}(B,E2^{m+5},E 2^{m+5},E)C}(m)$.  Then
  \begin{align*}
\fn{w}_{j+1,\mathrm{suc},E}(m)
&=    \fn{w}_{j,\fn{u},E}^{\mathfrak{a}(B, E 2^{m+5},BE)C^2+\mathfrak{i}(B,E2^{m+5},E 2^{m+5},E)C}(m)\\
&\leq\fn{w}_{j,\fn{u},E}^{f_3(m+c)}(m)\\
&\leq f_{2j+1}^{f_3(m+c)}(m+E+B+c)\\
&\leq f_{2j+1}^{f_{2j+1}(m+c)}(m+E+B+c)\\
&\leq f_{2j+2}^{m+c}(m+E+B+c)\\
&= f_{2j+3}(m+E+B+c).
  \end{align*}
\end{proof}

\subsection{Bounds on Uniform Continuity}
Before continuing, we need bounds on Theorem \ref{thm:q_vhs} and Theorem \ref{thm:meta_bnd_3}.

\begin{lemma}[Bounds on Theorem \ref{thm:q_vhs}]\label{thm:q_vhs_bnd}
Suppose $(*)^Q$ holds.  Let $E_\flat, D_\flat, n_\flat$ be given.  Suppose $\fn{m}_\flat(E_\flat 2^{m_0+5},m_0)\leq \fn{v}(m_0)$ for all $m_0\geq\max\{n_\flat, \ln D_\flat\}$.  Then there is an $m_\sharp\leq\fn{v}^{2^7B^2E_\flat^2}(\max\{n_\flat,\ln D_\flat\})$ such that, setting $D_\sharp=2^{m_\sharp}\geq D_\flat$, for each $m\in[m_\sharp,\fn{m}(D_\sharp,m_\sharp)]$, whenever $\mu(\sigma)<1/D_\sharp$, $|\nu_m(\sigma)|<1/E_\flat$.
\end{lemma}
\begin{proof}
Examining the proof of Theorem \ref{thm:q_vhs}, the function $\fn{m}(m_0)$ is just $\fn{m}_\flat(E_\flat 2^{m_0+5},m_0)$, which is bounded by $\fn{v}(m_0)$.  The sequence of values $m_i$ are given by $m_0=\max\{n_\flat,\ln D_\flat\}$ (to ensure that $E_\flat 2^{m_0+5}\geq D_\flat$), $m_{i+1}\leq \fn{m}_\flat(E_\flat 2^{m_i+5},m_i)\leq\fn{v}(m_i)$.  Since we have bounded fluctuations, we need only consider $m_{2^7B^2E_\flat^2}$, so $m_\sharp\leq\fn{v}^{2^7B^2E_\flat^2}(\max\{n_\flat,\ln D_\flat\})$ as claimed.
\end{proof}

\begin{lemma}[Bounds on Lemma \ref{thm:meta_bnd_1}]\label{thm:meta_bnd_1_q}
    Suppose $(*)^Q$ holds.  For any $d, E_\flat, \fn{m}_\flat, \fn{q}_\flat, n_\flat, p_\flat$ such that:
  \begin{itemize}
  \item For any $D, m, q, \fn{r}$ such that $\fn{r}$ is $\fn{w}_{j,\fn{u},E_\flat}$-bounded by $C$ above $\max\{m,q,\ln D,n_\flat,p_\flat\}$, $\fn{m}_\flat(D,m,q,\fn{r})\leq\fn{w}^{2^{12}B^4E^4_\flat}_{j+d,\fn{u},E_\flat}(\max\{m,q,\ln D,n_\flat,p_\flat\})$,
  \item For any $D, m, q, \fn{r}$ such that $\fn{r}$ is $\fn{w}_{j,\fn{u},E_\flat}$-bounded by $C$ above $\max\{m,q,\ln D,n_\flat,p_\flat\}$, $\fn{q}_\flat(D,m,q,\fn{r})\leq\fn{w}^{2^{12}B^4E^4_\flat}_{j+d,\fn{u},E_\flat}(\max\{m,q,\ln D,n_\flat,p_\flat\})$, 
 \end{itemize}
there exist $D_\sharp, m_\sharp\geq n_\flat, q_\sharp\geq p_\flat, \fn{r}_\sharp$ such that:
\begin{itemize}
  \item $m_\sharp\leq \fn{w}_{d(2^9B^2E^2_\flat),\fn{u},E_\flat}^{C}(\max\{n_\flat,p_\flat\})$
  \item $q_\sharp\leq  \fn{w}_{d(2^9B^2E^2_\flat),\fn{u},E_\flat}^{C}(\max\{n_\flat,p_\flat\})$,
  \item $D_\sharp=E_\flat 2^{q_\sharp+5}$,
  \item $\fn{r}_\sharp$ is $\fn{w}^C_{d(2^9B^2E^2_\flat),\fn{u},E_\flat}$-bounded by $1$ above $\max\{m_\sharp,q_\sharp,n_\flat,p_\flat\}$,
\end{itemize}
and, setting $m_\flat=\fn{m}_\flat(D_\sharp,m_\sharp,q_\sharp,\fn{r}_\sharp)$, if $q_\flat=\fn{q}_\flat(D_\sharp,m_\sharp,q_\sharp,\fn{r}_\sharp)\geq q_\sharp$ then $r_\sharp=\fn{r}_\sharp(m_\flat,q_\flat)\geq q_\flat$, there is a $\sigma_0$ such that whenever $\mu(\sigma_0\bigtriangleup\sigma)<1/D_\sharp$, $|(\rho_{m_\sharp}\lambda_{r_\sharp})(\sigma)-(\rho_{m_\flat}\lambda_{r_\sharp})(\sigma)|<4/E_\flat$, and whenever $\mu(\sigma)<2/D_\sharp$, $|(\rho_{m_\sharp}\lambda_{r_\sharp})(\sigma)|<1/16E_\flat$.
\end{lemma}
\begin{proof}
We need the following inductive hypothesis: for each $i, D, n, p$ there is a function $\fn{r}_{i,D,n,p}$ which is $\fn{w}_{di,\fn{u},E_\flat}$-bounded by $C$ above $\max\{n,p,\ln D,n_\flat,p_\flat\}$ so that for each $m,q$, either:
\begin{itemize}
\item There exist $D_\sharp, m_\sharp, q_\sharp, \fn{r}_\sharp$ satisfying the conclusion of Lemma 4.10 with:
  \begin{itemize}
  \item $m_\sharp\leq \fn{w}_{di,\fn{u},E_\flat}^{C+1}(\max\{m,q,n,p,n_\flat,p_\flat\})$
  \item $q_\sharp\leq  \fn{w}_{di,\fn{u},E_\flat}^{C+1}(\max\{m,q,n,p,n_\flat,p_\flat\})$,
  \item $D_\sharp=E_\flat 2^{q_\sharp+5}$,
  \item $\fn{r}_\sharp$ is $\fn{w}^C_{di,\fn{u},E_\flat}$-bounded by $1$ above $\max\{m_\sharp,q_\sharp,m,q,n,p, n_\flat,p_\flat\}$,
  \end{itemize}
or,
\item one of the other cases in the proof Lemma 4.10 holds.
\end{itemize}

For $\fn{r}_{0,D,n,p}(m,q)=q$ this is immediate since $q\leq \fn{u}(q)$ for all $q$.

Suppose we have shown that for every $D,n,p$, $\fn{r}_{i,D,n,p}$ is $\fn{w}_{di,\fn{u},E_\flat}$-bounded by $C$ above $\max\{n,p,\ln D,n_\flat,p_\flat\}$.  We now attempt to bound $\fn{r}_{i+1,D,n,p}$.  Let $m,q$ be given.  Observe that $\fn{m}_\flat(D^*,m,q^*,\fn{r}_{i,D^*,m,q^*})\leq\fn{w}^{2^{12}B^4E^4_\flat}_{d(i+1),\fn{u},E_\flat}(\max\{q^*,\ln D^*,m,q,n,p,n_\flat,p_\flat\})$ and similarly for $\fn{q}_\flat$, so 
\[\fn{r}^*(D^*,q^*)\leq\fn{w}_{d(i+1),\fn{u},E_\flat}^{{2^{12}B^4E^4_\flat}+1}(\max\{q^*,\ln D^*,m,q,n,p,\ln D+\ln 2,n_\flat,p_\flat\}).\]
 In particular,
\[\fn{r}^*(E_\flat 2^{q^*+5},q^*)\leq\fn{w}_{d(i+1),\fn{u},E_\flat}^{(2^{12}B^4E^4_\flat)+1}(\max\{q^*+\ln E_\flat+1,m,q,n,p,\ln D+\ln 2,n_\flat,p_\flat\}),\]
so by Lemma \ref{thm:q_vhs_bnd} applied to $16E_\flat$ we obtain $q^*\leq \fn{w}_{d(i+1),\fn{u},E_\flat}^{((2^{12}B^4E^4_\flat)+2)2^{15}B^2E^2_\flat}(\max\{m,q,n,p,\ln D,n_\flat,p_\flat\})$ and $D^*=E_\flat 2^{q^*+5}$.  Then
\[\fn{r}_{i+1,D,n,p}(m,q)\leq \fn{w}_{d(i+1),\fn{u},E_\flat}^{C}(\max\{m,q,n,p,\ln D,n_\flat,p_\flat\}).\]
as required.

We also need to bound the possible witnesses $D_\sharp, m_\sharp, q_\sharp, \fn{r}_\sharp$.  We take $m'=\fn{m}_\flat(D^*,m,q^*,\fn{r}_{i,D^*,m,q^*})$ and $q'=\fn{q}_\flat(D^*,m,q^*,\fn{r}_{i,D^*,m,q^*})$, and have
\[m',q'\leq \fn{w}_{d(i+1),\fn{u},E_\flat}^{C}(\max\{m,q,n,p,\ln D,n_\flat,p_\flat\}).\]

In particular, when we apply the inductive hypothesis to $\fn{r}_{i,D^*,m,q^*}(m',q')$, we potentially obtain $D_\sharp, m_\sharp, q_\sharp, \fn{r}_\sharp$ with
\[m_\sharp,q_\sharp\leq\fn{w}_{di,\fn{u},E_\flat}^{C+1}(\max\{m',q',m,q^*,\ln D^*,n_\flat,p_\flat\})\leq
\fn{w}_{d(i+1),\fn{u},E_\flat}^{C+1}(\max\{n,p,m,q,\ln D,n_\flat,p_\flat\})\]
which satisfies the promised bounds.

If we choose $m_\sharp=m, q_\sharp=q^*, D_\sharp=D^*, \fn{r}_\sharp=\fn{r}_{i,D^*,m,q^*}$ then again the promised bounds hold.

In the proof, we work with 
\[\fn{r}^*(D^*,q^*)=\fn{r}_{2^9B^2E^2_\flat,D^*,1,q^*}(\fn{m}_\flat(D^*,1,q^*,\fn{r}_{2^9B^2E^2_\flat,D^*,1,q^*}),\fn{q}_\flat(D^*,1,q^*,\fn{r}_{2^9B^2E^2_\flat,D^*,1,q^*})).\]
In particular, as above
\[\fn{r}^*(E_\flat 2^{q^*+5},q^*)\leq\fn{w}_{d(2^9B^2E^2_\flat+1),\fn{u},E_\flat}^{{2^{12}B^4E^4_\flat}+1}(\max\{q^*+\ln E_\flat+1, n_\flat, p_\flat\})\]
and therefore when we apply Lemma \ref{thm:q_vhs_bnd} to $16E_\flat$ we obtain $q^*\leq\fn{w}_{d(2^9B^2E^2_\flat+1),\fn{u},E_\flat}^{2^{28}B^6E^6_\flat}(\max\{n_\flat,p_\flat\})$.  In particular, whatever case we are in, we have the specified bounds.
\end{proof}

\begin{lemma}[Bounds on Lemma \ref{thm:meta_bnd_3}]\label{thm:meta_bnd_3_q}
  Suppose $(*)^Q$ holds.  Let $d, E_\flat, n_\flat, p_\flat, \fn{m}_\flat, \fn{q}_\flat, d$ be given so that:
  \begin{itemize}
  \item For any $D, m, q, \fn{r}$ such that $\fn{r}$ is $\fn{w}_{j,\fn{u},E_\flat}$-bounded by $C$ above $\max\{m,q,\ln D,n_\flat,p_\flat\}$, $\fn{m}_\flat(D,m,q,\fn{r})\leq\fn{w}^{2^{12}B^4E^4_\flat}_{j+d,\fn{u},E_\flat}(\max\{m,q,\ln D+,n_\flat,p_\flat\})$,
  \item For any $D, m, q, \fn{r}$ such that $\fn{r}$ is $\fn{w}_{j,\fn{u},E_\flat}$-bounded by $C$ above $\max\{m,q,\ln D,n_\flat,p_\flat\}$, $\fn{q}_\flat(D,m,q,\fn{r})\leq\fn{w}^{2^{12}B^4E^4_\flat}_{j+d,\fn{u},E_\flat}(\max\{m,q,\ln D,n_\flat,p_\flat\})$, 
  \end{itemize}

Then there are $D_\sharp, m_\sharp, p_\sharp,\fn{r}_\sharp$ so that:
\begin{itemize}
  \item $m_\sharp\leq \fn{w}_{d(2^9B^2E^2_\flat),\fn{u},E_\flat}^{C}(\max\{n_\flat,p_\flat\})$
  \item $q_\sharp\leq  \fn{w}_{d(2^9B^2E^2_\flat),\fn{u},E_\flat}^{C}(\max\{n_\flat,p_\flat\})$,
  \item $D_\sharp=E_\flat 2^{q_\sharp+5}$,
  \item $\fn{r}_\sharp$ is $\fn{w}^C_{d(2^9B^2E^2_\flat),\fn{u},E_\flat}$-bounded by $1$ above $\max\{m_\sharp,q_\sharp,n_\flat,p_\flat\}$,
\item If $\fn{m}_\flat(D_\sharp,m_\sharp,q_\sharp,\fn{r}_\sharp)\geq m_\sharp$ and $\fn{q}_\flat(D_\sharp,m_\sharp,p_\sharp,\fn{r}_\sharp,m)\geq p_\sharp$ then $\fn{r}_\sharp(\fn{m}_\flat(D_\sharp,m_\sharp,q_\sharp,\fn{r}_\sharp),\fn{q}_\flat(D_\sharp,m_\sharp,p_\sharp,\fn{r}_\sharp))\geq\fn{q}_\flat(D_\sharp,m_\sharp,p_\sharp,\fn{r}_\sharp)$ and for any $\sigma$ with $\mu(\sigma)<1/D_\sharp$, $|(\rho_{\fn{m}_\flat(D_\sharp,m_\sharp,q_\sharp,\fn{r}_\sharp)}\lambda_{\fn{r}_\sharp(\fn{m}_\flat(D_\sharp,m_\sharp,q_\sharp,\fn{r}_\sharp),\fn{q}_\flat(D_\sharp,m_\sharp,p_\sharp,\fn{r}_\sharp))})(\sigma)|<1/E_\flat$.
\end{itemize}
\end{lemma}
\begin{proof}
Theorem \ref{thm:meta_bnd_3} follows by applying Lemma \ref{thm:meta_bnd_1} to $4E_\flat, \fn{m}_\flat, \fn{p}_\flat,0,0$.  Lemma \ref{thm:meta_bnd_1_q} is already the $4E_\flat$ case, so we simply apply the previous lemma to obtain the desired bounds.
\end{proof}

\subsection{Refining the Bounds}

\begin{lemma}[Bounds on \ref{thm:control_interval_D1}]
Suppose $(*)^Q$ holds.  Let $j,E_\flat, D^0_\flat\leq D^1_\flat, n_\flat, p_\flat, \fn{L}_\flat, \fn{S}_\flat$ be given so that:
\begin{itemize}
\item For any $m,q,\fn{k},\fn{r}$, the function $k\mapsto\fn{L}_\flat(m,q,\fn{k},\fn{r})(k,r)$ does not depend on $r$ and is $\fn{w}_{j,\fn{u},E_\flat}$-bounded by $C$ above $\max\{n_\flat,p_\flat,m,q+1\}$,
\item For any $m,q,\fn{k},\fn{r}$, the function $r\mapsto \fn{S}_\flat(m,q,\fn{k},\fn{r})(k,r)$ does not depend on $k$ and is $\fn{w}_{j,\fn{u},E_\flat}$-bounded by $C$ above $\max\{n_\flat,p_\flat,m+1,q\}$.
\end{itemize}

Then there are $m_\sharp\geq n_\flat$, $q_\sharp\geq p_\flat$, $\fn{k}_\sharp$, and $\fn{r}_\sharp$ such that, setting $c_\sharp=\fn{w}^{2^7B^2E_\flat^2}_{j+1,\fn{u},E_\flat}(\max\{n_\flat,p_\flat,\ln D^1_\flat\})$:
\begin{itemize}
\item $m_\sharp\leq \fn{w}_{j+1,\fn{u},E_\flat}(c_\sharp)$,
\item $q_\sharp\leq  \fn{w}_{j+1,\fn{u},E_\flat}(c_\sharp)$,
\item $\fn{k}_\sharp$ is $\fn{w}_{j,\fn{u},E_\flat}$-bounded by $\mathfrak{a}(B,E_\flat 2^{c_\sharp+5},BE_\flat)Cx+\mathfrak{i}(B,D^0_\flat,E_\flat 2^{c_\sharp+5},E_\flat)C$ above $\max\{m_\sharp, q_\sharp\}$,
\item $\fn{r}_\sharp$ is $\fn{w}_{j,\fn{u},E_\flat}$-bounded by $\mathfrak{a}(B,E_\flat 2^{c_\sharp+5},BE_\flat)Cx+\mathfrak{i}(B,D^0_\flat,E_\flat 2^{c_\sharp+5},E_\flat)C$ above $\max\{m_\sharp, q_\sharp\}$,
\item Setting $\fn{l}_\flat=\fn{L}_\flat(m_\sharp,q_\sharp,\fn{k}_\sharp,\fn{r}_\sharp)$, $\fn{s}_\flat=\fn{S}_\flat(m_\sharp,q_\sharp,\fn{k}_\sharp,\fn{r}_\sharp)$, $k_\sharp=\fn{k}_\sharp(\fn{l}_\flat,\fn{s}_\flat)$ and $r_\sharp=\fn{r}_\sharp(\fn{l}_\flat,\fn{s}_\flat)$, we have $k_\sharp\geq m_\sharp$, $r_\sharp\geq q_\sharp$, and for any $s\in[r_\sharp,\fn{s}_\flat(k_\sharp,r_\sharp)]$ and $l\in[k_\sharp,\fn{l}_\flat(k_\sharp,r_\sharp)]$, there are sets $\mathcal{B}^-$, $\mathcal{B}^{0,-}$, and $\mathcal{B}^{1,-}$ with $\mu(\mathcal{B}^-)<4/D^0_\flat$, $\mu(\mathcal{B}^{0,-})<2/D^0_\flat$, and $\mu(\mathcal{B}^{1,-})<2/D^1_\flat$ so that
\begin{align*}
\left|(\rho_{m_\sharp}\lambda_s)(\Omega)-(\rho_l\lambda_{q_\sharp})(\Omega)\right|
&\leq|\rho_{m_\sharp}\lambda_s|(\mathcal{B}^-)+|\rho_l\lambda_{q_\sharp}|(\mathcal{B}^-)+|\rho_{m_\sharp}\lambda_s|(\mathcal{B}^{0,-})\\
&\ \ \ \ \ \ \ \ +BD_\flat^0|\lambda_{q_\sharp}|(\mathcal{B}^{1,-})+|\rho_l\lambda_{q_\sharp}|(\mathcal{B}^{1,-})+\frac{7}{E_\flat}.\\
\end{align*}
\end{itemize}
\end{lemma}
\begin{proof}
  This lemma and the next amount to combining Lemma \ref{thm:control_interval} with Theorem \ref{thm:q_vhs}.  

Note that being $\fn{w}_{j,\fn{u},E_\flat}$-bounded by $C$ is the same as being $\fn{w}_{j,\fn{u},E_\flat}^C$-bounded by $1$, so the lemmas in the previous subsections apply.  Suppose we have fixed $D_0\geq D^1_\flat$ and $m_0\geq \max\{n_\flat,p_\flat\}$.  Let $c_\sharp=\fn{w}_{j,\fn{u},E_\flat}^{\mathfrak{i}(B,D^0_\flat,D_0,E_\flat)C}(m_0)$.  In particular, $\fn{m}_0(E_\flat 2^{m_0+5},m_0)\leq\fn{w}_{j,\fn{u},E_\flat}^{\mathfrak{i}(B,E_\flat 2^{m_0+5},E_\flat 2^{m_0+2},E_\flat)C}(m_0)\leq \fn{w}_{j+1,\fn{u},E_\flat}(m_0)$.  Applying the Lemma \ref{thm:q_vhs_bnd}, $m_0\leq \fn{w}^{32B^2E_\flat^2}_{j+1,\fn{u},E_\flat}(\max\{n_\flat,p_\flat,\ln D^1_\flat\})$.  Setting $c_\sharp=\fn{w}^{32B^2E_\flat^2}_{j+1,\fn{u},E_\flat}(\max\{n_\flat,p_\flat,\ln D^1_\flat\})$, the remaining bounds follow from the previous subsections.
\end{proof}

\begin{lemma}[Bounds on \ref{thm:control_interval_D2}]\label{thm:control_interval_D2_quant}
Suppose $(*)^Q$ holds.  Let $E_\flat, D_\flat,  n_\flat, p_\flat, \fn{L}_\flat, \fn{S}_\flat$, $\fn{u}$, be given so that:
\begin{itemize}
\item For any $m,q,\fn{k},\fn{r}$, the function $k\mapsto\fn{L}_\flat(m,q,\fn{k},\fn{r})(k,r)$ does not depend on $r$ and is $\fn{w}_{j,\fn{u},E_\flat}$-bounded by $C$ above $\max\{n_\flat,p_\flat,m,q+1\}$,
\item For any $m,q,\fn{k},\fn{r}$, the function $r\mapsto \fn{S}_\flat(m,q,\fn{k},\fn{r})(k,r)$ does not depend on $k$ and is $\fn{w}_{j,\fn{u},E_\flat}$-bounded by $C$ above $\max\{n_\flat,p_\flat,m+1,q\}$.
\end{itemize}

Then there are $m_\sharp\geq n_\flat$, $q_\sharp\geq p_\flat$, $\fn{k}_\sharp$, and $\fn{r}_\sharp$ such that, setting $c_\sharp=\fn{w}^{2^{11}B^4E_\flat^4}_{j+1,\fn{u},E_\flat}(\max\{n_\flat,p_\flat,\ln D_\flat\})$:
\begin{itemize}
\item $m_\sharp\leq \fn{w}^{32B^2E_\flat^2+1}_{j+1,\fn{u},E_\flat}(c_\sharp)$,
\item $q_\sharp\leq  \fn{w}^{32B^2E_\flat^2+1}_{j+1,\fn{u},E_\flat}(c_\sharp)$,
\item $\fn{k}_\sharp$ is $\fn{w}_{j,\fn{u},E_\flat}$-bounded by $\mathfrak{a}(B, E_\flat 2^{c_\sharp+5},BE_\flat)Cx+\mathfrak{i}(B,D_\flat, E_\flat 2^{c_\sharp+5},E_\flat)C$ above $\max\{m_\sharp, q_\sharp\}$,
\item $\fn{r}_\sharp$ is $\fn{w}_{j,\fn{u},E_\flat}$-bounded by $\mathfrak{a}(B, E_\flat 2^{c_\sharp+5},BE_\flat)Cy+\mathfrak{i}(B,D_\flat, E_\flat 2^{c_\sharp+5},E_\flat)C$ above $\max\{m_\sharp, q_\sharp\}$,
\item Setting $\fn{l}_\flat=\fn{L}_\flat(m_\sharp,q_\sharp,\fn{k}_\sharp,\fn{r}_\sharp)$, $\fn{s}_\flat=\fn{S}_\flat(m_\sharp,q_\sharp,\fn{k}_\sharp,\fn{r}_\sharp)$, $k_\sharp=\fn{k}_\sharp(\fn{l}_\flat,\fn{s}_\flat)$ and $r_\sharp=\fn{r}_\sharp(\fn{l}_\flat,\fn{s}_\flat)$, so that $k_\sharp\geq m_\sharp$, $r_\sharp\geq q_\sharp$, and for any $s\in[r_\sharp,\fn{s}_\flat(k_\sharp,r_\sharp)]$ and $l\in[k_\sharp,\fn{l}_\flat(k_\sharp,r_\sharp)]$, there are sets $\mathcal{B}^-$, $\mathcal{B}^{0,-}$, and $\mathcal{B}^{1,-}$ with $\mu(\mathcal{B}^-)<4/D^0_\flat$, $\mu(\mathcal{B}^{0,-})<2/D^0_\flat$, and $\mu(\mathcal{B}^{1,-})<2/D^1_\flat$ so that
\begin{align*}
\left|(\rho_{m_\sharp}\lambda_s)(\Omega)-(\rho_l\lambda_{q_\sharp})(\Omega)\right|
&\leq|\rho_{m_\sharp}\lambda_s|(\mathcal{B}^-)+|\rho_l\lambda_{q_\sharp}|(\mathcal{B}^-)+|\rho_{m_\sharp}\lambda_s|(\mathcal{B}^{0,-})+|\rho_l\lambda_{q_\sharp}|(\mathcal{B}^{1,-})+\frac{8}{E_\flat}.\\
\end{align*}
\end{itemize}
\end{lemma}
\begin{proof}
  We again combine the previous lemma with Theorem \ref{thm:q_vhs}.  

The function $q_0\mapsto \fn{q}_0(E_\flat 2^{q_0+5},q_0)$ is bounded by $\fn{w}^{32B^2E_\flat^2+1}_{j+1,\fn{u},E_\flat}(\max\{q_0,\ln (E_\flat 2^{q_0+5}),n_\flat,p_\flat,\ln D_\flat\})\leq\fn{w}^{32B^2E_\flat^2+1}_{j+1,\fn{u},E_\flat}(\max\{n_\flat,q_0\})$, so by Lemma \ref{thm:q_vhs_bnd} we have $q_0\leq\fn{w}^{2^{11}B^4E_\flat^4}_{j+1,\fn{u},E_\flat}(\max\{n_\flat,p_\flat,\ln D_\flat\})$.  Again, the remaining bounds follow from the previous subsection.
\end{proof}

\begin{lemma}[Bounds on \ref{thm:control_interval_E1}]
Suppose $(*)^Q$ holds.  Let $E_\flat, D_\flat, p_\flat, n_\flat, \fn{L}_\flat, \fn{r}_\flat$ be given so that:
\begin{itemize}
\item For any $m,q,\fn{k},\fn{s}$, the function $k\mapsto\fn{L}_\flat(m,q,\fn{k},\fn{s})(k,r)$ does not depend on $r$ and is $\fn{u}$-bounded by $1$ above $\max\{n_\flat,p_\flat,m,q+1\}$,
\item For any $m,q,\fn{k},\fn{s}$, $\fn{r}_\flat(m,q,\fn{k},\fn{s})=m+1$.
\end{itemize}

Then there are $m_\sharp\geq n_\flat, q_\sharp\geq p_\flat, \fn{k}_\sharp, \fn{s}_\sharp$ such that:
\begin{itemize}
\item $m_\sharp\leq\fn{w}_{2^9B^2E^2_\flat+1,\fn{u},E_\flat}^{2^{12}B^4E^4_\flat+1}(\max\{\ln D_\flat,n_\flat,p_\flat\})$,
\item $q_\sharp\leq\fn{w}_{2^9B^2E^2_\flat+1,\fn{u},E_\flat}^{2^{12}B^4E^4_\flat+1}(\max\{\ln D_\flat, n_\flat,p_\flat\})$,
\item $\fn{k}_\sharp$ is $\fn{w}_{2^9B^2E^2_\flat+1,\fn{u},E_\flat}$-bounded by $\mathfrak{a}(B, E_\flat 2^{c_\dagger+5},BE_\flat)Cx+\mathfrak{i}(B,D_\flat, E_\flat 2^{c_\sharp+5},E_\flat)C$ above $\max\{m_\sharp, q_\sharp,\ln D_\flat, n_\flat,p_\flat\}$,
\item $\fn{s}_\sharp$ is $\fn{w}_{2^9B^2E^2_\flat+1,\fn{u},E_\flat}$-bounded by $\mathfrak{a}(B, E_\flat 2^{c_\dagger+5},BE_\flat)Cy+\mathfrak{i}(B,D_\flat, E_\flat 2^{c_\sharp+5},E_\flat)C$ above $\max\{m_\sharp, q_\sharp,\ln D_\flat, n_\flat,p_\flat\}$, and
\item setting $\fn{l}_\flat=\fn{L}_\flat(m_\sharp,q_\sharp,\fn{k}_\sharp,\fn{s}_\sharp)$, $r_\flat=\fn{r}_\flat(m_\sharp,q_\sharp,\fn{k}_\sharp,\fn{s}_\sharp)$, $k_\sharp=\fn{k}_\sharp(\fn{l}_\flat,r_\flat)$, and $s_\sharp=\fn{s}_\sharp(\fn{l}_\flat,r_\flat)$, if $r_\flat\geq q_\sharp$ then $s_\sharp\geq r_\flat$ and if also $l_\flat\geq k_\sharp$ then for any $l\in[k_\sharp,l_\flat]$ there are sets $\mathcal{B}^-$ and $\mathcal{B}^{1,-}$ with $\mu(\mathcal{B}^-)<4/D_\flat$, $\mu(\mathcal{B}^{1,-})<2/D_\flat$ so that
\[\left|(\rho_{m_\sharp}\lambda_{s_\sharp})(\Omega)-(\rho_l\lambda_{q_\sharp})(\Omega)\right|\leq|\rho_l\lambda_{q_\sharp}|(\mathcal{B}^-)+|\rho_l\lambda_{q_\sharp}|(\mathcal{B}^{1,-})+\frac{20}{E_\flat}.\]
\end{itemize}
\end{lemma}
\begin{proof}
As always, we examine the proof.

Suppose $D_0, m_0, p_0, \fn{r}_0$ are fixed and that $\fn{r}_0$ is $\fn{w}_{j,\fn{u},E_\flat}$-bounded by $C$ above $\max\{m_0,p_0,\ln D_0,p_\flat,n_\flat, \ln D_\flat\}$.  Then:
\begin{itemize}
\item $\fn{s}_{m_\dagger,r}(k,r')\leq\fn{w}^C_{j,\fn{u},E_\flat}(\max\{m_\dagger,r,r',m_0,p_0,\ln D_0,p_\flat,n_\flat, \ln D_\flat\})$,
\item If $\fn{l}$ is $\fn{u}$-bounded by $1$ above $c$, $\fn{l}_{m_\dagger,\fn{l},r}(k,r')\leq\fn{u}(\max\{q_\dagger+1,k,c\})$,
\item $\fn{L}_\dagger(m_\dagger,q_\dagger,\fn{k}_\dagger,\fn{r}_\dagger)(k,r)\leq\fn{u}(\max\{q_\dagger+1,k,m_\dagger,n_\flat,p_\flat\})$,
\item $\fn{S}_\dagger(m_\dagger,q_\dagger,\fn{k}_\dagger,\fn{r}_\dagger)(k,r)=\fn{s}_{m_\dagger,m_\dagger+1}(k,r')\leq\fn{w}^C_{j,\fn{u},E_\flat}(\max\{r',m_\dagger+1,m_0,p_0,p_\flat,n_\flat\})$.
\end{itemize}

In particular, both $\fn{u}$ and $\fn{r}_0$ are bounded by $\fn{w}^C_{j,\fn{u},E_\flat}$.  We apply Lemma \ref{thm:control_interval_D2_quant} to $E_\flat, \max\{D_0,D_\flat\}, \max\{m_0,n_\flat\}, \max\{p_0,p_\flat\}, \fn{L}_\dagger, \fn{s}_\dagger, \fn{w}_{j,\fn{u},E_\flat}$, and therefore obtain the bounds:
\begin{itemize}
\item $c_\dagger=\fn{w}^{2^{11}B^4E^4_\flat}_{j+1,\fn{u},E_\flat}(\max\{m_0,p_0,\ln D_0,n_\flat,p_\flat,\ln D_\flat\})$,
\item $m_\dagger\leq\fn{w}^{32B^2E_\flat^2+1}_{j+1,\fn{u},E_\flat}(c_\dagger)$,
\item $q_\dagger\leq\fn{w}^{32B^2E_\flat^2+1}_{j+1,\fn{u},E_\flat}(c_\dagger)$,
\item $\fn{k}_\dagger$ is $\fn{w}_{j,\fn{u},E_\flat}$-bounded by $\mathfrak{a}(B, E_\flat 2^{c_\dagger+5},BE_\flat)Cx+\mathfrak{i}(B,\max\{D_0,D_\flat\}, E_\flat 2^{c_\sharp+5},E_\flat)C$ above $\max\{m_\dagger, q_\dagger,m_0,p_0,\ln D_0,n_\flat,p_\flat,\ln D_\flat\}$,
\item $\fn{r}_\dagger$ is $\fn{w}_{j,\fn{u},E_\flat}$-bounded by $\mathfrak{a}(B, E_\flat 2^{c_\dagger+5},BE_\flat)Cy+\mathfrak{i}(B,\max\{D_0,D_\flat\}, E_\flat 2^{c_\dagger+5},E_\flat)C$ above $\max\{m_\dagger, q_\dagger,m_0,p_0,\ln D_0,n_\flat,p_\flat,\ln D_\flat\}$.
\end{itemize}
Note that $\fn{l}_\dagger=\fn{L}_\dagger(m_\dagger,q_\dagger,\fn{k}_\dagger,\fn{r}_\dagger)(k,r)=\fn{u}(\max\{q_\dagger+1,k\})$ while $\fn{s}_\dagger=\fn{S}_\dagger(m_\dagger,q_\dagger,\fn{k}_\dagger,\fn{r}_\dagger)$ is $\fn{w}_{j,\fn{u},E_\flat}$-bounded by $C$ above $\max\{m_\dagger+1,m_0,p_0,\ln D_0,n_\flat,p_\flat,\ln D_\flat\}$, so $r_\dagger=\fn{r}_\dagger(\fn{l}_\dagger,\fn{s}_\dagger)$ is bounded by 
\[\fn{w}_{j,\fn{u},E_\flat}^{\mathfrak{a}(B, E_\flat 2^{c_\dagger+5},BE_\flat)C+\mathfrak{i}(B,\max\{D_0,D_\flat\}, E_\flat 2^{c_\dagger+5},E_\flat)C}(\max\{m_\dagger,q_\dagger,m_0,p_0,\ln D_0,n_\flat,p_\flat,\ln D_\flat\})\]
which is in turn bounded by
\[\fn{w}^{2^{12}B^4E^4_\flat}_{j+1,\fn{u},E_\flat}(\max\{m_0,p_0,\ln D_0,n_\flat,p_\flat,\ln D_\flat\}).\]

In particular,
\begin{itemize}
\item $\fn{m}_0(D_0,m_0,p_0,\fn{r}_0)\leq \fn{w}^{2^{12}B^4E^4_\flat}_{j+1,\fn{u},E_\flat}(\max\{m_0,p_0,\ln D_0,n_\flat,p_\flat,\ln D_\flat\})$,
\item $\fn{q}_0(D_0,m_0,p_0,\fn{r}_0)\leq\fn{w}^{2^{12}B^4E^4_\flat}_{j+1,\fn{u},E_\flat}(\max\{m_0,p_0,\ln D_0,n_\flat,p_\flat,\ln D_\flat\})$.
\end{itemize}

This puts us in the setting of Lemma \ref{thm:meta_bnd_3_q} with $d=1$, so we obtain 
\begin{itemize}
  \item $m_0\leq \fn{w}_{2^9B^2E^2_\flat,\fn{u},E_\flat}^{C}(\max\{n_\flat,p_\flat,\ln D_\flat\})$
  \item $q_0\leq  \fn{w}_{2^9B^2E^2_\flat,\fn{u},E_\flat}^{C}(\max\{n_\flat,p_\flat,\ln D_\flat\})$,
  \item $D_0=E_\flat 2^{q_0+5}$,
  \item $\fn{r}_0$ is $\fn{w}^C_{2^9B^2E^2_\flat,\fn{u},E_\flat}$-bounded by $1$ above $\max\{m_0,q_0,n_\flat,p_\flat,\ln D_\flat\}$.
\end{itemize}

Therefore:
\begin{itemize}
\item $m_\sharp\leq\fn{w}_{2^9B^2E^2_\flat+1,\fn{u},E_\flat}^{2^{12}B^4E^4_\flat+1}(\max\{n_\flat,p_\flat,\ln D_\flat\})$,
\item $q_\sharp\leq\fn{w}_{2^9B^2E^2_\flat+1,\fn{u},E_\flat}^{2^{12}B^4E^4_\flat+1}(\max\{n_\flat,p_\flat,\ln D_\flat\})$,
\item $\fn{k}_\sharp$ is $\fn{w}_{2^9B^2E^2_\flat+1,\fn{u},E_\flat}$-bounded by $\mathfrak{a}(B, E_\flat 2^{c_\dagger+5},BE_\flat)Cx+\mathfrak{i}(B,D_0, E_\flat 2^{c_\sharp+5},E_\flat)C$ above $\max\{m_\sharp, q_\sharp,n_\flat,p_\flat,\ln D_\flat\}$,
\item $\fn{s}_\sharp$ is $\fn{w}_{2^9B^2E^2_\flat+1,\fn{u},E_\flat}$-bounded by $\mathfrak{a}(B, E_\flat 2^{c_\dagger+5},BE_\flat)Cy+\mathfrak{i}(B,D_0, E_\flat 2^{c_\sharp+5},E_\flat)C$ above $\max\{m_\sharp, q_\sharp,n_\flat,p_\flat,\ln D_\flat\}$.
\end{itemize}
\end{proof}

At least we obtain actual (large) numeric bounds.

\begin{theorem}[Bounds on \ref{thm:control_interval_E2}]
  Suppose $(*)^Q$ holds.  Then for every $E_\flat$ there exist $m_\sharp<s_\sharp$ and $q_\sharp<l_\sharp$ such that:
  \begin{itemize}
  \item $s_\sharp\leq \fn{w}^{C+2}_{2^{20}B^4 E^4_\flat,\mathrm{suc},E_\flat}(1)$,
  \item $l_\sharp\leq \fn{w}^{C+2}_{2^{20}B^4E^4_\flat,\mathrm{suc},E_\flat}(1)$,
  \item $|(\rho_{m_\sharp}\lambda_{s_\sharp})(\Omega)-(\rho_{l_\sharp}\lambda_{q_\sharp})(\Omega)|<\frac{20}{E_\flat}$.
  \end{itemize}
\end{theorem}
\begin{proof}
  We apply Lemma \ref{thm:control_interval_E2} with $\fn{k}_\flat(m_\sharp,q_\sharp,\fn{l}_\sharp,\fn{s}_\sharp)=q_\sharp+1$ and $\fn{r}_\flat(m_\sharp,q_\sharp,\fn{l}_\sharp,\fn{s}_\sharp)=m_\sharp+1$.

Suppose we are given $D_0, q_0, n_0$, and $\fn{k}_0$ which is $\fn{w}^C_{j,\mathrm{suc},E_\flat}$-bounded by $C$ above $\max\{n_0,q_0,\ln D_0\}$.  Then
\begin{itemize}
\item $\fn{l}_{\star,q,k}(k',r)\leq\fn{w}^C_{j,\mathrm{suc},E_\flat}(\max\{q,k,k'\})$,
\item $\fn{r}_\dagger(m_\dagger,q_\dagger,\fn{s}_\dagger,\fn{k}_\dagger)=m_\dagger+1$,
\item $\fn{L}_\dagger(m_\dagger,q_\dagger,\fn{s}_\dagger,\fn{k}_\dagger)(k',r)=\fn{l}_{\star,q_\dagger,q_\dagger+1}\leq \fn{w}^C_{j,\mathrm{suc},E_\flat}(\max\{q_\dagger+1,k'\})$.
\end{itemize}

Therefore by the preceding lemma we obtain $m_\dagger,q_\dagger,\fn{s}_\dagger,\fn{k}_\dagger$ with:
\begin{itemize}
\item $m_\dagger\leq\fn{w}^{2^{12}B^4E^4_\flat+1}_{2^9B^2E^2_\flat+j+2,\mathrm{suc},E_\flat}(\max\{n_0,q_0, \ln D_0\})$,
\item $q_\dagger\leq\fn{w}^{2^{12}B^4E^4_\flat+1}_{2^9B^2E^2_\flat+j+2,\mathrm{suc},E_\flat}(\max\{n_0,q_0,\ln D_0\})$,
\item $\fn{k}_\dagger$ is $\fn{w}_{2^9B^2E^2_\flat+j+2,\mathrm{suc},E_\flat}$-bounded by $\mathfrak{a}(B,E_\flat 2^{q_\dagger+5},BE_\flat)Cx+\mathfrak{i}(B,D_0,E_\flat 2^{q_\dagger+5},E_\flat)C$ above $\max\{m_\dagger,q_\dagger,n_0,q_0,\ln D_0\}$,
\item $\fn{s}_\dagger$ is $\fn{w}_{2^9B^2E^2_\flat+j+2,\mathrm{suc},E_\flat}$-bounded by $\mathfrak{a}(B,E_\flat 2^{q_\dagger+5},BE_\flat)Cy+\mathfrak{i}(B,D_0,E_\flat 2^{q_\dagger+5},E_\flat)C$ above $\max\{m_\dagger,q_\dagger,n_0,q_0,\ln D_0\}$.
\end{itemize}

In particular, since $\fn{l}_\dagger=\fn{L}_\dagger(m_\dagger,q_\dagger,\fn{s}_\dagger,\fn{k}_\dagger)$ is bounded by $\fn{w}^C_{j,\mathrm{suc},E_\flat}(\max\{q_\dagger+1,k'\})$, $k_\dagger=\fn{k}_\dagger(r_\dagger,\fn{l}_\dagger)$ is bounded by 
\[\fn{w}_{2^9B^2E^2_\flat+j+2,\mathrm{suc},E_\flat}^{\mathfrak{a}(B,E_\flat 2^{q_\dagger+5},BE_\flat)C^2+\mathfrak{i}(B,D_0,E_\flat 2^{q_\dagger+5},E_\flat)C}(\max\{m_\dagger+1,q_\dagger+1,n_0,q_0,\ln D_0\})\leq\fn{w}_{2^9B^2E^2_\flat+j+3,\mathrm{suc},E_\flat}(\max\{m_\dagger+1,q_\dagger+1,n_0,q_0,\ln D_0\}).\]

This puts us in the setting of Lemma \ref{thm:meta_bnd_3_q} with $d=2^9B^2E^2_\flat+5$, so we ultimately obtain:
\begin{itemize}
\item $q_0\leq\fn{w}^C_{2^{19}B^4E^4_\flat,\mathrm{suc},E_\flat}(1)$,
\item $n_0\leq\fn{w}^C_{2^{19}B^4E^2_\flat,\mathrm{suc},E_\flat}(1)$,
\item $D_0=E_\flat 2^{q_0+5}$,
\item $\fn{k}_0$ is $\fn{w}^C_{2^{19}B^4E^2_\flat,\mathrm{suc},E_\flat}$-bounded by $1$ above $\max\{q_0,n_0\}$.
\end{itemize}

In particular, 
\begin{itemize}
\item $m_\sharp=m_\dagger\leq\fn{w}^{C+2^{12}B^4E^4_\flat+1}_{2^{19}B^4E^2_\flat,\mathrm{suc},E_\flat}(1)$,
\item $q_\sharp=q_\dagger\leq\fn{w}^{C+2^{12}B^4E^4_\flat+1}_{2^{19}B^4E^2_\flat,\mathrm{suc},E_\flat}(1)$,
\item $k_\dagger\leq\fn{w}^C_{2^{20}B^4E^4_\flat,\mathrm{suc},E_\flat}(\max\{m_\sharp,q_\sharp\})$,
\item $r_\dagger=m_\sharp+1$,
\item $l_\sharp\leq\fn{w}^{C+1}_{2^{20}B^4E^4_\flat,\mathrm{suc},E_\flat}(\max\{m_\sharp,q_\sharp\})$,
\item $s_\sharp\leq\fn{w}^{C+1}_{2^{20}B^4E^4_\flat,\mathrm{suc},E_\flat}(\max\{m_\sharp,q_\sharp\})$.
\end{itemize}

\end{proof}

In particular, bounds on $s_\sharp$ and $l_\sharp$ as a function of $B\cdot E_\flat$ are given by
\[B\cdot E_\flat\mapsto \fn{w}^{C+2}_{2^{20}(BE_\flat)^4,\mathrm{suc},E_\flat}(1)\leq f_{2^{22}(BE_\flat)^4}(B+E+c).\]
Recall the function $f_\omega(m)=f_m(m)$; then $s_\sharp$ and $l_\sharp$ are bounded by
\[f_\omega(2^{22}(BE_\flat)^4+c)\]
for some constant $c$.

\bibliographystyle{plain}
\bibliography{../../Bibliographies/main}

\begin{thebibliography}{10}

\bibitem{avigad:MR1640329}
Jeremy Avigad and Solomon Feferman.
\newblock G\"odel's functional (``{D}ialectica'') interpretation.
\newblock In {\em Handbook of proof theory}, volume 137 of {\em Stud. Logic
  Found. Math.}, pages 337--405. North-Holland, Amsterdam, 1998.

\bibitem{avigad:MR2550151}
Jeremy Avigad, Philipp Gerhardy, and Henry Towsner.
\newblock Local stability of ergodic averages.
\newblock {\em Trans. Amer. Math. Soc.}, 362(1):261--288, 2010.

\bibitem{croot_Sz_notes}
Ernie Croot.
\newblock Notes on szemeredi's regularity lemma.
\newblock \url{http://people.math.gatech.edu/~ecroot/regularity.pdf}.

\bibitem{MR737004}
Joseph Diestel.
\newblock {\em Sequences and series in {B}anach spaces}, volume~92 of {\em
  Graduate Texts in Mathematics}.
\newblock Springer-Verlag, New York, 1984.

\bibitem{gaspar}
Jaime Gaspar and Ulrich Kohlenbach.
\newblock On {T}ao's ``finitary'' infinite pigeonhole principle.
\newblock {\em J. Symbolic Logic}, 75(1):355--371, 2010.

\bibitem{MR2373327}
Philipp Gerhardy and Ulrich Kohlenbach.
\newblock General logical metatheorems for functional analysis.
\newblock {\em Trans. Amer. Math. Soc.}, 360(5):2615--2660, 2008.

\bibitem{goldbring:_approx_logic_measure}
Isaac Goldbring and Henry Towsner.
\newblock An approximate logic for measures.
\newblock {\em Israel Journal of Mathematics}, 199(2):867--913, 2014.

\bibitem{kohlenbach:MR2445721}
U.~Kohlenbach.
\newblock {\em Applied proof theory: proof interpretations and their use in
  mathematics}.
\newblock Springer Monographs in Mathematics. Springer-Verlag, Berlin, 2008.

\bibitem{MR2943215}
U.~Kohlenbach and L.~Leu{\c{s}}tean.
\newblock On the computational content of convergence proofs via {B}anach
  limits.
\newblock {\em Philos. Trans. R. Soc. Lond. Ser. A Math. Phys. Eng. Sci.},
  370(1971):3449--3463, 2012.

\bibitem{MR2054493}
U.~Kohlenbach and P.~Oliva.
\newblock Proof mining: a systematic way of analyzing proofs in mathematics.
\newblock {\em Tr. Mat. Inst. Steklova}, 242(Mat. Logika i Algebra):147--175,
  2003.

\bibitem{MR1195271}
Ulrich Kohlenbach.
\newblock Effective bounds from ineffective proofs in analysis: an application
  of functional interpretation and majorization.
\newblock {\em J. Symbolic Logic}, 57(4):1239--1273, 1992.

\bibitem{MR1428007}
Ulrich Kohlenbach.
\newblock Analysing proofs in analysis.
\newblock In {\em Logic: from foundations to applications ({S}taffordshire,
  1993)}, Oxford Sci. Publ., pages 225--260. Oxford Univ. Press, New York,
  1996.

\bibitem{MR3278188}
Ulrich Kohlenbach and Angeliki Koutsoukou-Argyraki.
\newblock Rates of convergence and metastability for abstract {C}auchy problems
  generated by accretive operators.
\newblock {\em J. Math. Anal. Appl.}, 423(2):1089--1112, 2015.

\bibitem{MR3175627}
Daniel K{\"o}rnlein and Ulrich Kohlenbach.
\newblock Rate of metastability for {B}ruck's iteration of pseudocontractive
  mappings in {H}ilbert space.
\newblock {\em Numer. Funct. Anal. Optim.}, 35(1):20--31, 2014.

\bibitem{MR3210080}
Lauren{\c{t}}iu Leu{\c{s}}tean.
\newblock An application of proof mining to nonlinear iterations.
\newblock {\em Ann. Pure Appl. Logic}, 165(9):1484--1500, 2014.

\bibitem{tao08Norm}
Terence Tao.
\newblock Norm convergence of multiple ergodic averages for commuting
  transformations.
\newblock {\em Ergodic Theory Dynam. Systems}, 28(2):657--688, 2008.

\bibitem{tao:MR2408398}
Terence Tao.
\newblock Norm convergence of multiple ergodic averages for commuting
  transformations.
\newblock {\em Ergodic Theory Dynam. Systems}, 28(2):657--688, 2008.

\bibitem{towsner:banach}
Henry Towsner.
\newblock An inverse {A}ckermannian lower bound on the local unconditionality
  constant of the {J}ames space.
\newblock draft.

\bibitem{towsner_worked}
Henry Towsner.
\newblock A worked example of the functional interpretation.
\newblock draft.

\end{thebibliography}
\end{document}